\documentclass[letterpaper,11pt,oneside,reqno]{amsart}

\usepackage{amsmath,amssymb,amsthm,amsfonts}
\usepackage[colorlinks=true,linkcolor=blue,citecolor=red]{hyperref}
\usepackage{graphicx,color}
\usepackage{upgreek}
\usepackage[mathscr]{euscript}

\allowdisplaybreaks
\numberwithin{equation}{section}

\usepackage{tikz}
\usetikzlibrary{shapes,arrows,positioning,decorations.markings}

\usepackage{array}
\usepackage{adjustbox}
\usepackage{cleveref}
\usepackage{enumerate}
\usepackage{caption}
\usepackage[DIV=12]{typearea}

\usepackage[shadow=true]{todonotes}

\newcommand{\qm}{\mathfrak{q}}
\newcommand{\tm}{\mathfrak{t}}

\newtheorem{proposition}{Proposition}[section]
\newtheorem{lemma}[proposition]{Lemma}
\newtheorem{corollary}[proposition]{Corollary}
\newtheorem{theorem}[proposition]{Theorem}
\newtheorem{conjecture}[proposition]{Conjecture}
\theoremstyle{definition}
\newtheorem{definition}[proposition]{Definition}
\newtheorem{properties}[proposition]{Properties}
\newtheorem{remark}[proposition]{Remark}

\begin{document}

\title[Refined Cauchy identity for spin Hall--Littlewood functions]
{Refined Cauchy identity for spin Hall--Littlewood symmetric rational functions}
\author{Leonid Petrov}

\date{}

\maketitle

\begin{abstract}
	Fully inhomogeneous spin Hall--Littlewood symmetric rational
	functions $\mathsf{F}_\lambda$ arise in the context of $\mathfrak{sl}(2)$ higher spin six vertex models,
	and are multiparameter deformations of the classical Hall--Littlewood symmetric polynomials.
	We obtain a refined Cauchy identity
	expressing a 
	weighted sum of the product of two $\mathsf{F}_\lambda$'s
	as a determinant. The determinant is of Izergin--Korepin type: it is
	the partition function of the six vertex model
	with suitably decorated domain wall boundary conditions.
	The proof of equality of two partition functions
	is based on the Yang--Baxter equation.

	We rewrite our Izergin--Korepin type determinant
	in a different form which includes one of the sets of variables in a completely symmetric way.
	This determinantal identity might be of independent interest,
	and also allows to directly link the spin Hall--Littlewood rational functions
	with (the Hall--Littlewood particular case of) the interpolation Macdonald polynomials.
	In a different direction, a Schur expansion of our Izergin--Korepin type determinant
	yields a deformation of Schur symmetric polynomials.
	
	In the spin-$\frac12$ specialization, our refined Cauchy identity leads to a summation identity for 
	eigenfunctions of the ASEP (Asymmetric Simple Exclusion Process), a celebrated stochastic
	interacting particle system
	in the Kardar--Parisi--Zhang universality class.
	This produces explicit integral formulas for certain multitime probabilities in ASEP.
\end{abstract}

\tableofcontents

\section{Introduction}
\label{sec:intro}

\subsection{Background}

This paper deals with summation identities for spin Hall--Littlewood symmetric
rational functions. These functions 
appeared in \cite{Borodin2014vertex}, \cite{BorodinPetrov2016inhom}
as 
partition functions (with boundary conditions of a rather general form)
of square lattice
vertex models possessing Yang--Baxter integrability
which is traced to the quantum group $U_q(\widehat{\mathfrak{sl}_2})$.

The spin Hall--Littlewood functions
may also be identified with Bethe Ansatz eigenfunctions 
of the higher spin six vertex model on $\mathbb{Z}$, cf. \cite[Ch. VII]{QISM_book}.
They also arise as eigenfunctions of certain stochastic particle systems
\cite{Povolotsky2013}, \cite{BCPS2014_arXiv_v4}, \cite{CorwinPetrov2015arXiv}.
In \cite{Borodin2014vertex}, \cite{BorodinPetrov2016inhom} and subsequent works
the spin Hall--Littlewood functions and their relatives are treated from the point of
view of the theory of symmetric functions.
A classical reference on symmetric functions in the book
\cite{Macdonald1995} where Schur, Hall--Littlewood, and Macdonald symmetric
polynomials
and symmetric functions (=~property understood symmetric polynomials
in infinitely many variables) are developed.

\medskip

One of the most important features 
common for many
families of symmetric polynomials
is a \emph{Cauchy type summation identity}. For example, the Hall--Littlewood polynomials
(\cite[Ch. III]{Macdonald1995}; we also briefly recall them in \Cref{sec:interp_HL} below)
satisfy the following Cauchy type identity:
\begin{equation}
	\label{eq:HL_Cauchy}
	\sum_{\lambda_1\ge \lambda_2\ge \ldots\ge \lambda_N\ge0 }
	P_\lambda^{HL}(u_1,\ldots,u_N;\tm )
	\,
	Q_\lambda^{HL}(v_1,\ldots,v_N;\tm )
	=
	\prod_{i,j=1}^{N}\frac{1-\tm u_iv_j}{1-u_iv_j},
\end{equation}
where the sum is over ordered $N$-tuples of integers $\lambda$.
In \cite{Borodin2014vertex}, 
\cite{BorodinPetrov2016inhom}
similar Cauchy type summation identities were proven 
for the spin Hall--Littlewood symmetric functions. 
The proofs are based on the Yang--Baxter equation for the higher spin six vertex model.
Namely, both sides are identified as certain 
partition functions\footnote{By a partition function we mean 
the sum of total weights of all configurations, where the total weight of a configuration is a 
product of local vertex weights viewed as ``Boltzmann weights''.}
which are equal thanks to a sequence of elementary Yang--Baxter steps.
The product in the right-hand side comes from the fact that the corresponding partition
function has only one nontrivial configuration.

\medskip

We are interested in \emph{refinements} of Cauchy type identities like \eqref{eq:HL_Cauchy}
which are obtained by inserting a new factor (depending on $\lambda$)
into each term in the sum over $\lambda$ in the left-hand side.
This complicates
the right-hand side: instead of a relatively simple product it becomes a determinant.

Refinements of Cauchy type identities 
for various symmetric functions
appeared since \cite{kirillov1999q},
\cite{warnaar2008bisymmetric}, see also
\cite{betea2016refined}, \cite{betea2015refined}.
In \cite{wheeler2015refined}, a
novel method for proving a number of refined Cauchy (and also Littlewood) type identities 
was introduced based on the Yang--Baxter equation. 
Namely, 
the determinant in the right-hand side of a refined identity
arises as 
the partition function of the six-vertex model with domain wall boundary 
conditions (or a suitable modification thereof). The domain wall six vertex partition function is
given by
the celebrated
Izergin--Korepin determinant \cite{Izergin1987},
\cite[Ch. VII.10]{QISM_book}. 
The particular
determinantal answer for each refined identity
is uniquely identified with the help of Lagrange interpolation.

\begin{remark}
	After completing this manuscript, 
	another refined summation identity for the spin 
	Hall--Littlewood functions --- a Littlewood-type identity --- 
	was proven in \cite{gavrilova2021refined}.
	See also 
	\cite{chen2021stable}
	for probabilistic applications of Littlewood-type identities
	for the spin Hall--Littlewood functions.
\end{remark}

\subsection{Refined Cauchy identity for spin Hall--Littlewood functions}

Our first main result is a lifting of the refined Cauchy identity for Hall--Littlewood polynomials
to the spin Hall--Littlewood level. Namely, we consider the fully inhomogeneous 
spin Hall--Littlewood 
symmetric rational functions
introduced in \cite{BorodinPetrov2016inhom} which have the following explicit
symmetrization form:
\begin{equation*}
	\begin{split}
		\mathsf{F}_\lambda(u_1,\ldots,u_N )
		&=
		\sum_{\sigma\in S_N}
		\sigma
		\Biggl( 
			\prod_{1\le i<j\le N}\frac{u_i-q u_j}{u_i-u_j}\,
			\prod_{i=1}^{N}
			\biggl(
				\frac{1-q}{1-s_{\lambda_i}\xi_{\lambda_i}u_i}
				\prod_{j=0}^{\lambda_i-1}\frac{\xi_j u_i-s_j}{1-\xi_j s_j u_i}
			\biggr)
		\Biggr),
	\end{split}
\end{equation*}
where 
$\lambda=(\lambda_1\ge\ldots\ge \lambda_N\ge0 )$.
Here $\sigma\in S_N$
acts by permuting the $u_i$'s and not the $\lambda_i$'s.
We employ the same convention about permutation action
in all symmetrization formulas
throughout the text.
The function $\mathsf{F}_\lambda$ depends on the ``quantum parameter''
$q$, the variables $u_j$, and the inhomogeneities
$\xi_x$ and $s_x$, where $x\in \mathbb{Z}_{\ge0}$.
These parameters are assumed to be 
generic complex numbers. When required, for certain statements
(like in \Cref{thm:intro_sHL_refined} below)
we impose additional conditions on the parameters.
If $s_x=0$ and $\xi_x=1$ for all $x$, then the 
functions $\mathsf{F}_\lambda$ reduce to the usual
Hall--Littlewood symmetric polynomials.

To formulate the result we need more notation. Let
$m_0(\lambda)$ is the number of parts in $\lambda$ equal to zero (so that $\lambda_{N-m_0(\lambda)}>0$
and $\lambda_{N-m_0(\lambda)+1}=0$), and 
$(a;q)_k=(1-a)(1-aq)\ldots(1-aq^{k-1})$ be the $q$-Pochhammer symbol.
We also employ
dual spin Hall--Littlewood
functions $\mathsf{F}_\lambda^*$ which 
are given by a symmetrization expression similarly to $\mathsf{F}_\lambda$,
see formulas \eqref{eq:F_Fstar_conjugation}, \eqref{eq:F_star_explicit} in the text.

\begin{theorem}[Refined Cauchy identity for spin Hall--Littlewood functions]
	\label{thm:intro_sHL_refined}
	For any $\gamma\ne0$ and provided that $u_i,v_j,s_x,\xi_x$
	satisfy certain conditions so that the series in the left-hand side converges 
	(see \eqref{eq:admissible_u_v} in the text), we have
	\begin{align}
		\nonumber
			&
			\sum_{\lambda=(\lambda_1\ge \lambda_2\ge \ldots\ge\lambda_N\ge0 )}
			\frac{(\gamma q;q)_{m_0(\lambda)}(\gamma^{-1}s_0^2;q)_{m_0(\lambda)}}{(q;q)_{m_0(\lambda)}(s_0^2;q)_{m_0(\lambda)}}
			\,
			\mathsf{F}_{\lambda}(u_1,\ldots,u_N )\,
			\mathsf{F}^*_{\lambda}(v_1,\ldots,v_N )
			\\&\hspace{30pt}=
			\prod_{j=1}^{N}
			\frac{1}{(1-s_0\xi_0 u_j)(1-\xi_0^{-1}s_0v_j)}\,
			\frac{\prod_{i,j=1}^{N}(1-qu_iv_j)}{\prod_{1\le i<j\le N}(u_i-u_j)(v_i-v_j)}
		\label{eq:intro_sHL_refined}
			\\&\hspace{56pt}\times
			\det\left[
				\frac{(1-\gamma)(q-\gamma^{-1} s_0^2)(1-u_iv_j)+(1-q)(1-\xi_0s_0 u_i)(1-\xi_0^{-1}s_0v_j)}
					{(1-u_iv_j)(1-qu_iv_j)}
				\right]_{i,j=1}^{N}.
				\nonumber
	\end{align}
\end{theorem}

Observe that 
a priori 
the left-hand side of \eqref{eq:intro_sHL_refined}
should depend on the parameters
$s_x,\xi_x$ for all $x\in \mathbb{Z}_{\ge0}$.
However, it turns out that the dependence
on all of them except $(s_0,\xi_0)$ is artificial and 
disappears  after summing over $\lambda$.

We prove \Cref{thm:intro_sHL_refined} 
in \Cref{sec:sHL_proof} with the help of the Yang--Baxter equation.
The Izergin--Korepin type determinant in the right-hand side of 
\eqref{eq:intro_sHL_refined}
comes from 
the partition function of the six vertex model
with suitably decorated domain wall boundary conditions,
and is uniquely identified 
with the help of Lagrange interpolation.
That is, we formulate a number of properties which follows from the 
description of the six vertex partition function, and check that the 
determinant satisfies these properties (uniqueness is guaranteed by Lagrange interpolation).
This approach to proving the refined identity is essentially parallel to \cite{wheeler2015refined}. 
However, one of the steps in identifying 
the Izergin--Korepin type determinant
turns out to be more involved and requires to compute a nontrivial auxiliary determinant
which evaluates in a product form
(see \Cref{lemma:u0_det_computation} in the text).

\subsection{Determinantal identity}

Our next result is an alternative determinantal expression for the right-hand
side of the refined Cauchy identity \eqref{eq:intro_sHL_refined}. 
Namely, we observe that the (suitably normalized)
determinant in the right-hand side of \eqref{eq:intro_sHL_refined} is a skew-symmetric polynomial 
in $v_1,\ldots,v_N $. As such, it should be divisible by the product
$\prod_{i<j}(v_i-v_j)$ in the denominator. Remarkably, the
ratio admits a nice determinantal form which includes all the variables $v_1,\ldots,v_N $
in a manifestly symmetric way:

\begin{theorem}
	\label{thm:intro_det_identity}
	The following identity of $N\times N$ determinants holds:
	\begin{equation}
		\label{eq:intro_det_identity}
		\begin{split}
			&
			\frac{\prod_{i,j=1}^{N}(1-qu_iv_j)}{\prod_{1\le i<j\le N}(v_i-v_j)}
			\,\det\left[
				\frac{(1-\gamma)(q-\gamma^{-1} s_0^2)(1-u_iv_j)+(1-q)(1-\xi_0s_0 u_i)(1-\xi_0^{-1}s_0v_j)}
					{(1-u_iv_j)(1-qu_iv_j)}
			\right]_{i,j=1}^{N}
			\\&\hspace{20pt}=
			\det
			\Biggl[ 
			u_j^{N-i-1}
				\biggl\{
					\left( 1-s_0\xi_0 u_j \right)
					\left( u_j-\frac{s_0}{\xi_0} \right)
					\prod_{l=1}^{N}\frac{1-q u_jv_l}{1-u_jv_l}
					\\
					&\hspace{200pt}
					-
					\gamma^{-1}q^{N-i}
					(\gamma-s_0\xi_0 u_j)
					\left( \gamma q u_j-\frac{s_0}{\xi_0} \right)
				\biggr\}
			\Biggr]_{i,j=1}^{N}.
		\end{split}
	\end{equation}
\end{theorem}
We prove
\Cref{thm:intro_det_identity}
in \Cref{sec:det_identity_proof}
by checking that the (suitably normalized)
right-hand side of \eqref{eq:intro_det_identity}
satisfies properties which identify the answer by Lagrange interpolation.
This involves computing 
another nontrivial auxiliary determinant.
This determinant is
related to the tridiagonal matrix 
of three-term relation coefficients for
the $q$-Krawtchouk orthogonal polynomials
\cite[Section 3.15]{Koekoek1996}. Using this connection,
we are able to evaluate the determinant in a product form
(see \Cref{lemma:M_det_4} in the text).

\subsection{Relation to interpolation Hall--Littlewood polynomials}

The determinantal identity of \Cref{thm:intro_det_identity}
generalizes the one from
\cite[Section 4.5]{cuenca2018interpolation};
the latter is recovered by setting $s_0=0$ in 
\eqref{eq:intro_det_identity}.

The paper \cite{cuenca2018interpolation}
in particular deals with the Hall--Littlewood particular
case $F_\lambda^{HL}$ of interpolation Macdonald symmetric polynomials 
\cite{knop1996symmetric}, \cite{okounkov1997binomial},
\cite{sahi1996interpolation},
and derives a refined Cauchy identity for them.
In that identity, the determinant in the right-hand side is like the one in the right-hand
side of \eqref{eq:intro_det_identity}.
Combining this with the Izergin--Korepin type determinant from \cite{wheeler2015refined}
leads to a determinantal identity given in
\cite[Section 4.5]{cuenca2018interpolation}.

This connection between determinants
suggests a direct limit transition from the 
fully inhomogeneous spin Hall--Littlewood symmetric rational functions
$\mathsf{F}_\lambda$
to the inhomogeneous Hall--Littlewood polynomials $F_\lambda^{HL}$
(\Cref{prop:sHL_to_iHL} in the text).
This resolves a question asked in 
\cite{olshanski2019interpolation} and \cite{cuenca2018interpolation}
about connections between interpolation Hall--Littlewood polynomials and 
spin Hall--Littlewood functions. We also obtain an explicit symmetrization
formula for the dual interpolation Hall--Littlewood functions $G_\lambda^{HL}$
(\Cref{prop:dual_iHL_symmetrization} in the text).
The functions $G_\lambda^{HL}$ together with 
$F_\lambda^{HL}$
satisfy a Cauchy identity with a product form right-hand side.

\subsection{Schur expansion}

The right-hand side of the refined Cauchy identity 
\eqref{eq:intro_sHL_refined} is symmetric in $v_1,\ldots,v_N $.
Looking into its
Taylor series expansion
into the Schur symmetric polynomials
$s_\lambda(v_1,\ldots,v_N )$ leads to a sum of the form
$\sum_\lambda \mathsf{C}_\lambda(u_1,\ldots,u_N )\,s_\lambda(v_1,\ldots,v_N )$,
where $\mathsf{C}_\lambda$ are new symmetric polynomials
which are deformations of the Schur polynomials. 
Using both determinants in \eqref{eq:intro_det_identity} and orthogonality of Schur polynomials,
we obtain two formulas for them:
\begin{equation}
	\label{eq:generalized_JT_identity}
	\begin{split}
		\mathsf{C}_\lambda(u_1,\ldots,u_N )&=
		\frac{\det[\mathsf{c}_{\lambda_i+N-i}(u_j)]_{i,j=1}^{N}}{\prod_{1\le i<j\le N}(u_i-u_j)}
		\\&=
		\det\Bigl[
			A^-_{\lambda_j}h_{\lambda_j+i-j-1}
			+
			A^0_{\lambda_j}
			h_{\lambda_j+i-j}
			+
			A^+_{\lambda_j}
			h_{\lambda_j+i-j+1}
		\Bigr]_{i,j=1}^N,
	\end{split}
\end{equation}
where $\mathsf{c}_k(u)$ is a certain three-term linear combination of $u^{k-1},u^k,u^{k+1}$ (with coefficients
depending on $k$),
and $A_k^{\pm}, A_k^0$ are also explicit 
(see formulas \eqref{eq:small_c_function} and \eqref{eq:C_lambda_H}
in the text).
Here $h_k=h_k(u_1,\ldots,u_N )$ are the complete homogeneous symmetric polynomials.
Identity \eqref{eq:generalized_JT_identity}
might be called a generalized Jacobi--Trudi formula, 
cf. \cite{sergeev2014jacobi}, \cite{harnad2018symmetric}.

In the particular case $s_0=0$
we have the proportionality
$\mathsf{C}_\lambda=
s_\lambda
\prod_{j=1}^{N}(1-(\gamma q)q^{\lambda_j+N-j})$. This connects our
refined Cauchy identity of \Cref{thm:intro_sHL_refined} for $s_0=0$
to a refined Cauchy identity for Schur polynomials.
The latter is known to generalize to the full Macdonald level
\cite{warnaar2008bisymmetric},
and is related to Macdonald probability measures on partitions \cite{borodin2016stochastic_MM}.
We discuss this probabilistic connection in \Cref{sub:measures_on_partitions}.

\subsection{Application to ASEP eigenfunctions}

The original motivation for this work was to explain
determinants arising from summing eigenfunctions of the ASEP
(Asymmetric Simple Exclusion Process)
recently observed by Corwin and Liu 
\cite{LiuCorwinPrivate2019}. 
In \Cref{sec:6v_reduction}
we show how our refined Cauchy identity of \Cref{thm:intro_sHL_refined}
reduces to a summation identity for ASEP eigenfunctions. 
In particular, this leads to contour
integral formulas for certain multitime distributions in ASEP.
We illustrate this by writing down the
two-time 
distribution of the first particle in ASEP
$\mathrm{Prob}\bigl(x_1(t_1)\ge k_1,\ x_1(t_2)\ge k_2\bigr)$
(\Cref{thm:ASEP_two_time} in the text).
This formula involves the
Izergin--Korepin
type determinant under the integral, and so it is not clear
at this point whether it is amenable to asymptotic analysis.
The conjectural asymptotic behavior (at least for fixed $q$ not going to $1$)
of these formulas should essentially match the behavior in the simpler case of TASEP
(i.e., ASEP with particle jumps in only one direction).
For 
TASEP,
multitime formulas and their asymptotics were recently studied in
\cite{JohanssonRahman2019},
\cite{liu2019multi}.

\subsection{Notation}
\label{sub:notation}

The main quantization parameter is denoted by $q$ throughout the paper
and is assumed to belong to $(0,1)$.
When dealing with the usual Macdonald and 
Hall--Littlewood symmetric polynomials,
the deformation parameters are denoted by $(\qm,\tm)$ and $\tm$, respectively.
We always 
make explicit comments
about renaming our main parameter $q$ to the Hall--Littlewood parameter $\tm$.

The identities we obtain in the paper are valid
for generic complex values of parameters entering them
(sometimes under restrictions imposed so that certain series
converge). Vanishing of denominators may make some of the formulas meaningless,
but here we do not discuss necessary modifications which might restore some
formulas under such degenerations.

The indicator of an event or condition $A$ is denoted by $\mathbf{1}_A$.
We use the $q$-Pochhammer symbol notation
\begin{equation}
	\label{eq:q_Pochhammer}
	(a;q)_k  = (1-a)(1-aq)\ldots(1-aq^{k-1}),
	\qquad 
	(a;q)_0 = 1.
\end{equation}
Since $q\in(0,1)$, the infinite $q$-Pochhammer symbol $(a;q)_\infty=\prod_{j=1}^{\infty}(1-aq^{j-1})$
also makes sense. 

By $E\vert_{a\to b}$ we denote the substitution of $a$ into $b$
everywhere
in an expression $E$.

A nonnegative \emph{signature}
of \emph{length} $N$ is an weakly decreasing sequence of integers
\begin{equation*}
	\lambda=(\lambda_1\ge \lambda_2\ge \ldots\ge \lambda_N\ge0 ).
\end{equation*}
Denote the set of nonnegative signatures of length $N$ by $\mathrm{Sign}_N$.
Throughout most of the paper we need only nonnegative signatures and so omit the word ``nonnegative''.
We will use the notation $|\lambda| = \lambda_1+\ldots+\lambda_N $.

We often need the \emph{multiplicative} notation for signatures, 
$\lambda=(0^{m_0(\lambda)}1^{m_1(\lambda)}2^{m_2(\lambda)}\ldots )$,
where $m_j(\lambda)$ is the number of parts in $\lambda$ which are equal to $j$.
Note that $\sum_{j\ge0}m_j(\lambda)=N$.
By $\ell(\lambda)$ we denote the number of nonzero parts in $\lambda$, so that
$m_0(\lambda)=N-\ell(\lambda)$.

\subsection{Organization of the paper}

In \Cref{sec:vertex_and_sHL} we recall the basic notation,
definition, and some properties of the spin Hall--Littlewood rational symmetric functions.
In \Cref{sec:sHL_proof} we prove the refined Cauchy identity 
for the spin Hall--Littlewood functions
(\Cref{thm:intro_sHL_refined}).
In \Cref{sec:det_identity_proof}
we derive an alternative determinantal expression for the right-hand side of the
refined Cauchy identity.
In \Cref{sec:interp_HL} we discuss connections of our spin Hall--Littlewood functions
with the Hall--Littlewood degeneration of the Macdonald interpolation symmetric polynomials. 
In \Cref{sec:expansions_IK} we derive a Schur expansion of the determinant in the right-hand
side of the Cauchy identity, and connect these results to measures on partitions.
Finally, 
in \Cref{sec:6v_reduction}
we specialize the spin Hall--Littlewood functions to
eigenfunctions of the ASEP (Asymmetric Simple Exclusion Process), which
leads to determinantal summation formulas discovered earlier by Corwin and Liu
\cite{LiuCorwinPrivate2019}.

\subsection{Acknowledgments}

I~am very grateful to Zhipeng Liu who pointed to me 
the emergence of determinants in sums of eigenfunctions of ASEP
observed in the ongoing work \cite{LiuCorwinPrivate2019}. Understanding this phenomenon in the full
generality of spin Hall--Littlewood functions was the initial impulse for this paper.
I~am also grateful to 
Alexei Borodin, Filippo Colomo, Cesar Cuenca, and Grigori Olshanski for helpful conversations.
The work was
partially supported by the NSF grant DMS-1664617.

\section{Vertex weights and spin Hall--Littlewood functions}
\label{sec:vertex_and_sHL}

In this section we recall our main objects --- the higher spin six vertex model
weights, the Yang--Baxter equation for them, and
spin Hall--Littlewood symmetric functions.
This section closely follows the earlier works
\cite{BorodinPetrov2016inhom}, \cite{BufetovMucciconiPetrov2018}.

\subsection{Definition of vertex weights}

Here we recall the higher spin six vertex model weights $w_{u,s}$
from \cite{Borodin2014vertex}, \cite{BorodinPetrov2016inhom}.
They depend on $q$, the spectral parameter $u$, and
the spin parameter $s$.
They are given in \Cref{fig:w_weights}.
These weights $w_{u,s}(i_1,j_1;i_2,j_2)$ satisfy
arrow preservation: they vanish unless $i_1+j_1=i_2+j_2$.
Due to this arrow preservation, we will think that 
occupied edges form \emph{up-right paths} which can meet at a vertex.
There is at most one path allowed per each horizontal edge (spin-$\tfrac12$ situation),
while at each vertical edge an arbitrary number of paths is allowed.
\begin{figure}[htpb]
	\centering
	\includegraphics{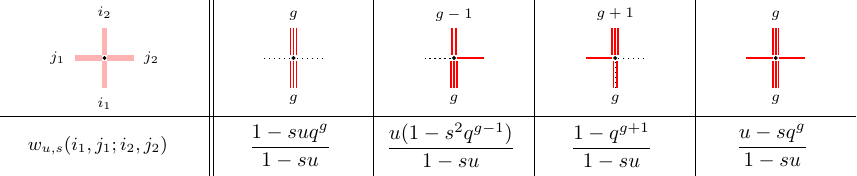}
	\caption{Vertex weights $w_{u,s}$. Here $i_1,i_2\in \mathbb{Z}_{\ge0}$
	and $j_1,j_2\in \left\{ 0,1 \right\}$.}
	\label{fig:w_weights}
\end{figure}

We will also use dual weights $w_{v,s}^*$ 
defined as
\begin{equation}
	\label{eq:w_wstar_relation}
	w_{v,s}^*(i_1,j_1;i_2,j_2)=\frac{(s^2;q)_{i_1}(q;q)_{i_2}}{(s^2;q)_{i_2}(q;q)_{i_1}}\,w_{v,s}(i_2,j_1;i_1,j_2).
\end{equation}
They are given in \Cref{fig:w_star_weights}.
The arrow preservation here reads $i_2+j_1=i_1+j_2$, and so we think of 
occupied edges as forming \emph{up-left} paths.

\begin{figure}[htpb]
	\centering
	\includegraphics{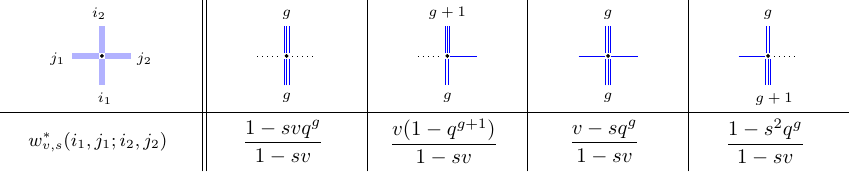}
	\caption{Dual vertex weights $w^*_{v,s}$.
	Here $i_1,i_2\in \mathbb{Z}_{\ge0}$
	and $j_1,j_2\in \left\{ 0,1 \right\}$.}
	\label{fig:w_star_weights}
\end{figure}

\subsection{Yang--Baxter equation}

The weights $w_{u,s}$ and $w^*_{v,s}$ are very special in that they satisfy a Yang--Baxter (RLL type)
equation. To write it down, we need additional cross vertex weights $R_z$ which also depend on 
$q$ (but not on the spin parameter $s$). The weights $R_z$ are given in 
\Cref{fig:R_weights}.

\begin{figure}[htpb]
	\centering
	\includegraphics{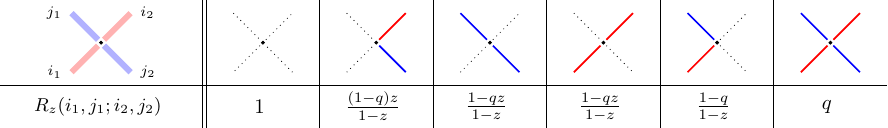}
	\caption{Cross vertex weights $R_z$.
	We have $i_1,j_1,i_2,j_2\in \left\{ 0,1 \right\}$.
	It is convenient to think that the red edges (labeled $i_1,i_2$) are directed to the right, while 
	the blue edges (labeled $j_1,j_2$) are directed to the left.}
	\label{fig:R_weights}
\end{figure}

\begin{proposition}[Yang--Baxter equation]
	\label{prop:YBE_w_wstar}
	For any $i_1,i_2,j_1,j_2\in  \left\{ 0,1 \right\}$ and 
	$i_3,j_3\in \mathbb{Z}_{\ge0}$ we have
	\begin{equation}
		\label{eq:YBE_w_wstar}
		\begin{split}
			&\sum_{k_1,k_2,k_3}R
			_{u v}(i_2, i_1; k_2, k_1)\,
			w^*_{v,s} (i_3, k_1; k_3, j_1)\,
			w_{u,s}(k_3,k_2; j_3,j_2) \\
			&\hspace{100pt}
			= 
			\sum_{k'_1,k'_2,k'_3} w^*_{v,s} (k'_3, i_1; j_3, k'_1)\,
			w_{u,s}(i_3,i_2; k'_3,k'_2)\,
			R_{u v}(k'_2, k'_1; j_2, j_1).
		\end{split}
	\end{equation}
	The sum in left-hand side of 
	\eqref{eq:YBE_w_wstar} goes over $k_1,k_2\in  \left\{ 0,1 \right\}$
	and $k_3\in \mathbb{Z}_{\ge0}$, and similarly for the right-hand side.
	See \Cref{fig:YBE_w_wstar} for an illustration.
\end{proposition}
\begin{proof}
	The Yang--Baxter equation \eqref{eq:YBE_w_wstar} 
	can be verified directly by considering 
	16 cases for $i_1,i_2,j_1,j_2\in \left\{ 0,1 \right\}$
	and arbitrary $i_3,j_3\in \mathbb{Z}_{\ge0}$.
	See \cite[Appendix B]{BufetovPetrovYB2017} for an explicit 
	listing of similar identities, and
	\cite{Borodin2014vertex} or \cite[(4.8)]{Mangazeev2014}
	for discussions of how the Yang--Baxter equation arises from 
	quantum integrability. 
\end{proof}

\begin{figure}[htpb]
	\centering
	\includegraphics{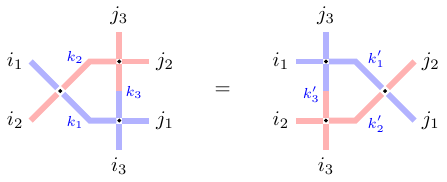}
	\caption{Graphical illustration of the Yang--Baxter equation \eqref{eq:YBE_w_wstar}.
	In the left-hand side, the weight $w_{u,s}$ is at the top, and 
	$w_{v,s}^*$ is at the bottom.
	The weights are switched in the right-hand side.}
	\label{fig:YBE_w_wstar}
\end{figure}

The cross vertex weights $R_{uv}$ in 
\eqref{eq:YBE_w_wstar} do not depend on $s$.
Moreover, they depend on the spectral parameters $u,v$ attached to the 
vertices $w_{u,s}$ and $w_{v,s}^*$ only through their product
$uv$.
Therefore, the same Yang--Baxter equation (with the same $R_{uv}$)
holds for a whole family of weights
$(w_{u\xi,s},w_{v/\xi,s}^*)$, where both $\xi$ and $s$ are arbitrary.
We use this flexibility in the next section to define fully
inhomogeneous spin Hall--Littlewood symmetric rational functions.

\subsection{Spin Hall--Littlewood symmetric functions}

Spin Hall--Littlewood symmetric functions are indexed by
signatures. We refer to \Cref{sub:notation} for notation related to signatures.
Here we need only nonnegative signatures.

Fix $q$, spin parameters $s_x$
and inhomogeneity parameters $\xi_x$, where $x=0,1,2,\ldots $ is the horizontal coordinate.
We define the \emph{spin Hall--Littlewood function}
$\mathsf{F}_{\lambda}(u_1,\ldots,u_N )$,
$\lambda\in \mathrm{Sign}_N$, as the partition 
function
of up-right path ensembles
in $\mathbb{Z}_{\ge0}\times \left\{ 1,\ldots,N  \right\}$,
such that:
\begin{enumerate}[$\bullet$]
	\item A path enters the region at each location 
		$(0,i)$, $i=1,\ldots,N $, on the left boundary;
	\item A path exits the region at locations
		$(\lambda_j,N)$, $j=1,\ldots,N $
		on the top boundary (multiple paths may exit through the same vertical edge);
	\item 
		No paths enter through the bottom boundary or escape to infinity far to the right;
	\item 
		At each vertex $(x,i)$, $x\in \mathbb{Z}_{\ge0}$, $i\in \left\{ 1,\ldots,N  \right\}$, 
		the vertex weight is taken to be $w_{u_i\xi_x,s_x}$.
\end{enumerate}
See \Cref{fig:F_Fstar_partition_functions} (left) for an illustration.
Note that this partition function is well-defined since the 
weight of the empty vertex is $w_{u,s}(0,0;0,0)=1$.
By the very definition, $\mathsf{F}_\lambda(u_1,\ldots,u_N )$ is a 
rational function of
the variables $u_i$ as well as of all parameters.
The functions $\mathsf{F}_\lambda$ appeared
in \cite{Borodin2014vertex}
in the homogeneous case $\xi_x\equiv 1$, $s_x\equiv s$, 
and the inhomogeneous generalization is performed in 
\cite{BorodinPetrov2016inhom}.

The \emph{dual spin Hall--Littlewood functions} are defined
in a similar way using the weights $w^*_{v_i/\xi_x,s_x}$
and ensembles of up-left paths, see \Cref{fig:F_Fstar_partition_functions} (right) for an illustration.
(Note that in \cite{Borodin2014vertex}, \cite{BorodinPetrov2016inhom} 
the functions $\mathsf{F}^*_\lambda$ were denoted by $\mathsf{F}^{\mathsf{c}}_\lambda$.)

\begin{figure}[htpb]
	\centering
	\includegraphics[width=\textwidth]{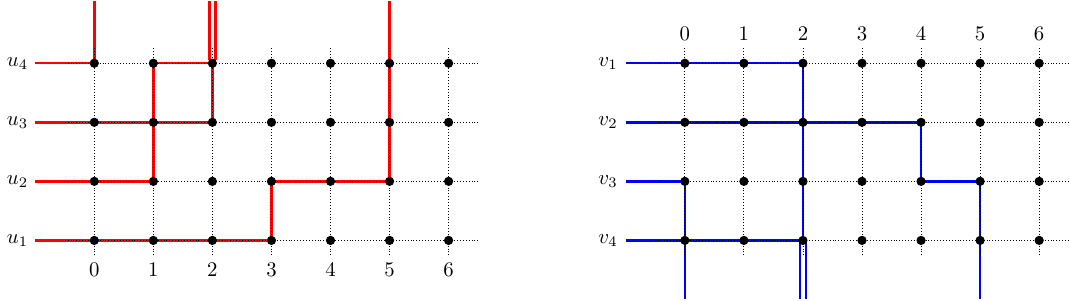}
	\caption{Left: An example of a path configuration
	contributing to the partition function $\mathsf{F}_\lambda(u_1,u_2,u_3,u_4)$, where
	$\lambda=(5,2,2,0)\in \mathrm{Sign}_4$.
	Right: An example of a path configuration for $\mathsf{F}^*_\lambda(v_1,v_2,v_3,v_4)$ 
	with the same $\lambda$.}
	\label{fig:F_Fstar_partition_functions}
\end{figure}

\subsection{Properties of spin Hall--Littlewood functions}

Here we summarize some properties of the spin Hall--Littlewood functions 
$\mathsf{F}_\lambda, \mathsf{F}^*_{\lambda}$.
First, 
for any $\lambda\in \mathrm{Sign}_N$, 
the functions 
$\mathsf{F}_\lambda(u_1,\ldots,u_N )$
and 
$\mathsf{F}^*_\lambda(v_1,\ldots,v_N )$
are \emph{symmetric} in their spectral parameters $u_i$ and $v_j$, respectively. 
This fact 
follows from the Yang--Baxter equation, cf. \cite[Theorem 3.6]{Borodin2014vertex} or
\cite[Proposition 4.5]{BorodinPetrov2016inhom}.

\medskip
Next, we can express the dual functions through the usual ones as follows:
\begin{equation}
	\label{eq:F_Fstar_conjugation}
	\mathsf{F}^*_{\lambda}(v_1,\ldots,v_N )=
	\prod_{r\ge0}\frac{(s_r^2;q)_{m_r}}{(q;q)_{m_r}}\,
	\mathsf{F}_{\lambda}(v_1,\ldots,v_N )\Big\vert_{\text{$\xi_x\to \xi_x^{-1}$ for all $x$}},
\end{equation}
where $\lambda=(0^{m_0}1^{m_1}2^{m_2}\ldots )$ in the multiplicative notation.
This follows from relation \eqref{eq:w_wstar_relation} between $w_{v/\xi,s}$ and $w^*_{v/\xi,s}$
together with telescopic cancellations occurring in weights
of path configurations.

\medskip 
The functions $\mathsf{F}_\lambda$ admit an explicit \emph{symmetrization formula}. 
Let 
\begin{equation}
	\label{eq:phi_notation}
	\varphi_k(u):=\frac{1-q}{1-s_k\xi_k u}\prod_{j=0}^{k-1}
	\frac{\xi_ju-s}{1-s_j\xi_j u},\qquad k\ge0.
\end{equation}
Then \cite[Theorem 4.14.1]{BorodinPetrov2016inhom}
\begin{equation}
	\label{eq:F_symmetrization}
	\mathsf{F}_\lambda(u_1,\ldots,u_N )
	=
	\sum_{\sigma\in S_N}
	\sigma
	\Biggl( 
		\prod_{1\le i<j\le N}\frac{u_i-q u_j}{u_i-u_j}\,
		\prod_{i=1}^{N}\varphi_{\lambda_i}(u_i)
	\Biggr),
\end{equation}
where $S_N$ is the symmetric group, and $\sigma$ acts by permuting the $u_i$'s and not the $\lambda_i$'s.
Formula \eqref{eq:F_symmetrization} is nontrivial and follows from an application of the algebraic Bethe Ansatz.
Consequently, via 
\eqref{eq:F_Fstar_conjugation}
the dual functions $\mathsf{F}^*_\lambda$ also possess a similar symmetrization formula:
\begin{equation}
	\label{eq:F_star_explicit}
		\mathsf{F}^*_\lambda(v_1,\ldots,v_N )
		=
		\prod_{r\ge0}\frac{(s_r^2;q)_{m_r(\lambda)}}{(q;q)_{m_r(\lambda)}}
		\sum_{\sigma\in S_N}
		\sigma
		\Biggl( 
			\prod_{1\le i<j\le N}\frac{v_i-q v_j}{v_i-v_j}\,
			\prod_{i=1}^{N}
			\biggl(
				\frac{1-q}{1-s_{\lambda_i}\xi_{\lambda_i}^{-1}v_i}
				\prod_{j=0}^{\lambda_i-1}\frac{\xi_j^{-1} v_i-s_j}{1-s_j \xi_j^{-1} v_i}
			\biggr)
		\Biggr).
\end{equation}

The functions $\mathsf{F}_\lambda$ satisfy a 
\emph{Cauchy type identity} together with another family of functions
denoted by $\mathsf{G}_{\lambda}^*(v_1,\ldots,v_M )$. 
The $\mathsf{G}_{\lambda}^*$'s are partition functions 
of configurations as in \Cref{fig:F_Fstar_partition_functions} (right) with weights $w^*_{v_i/\xi_x,s_x}$, 
but with different boundary conditions: the exiting paths all exit vertically through
$(0,N)$ instead of horizontally.
We refer to \cite[Section 4.3]{BorodinPetrov2016inhom} for details on the functions $\mathsf{G}^*_\lambda$,
see also \Cref{sub:Olshanski_from_us} below for an explicit symmetrization formula for 
$\mathsf{G}_\lambda^*$.
The Cauchy type identity reads \cite[Corollary 4.13]{BorodinPetrov2016inhom}
\begin{equation}
	\label{eq:F_G_Cauchy}
	\sum_{\lambda\in \mathrm{Sign}_N}
	\mathsf{F}_\lambda(u_1,\ldots,u_N )
	\,
	\mathsf{G}^*_\lambda(v_1,\ldots,v_M )=
	\frac{(q;q)_N}{\prod_{i=1}^{N}(1-s_0\xi_0 u_i)}\,
	\prod_{i=1}^{N}\prod_{j=1}^{M}\frac{1-q u_iv_j}{1-u_iv_j},
\end{equation}
provided that the variables satisfy
\begin{equation}
	\label{eq:admissible_u_v}
	\left|
	\frac{\xi_x u_i-s_x}{1-s_x\xi_x u_i}
	\cdot
	\frac{v_j/\xi_x-s_x}{1-s_x v_j/\xi_x}
	\right|\le 1-\varepsilon<1
	\qquad \textnormal{for all $i,j$ and all $x=0,1,2,\ldots $.}
\end{equation}
Note that when $\xi_x>0$, $-1<s_x\le 0$, and $u_i,v_j>0$, condition \eqref{eq:admissible_u_v}
follows from $u_iv_j\le 1-\varepsilon'<1$ (for a suitable choice of $\varepsilon'>0$).
Identity \eqref{eq:F_G_Cauchy} follows
from the Yang--Baxter equation.

\medskip
The following \emph{torus orthogonality} holds \cite[Corollary 7.5]{BorodinPetrov2016inhom}.
Provided some technical conditions on the parameters (which are given in the cited statement
and are not explicitly required for our purposes), 
for all $\lambda,\mu\in \mathrm{Sign}_N$ we have
\begin{equation}
	\label{eq:torus_orthogonality}
	\frac{1}{N!(2\pi \mathbf{i})^N}
	\oint du_1 \ldots \oint du_N \,
	\frac{\prod_{i\ne j}(u_i-u_j)}{\prod_{i,j=1}^{N}(u_i-q u_j)}\,
	\mathsf{F}_\lambda(u_1,\ldots,u_N )\,
	\mathsf{F}_\mu^*(u_1^{-1},\ldots,u_N^{-1} )
	=\mathbf{1}_{\lambda=\mu},
\end{equation}
where each integration is over the same positively oriented simple closed contour $C$
encircling all $s_x q^{j}/\xi_x$ (where $x\ge0$, $0\le j\le N-1$), and 
such that $C$ encircles the contour $qC$ (the image of $C$ under the multiplication by $q$).

\medskip

We also have an \emph{orthogonality relation of another type}:
\begin{equation}
	\label{eq:spectral_orthogonality}
	\begin{split}
		&
		\sum_{-\infty<\lambda_N\le \lambda_{N-1}\le\ldots\le \lambda_1< +\infty }
		\mathsf{F}_\lambda(u_1,\ldots,u_N )\,
		\mathsf{F}_\lambda^*(v_1^{-1},\ldots,v_N^{-1} )
		\\&\hspace{100pt}=
		(-1)^{\frac{N(N-1)}{2}}
		\frac{\prod_{i,j=1}^{N}(u_i-q v_j)}{\prod_{1\le i<j\le N}
		(u_i-u_j)(v_i-v_j)}\,\det\left[ \delta(v_i-u_j) \right]_{i,j=1}^{N}.
	\end{split}
\end{equation}
Here $\delta(\cdot)$ is the Dirac delta, and identity \eqref{eq:spectral_orthogonality}
should be understood in an integral sense (we refer to 
\cite[Theorem 7.7]{BorodinPetrov2016inhom} for detailed formulation and proof).
Note also that the summation in the left-hand side of \eqref{eq:spectral_orthogonality}
is over signatures $\lambda$ which are not necessarily nonnegative. For signatures with 
possibly negative parts, we trivially extend the definition of $\mathsf{F}_\lambda$ using the
shifting property
\begin{multline*}
	\mathsf{F}_{(\lambda_1+r,\ldots,\lambda_N+r )}(u_1,\ldots,u_N )
	\\=
	\prod_{i=1}^{N}\prod_{j=0}^{r-1}\frac{\xi_j u_i-s_j}{1-s_j\xi_j u_i}
	\left( 
		\mathsf{F}_{(\lambda_1,\ldots,\lambda_N )}(u_1,\ldots,u_N )\Big\vert_{(\xi_x,s_x)\to (\xi_{x+r},s_{x+r})\ \text{for all $x$}}
	\right)
	,
\end{multline*}
where $r\ge1$, 
and similarly for the $\mathsf{F}^*_{\lambda}$'s.

In \cite{BCPS2014_arXiv_v4} and \cite{BorodinPetrov2016inhom}
the torus orthogonality \eqref{eq:torus_orthogonality} is called 
\emph{spatial}, and \eqref{eq:spectral_orthogonality} is
referred to as the \emph{spectral orthogonality}.

\subsection{Stable spin Hall--Littlewood functions}

\label{sub:stable_sHL}

%
%

A useful variant (in fact, a particular case)
of the spin Hall--Littlewood functions was
introduced in
\cite{deGierWheeler2016}, \cite{BorodinWheelerSpinq}.
The
\emph{stable spin Hall--Littlewood functions}
$\widetilde{\mathsf{F}}_\lambda$,
$\widetilde{\mathsf{F}}^*_\lambda$ are
indexed by \emph{partitions} $\lambda=(\lambda_1\ge \lambda_2\ge \ldots\ge \lambda_{\ell(\lambda)}>0 )$
instead of signatures. That is, these functions only care about 
nonzero parts of $\lambda$, and the partition $\lambda$ 
may have an arbitrary length.
The passage to the stable functions is done by inserting infinitely many
incoming and outgoing vertical paths at location~$0$.
In particular, the resulting stable functions may depend on an arbitrary number of variables
$k\ge \ell(\lambda)$,
not tied with the length of the signature (if $k<\ell(\lambda)$, the stable function is zero, 
by agreement).

There are two ways to obtain stable functions indexed
by a partition $\lambda=(\lambda_1\ge \ldots\ge \lambda_{\ell}>0 )$. 
First, we have (cf. \cite[Section 3.3]{BufetovMucciconiPetrov2018})
\begin{equation}
	\label{eq:stable_f_1}
	\widetilde{\mathsf{F}}_\lambda(u_1,\ldots,u_k )
	=
	\prod_{i=1}^{k}(1-s_0\xi_0u_i)\cdot\lim_{m\to+\infty}\mathsf{F}_{\lambda\cup 0^{m+k-\ell}/0^m}(u_1,\ldots,u_k ),
\end{equation}
where in the right-hand side we have signatures with growing numbers of zeroes.
Alternatively, 
\begin{equation}
	\label{eq:stable_f_2}
	\widetilde{\mathsf{F}}_{\lambda}(u_1,\ldots,u_k )=
	\frac{1}{(q;q)_{k-\ell}}\,
	\mathsf{F}_{\lambda\cup 0^{k-\ell}}(u_1,\ldots,u_k )\Big\vert_{s_0=0}.
\end{equation}
Note that the stable functions $\widetilde{\mathsf{F}}_{\lambda}$ do not depend on $s_0$.
While this is clear from \eqref{eq:stable_f_2}, 
in \eqref{eq:stable_f_1} we needed the
prefactor to cancel the corresponding denominators.
We will mostly use the second way \eqref{eq:stable_f_2}
and so refer to 
\cite[Definitions 4.3 and 4.4]{BorodinPetrov2016inhom}
for
the definition of the skew spin Hall--Littlewood
functions 
$\mathsf{F}_{\lambda\cup 0^{m+k-\ell}/0^m}$.

Analogously, the dual stable functions are given by 
\begin{equation}
	\label{eq:dual_stable_f_1}
	\widetilde{\mathsf{F}}_\lambda^*(v_1,\ldots,v_k )
	=
	\prod_{i=1}^{k}(1-s_0v_i/\xi_0)\cdot\lim_{m\to+\infty}\mathsf{F}_{\lambda\cup 0^{m+k-\ell}/0^m}^*(v_1,\ldots,v_k ),
\end{equation}
or equivalently by
\begin{equation}
	\label{eq:dual_stable_f_2}
	\widetilde{\mathsf{F}}_{\lambda}^*(v_1,\ldots,v_k )=
	\mathsf{F}_{\lambda\cup 0^{k-\ell}}^*(v_1,\ldots,v_k )\Big\vert_{s_0=0}.
\end{equation}

The name ``stable'' comes from the fact that
\begin{equation*}
	\widetilde{\mathsf{F}}_\lambda(u_1,\ldots,u_{k-1},0 )
	=
	\widetilde{\mathsf{F}}_\lambda(u_1,\ldots,u_{k-1} )
	,
\end{equation*}
and similarly for the dual functions.

The stable spin Hall--Littlewood functions also satisfy the following Cauchy type identity:
\begin{proposition}[{\cite[Corollary 7.6]{BorodinWheelerSpinq}}]
	\label{prop:stable_Cauchy}
	If the variables satisfy \eqref{eq:admissible_u_v},
	then 
	\begin{equation}
		\label{eq:stable_Cauchy}
		\sum_{\lambda}
		\widetilde{\mathsf{F}}_\lambda(u_1,\ldots,u_n )\,
		\widetilde{\mathsf{F}}_\lambda^*(v_1,\ldots,v_m )=
		\prod_{i=1}^{n}\prod_{j=1}^{m}
		\frac{1-qu_iv_j}{1-u_iv_j},
	\end{equation}
	where the sum runs over all partitions $\lambda=(\lambda_1\ge \ldots\ge \lambda_{\ell(\lambda)}>0 )$
	of arbitrary length.
\end{proposition}

\section{Refined Cauchy identity. Proof of \texorpdfstring{\Cref{thm:intro_sHL_refined}}{the first main result}}
\label{sec:sHL_proof}

In this section we use the Yang--Baxter equation 
and Lagrange interpolation method to establish
our first main result, \Cref{thm:intro_sHL_refined}.

\subsection{Two partition functions}

We begin by defining two partition functions
depending on an integer $N$, spectral parameters
$u_1,\ldots,u_N $, $v_1,\ldots,v_N $, 
other parameters $q,s_x,\xi_x$
of the model, and an additional integer $a\in \left\{ 0,1,2,\ldots  \right\}\cup \left\{ +\infty \right\}$.
For future use we also denote $\gamma:=q^{a}$ (when $a=+\infty$, we have $\gamma=0$),
and treat $\gamma$ as a generic complex parameter. This is possible because,
as we observe throughout the computations,
our quantities of interest,
$\mathscr{Z}^{q^a}_N$ (\Cref{def:Z})
and 
$\mathscr{S}^{q^a}_N$ (\Cref{def:S}),
are rational functions of $q^a$.
Hence these quantities
admit a meromorphic continuation in the variable $\gamma=q^a$.

\begin{definition}[Domain wall type partition function]
	\label{def:Z}
	Denote by 
	\begin{equation*}
		\mathscr{Z}^{\gamma}_{N}=\mathscr{Z}^{\gamma}_{N}(u_1,\ldots,u_N\mid v_1,\ldots,v_N;s_0)
	\end{equation*}
	the partition
	function of the cross vertex configuration as in \Cref{fig:Z_S_partition_functions} (left), 
	where the cross vertex weights are $R_{u_iv_j}$, 
	and the weights on the right
	are $w_{u_i\xi_0,s_0}$ and $w^*_{v_j/\xi_0,s_0}$.
	The boundary conditions are $1$'s on the left, $0$'s on the right, and 
	there are $a$ vertical arrows entering from below 
	and exiting from the top. 

	When $a=0$ (that is, $\gamma=1$), 
	we can ignore the vertical column of vertices on the right, and so
	$\mathscr{Z}^{\gamma=1}_{N}$ essentially becomes the partition function of the inhomogeneous six vertex model with 
	weights $R_{u_iv_j}$ (given in \Cref{fig:R_weights})
	and domain wall boundary conditions.
	For general $a\ne 0$ (that is, for $\gamma\ne 1$) we may call the boundary conditions
	the \emph{decorated domain wall}. 
	The decorated boundary conditions
	depend on three extra parameters, $s_0$, $\xi_0$, and $\gamma=q^a$.
\end{definition}

\begin{definition}[Refined Cauchy partition function]
	\label{def:S}
	Assume that the spectral parameters $u_i,v_j$ satisfy
	\eqref{eq:admissible_u_v}. Denote by
	\begin{equation*}
		\mathscr{S}^{\gamma}_{N}=\mathscr{S}^{\gamma}_{N}(u_1,\ldots,u_N\mid v_1,\ldots,v_N;s_0)
	\end{equation*}
	the partition function of the vertex configuration as in \Cref{fig:Z_S_partition_functions} (right).
	Here the vertex weights in the bottom part are $w_{u_i\xi_x,s_x}$ and 
	the weights in the top part are $w_{v_j/\xi_x,s_x}^*$.
	The boundary conditions are $1$ on the left, $0$ on the right and 
	everywhere at the top and the bottom except the zeroth column.
	In the zeroth column, there are $a$ arrows entering from the bottom
	and exiting from the top.
\end{definition}

Note that the number of configurations contributing to 
$\mathscr{Z}_N^\gamma$
is finite, while this number is infinite for 
$\mathscr{S}_N^\gamma$. Hence the need for the convergence 
condition \eqref{eq:admissible_u_v} in \Cref{def:S}.
Also, we explicitly
indicate the dependence of the partition functions $\mathscr{Z}^{\gamma}_N$ and 
$\mathscr{S}^{\gamma}_N$ on $s_0$ for future convenience,
as in some formulas we would like to change the value
of $s_0$. Since all other parameters $\xi_0$ and $(s_x,\xi_x)$, $x\ge1$,
remain the same, we do not indicate them in the notation.

\begin{figure}[htpb]
	\centering
	\includegraphics[width=.9\textwidth]{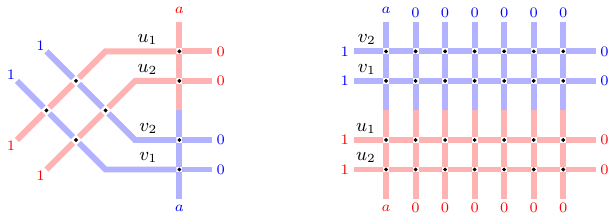}
	\caption{Boundary conditions and spectral
	parameters corresponding to the partition functions
	$\mathscr{Z}^{\gamma}_{N}$ (left) and
	$\mathscr{S}^{\gamma}_{N}$ (right) with $N=2$. The cross vertex weights
	are $R_{u_iv_j}$, and the grid vertex weights are
	$w_{u_i\xi_x,s_x}$ (red) and $w^*_{v_i/\xi_x,s_x}$ (blue).}
	\label{fig:Z_S_partition_functions}
\end{figure}

\begin{remark}
	\label{rmk:S_via_skew}
	Using skew spin Hall--Littlewood functions
	(we refer to 
	\cite[Definitions 4.3 and 4.4]{BorodinPetrov2016inhom}
	for their definition),
	we can write
	\begin{equation*}
		\mathscr{S}^{\gamma}_{N}(u_1,\ldots,u_N\mid v_1,\ldots,v_N;s_0)
		=
		\sum_{\lambda\in \mathrm{Sign}_{N} }
		\mathsf{F}_{\lambda\cup 0^{a}/0^a}(u_1,\ldots,u_N )\,
		\mathsf{F}^*_{\lambda\cup 0^{a}/0^{a}}(v_1,\ldots,v_N ).
	\end{equation*}
\end{remark}

\subsection{Refined Cauchy type sum}

Here we deal with the refined Cauchy type sum $\mathscr{S}^{\gamma}_{N}$,
and rewrite it in terms of the spin Hall--Littlewood functions.
The resulting expression would look similar to the 
left-hand sides of the known Cauchy identities \eqref{eq:F_G_Cauchy}, \eqref{eq:stable_Cauchy},
but with a new \emph{refinement factor}.

\begin{proposition}
	\label{prop:S_as_refined}
	Let the parameters $u_i,v_j$ satisfy \eqref{eq:admissible_u_v}.
	When $\gamma\ne 0$, we have
	\begin{equation}
		\label{eq:S_N_general}
		\begin{split}
			&\mathscr{S}_{N}^{\gamma}(u_1,\ldots,u_N\mid v_1,\ldots,v_N;\gamma^{-1}s_0)
			=
			\prod_{j=1}^{N}
			\frac{1-\xi_0s_0u_j}{1-\xi_0s_0\gamma^{-1}u_j}\,
			\frac{1-s_0 v_j/\xi_0}{1-s_0 \gamma^{-1}v_j/\xi_0}
			\\
			&\hspace{60pt}\times
			\sum_{\lambda\in \mathrm{Sign}_{N}}
			\frac{(\gamma q;q)_{N-\ell(\lambda)}}{(q;q)_{N-\ell(\lambda)}}
			\frac{(\gamma^{-1}s_0^2;q)_{N-\ell(\lambda)}}{(s_0^2;q)_{N-\ell(\lambda)}}
			\,
			\mathsf{F}_{\lambda}(u_1,\ldots,u_N )\,
			\mathsf{F}^*_{\lambda}(v_1,\ldots,v_N ).
	\end{split}
	\end{equation}
	Here in the right-hand side the spin Hall--Littlewood 
	functions $\mathsf{F}_\lambda,\mathsf{F}_\lambda^*$
	contain the original parameter $s_0$.
	Moreover, 
	for $\gamma=0$ we have
	\begin{equation}
		\label{eq:S_N_gamma0}
		\begin{split}
			&\mathscr{S}^{\gamma=0}_N(u_1,\ldots,u_N\mid v_1,\ldots,v_N ;s_0 )
			=
			\prod_{j=1}^{N}\frac{1}{(1-s_0\xi_0u_j)(1-s_0v_j/\xi_0)}
			\\&\hspace{180pt}\times
			\sum_{\nu}
			\widetilde{\mathsf{F}}_{\nu}(u_1,\ldots,u_N )\,
			\widetilde{\mathsf{F}}^*_{\nu}(v_1,\ldots,v_N ),
		\end{split}
	\end{equation}
	where the sum is over partitions $\nu$ of length at most $N$, and 
	$\widetilde{\mathsf{F}}_{\nu},\widetilde{\mathsf{F}}^*_{\nu}$
	are the stable spin Hall--Littlewood functions.
\end{proposition}
\begin{proof}
	Throughout the proof we assume that $a\ge N$, so that no negative arrow numbers occur
	in the partition function, and no configurations are forbidden.
	As the resulting identity depends on $\gamma=q^a$ in a rational way
	(due to \eqref{eq:admissible_u_v}
	all infinite sums converge, and are equal to rational functions),
	this assumption
	does not restrict the generality.

	Fix an arbitrary path configuration contributing to the partition function
	$\mathscr{S}^\gamma_N$.
	Let $\lambda\in \mathrm{Sign}_N$ encode the 
	intermediate arrow configuration between the red and the blue parts.
	In the bottom (red) part of \Cref{fig:Z_S_partition_functions} (right),
	the number of arrows in the zeroth column at height $N$ is 
	equal to $a+N-\ell(\lambda)$, where $\ell(\lambda)$ is the number of nonzero parts in $\lambda$.
	Apart from the zeroth column, the weights $w_{u_j\xi_x,s_x}$, $x\ge1$, of all vertices 
	are the same as in the definition of $\mathsf{F}_{\lambda}(u_1,\ldots,u_N )$.
	In the zeroth column, there are two types of vertices, and their weights depend on $\gamma$
	in the following way:
	\begin{equation}
		\label{eq:F_gamma_modified_computation}
		\begin{split}
			w_{u_j\xi_0,s_0/\gamma}(a+i,1;a+i,1)=
			\frac{\xi_0u_j-s_0 q^i}{1-\xi_0s_0u_j/\gamma}
			,\\
			w_{u_j\xi_0,s_0/\gamma}(a+i,1;a+i+1,0)=
			\frac{1-\gamma q^{i+1}}{1-\xi_0s_0u_j/\gamma}
			.
		\end{split}
	\end{equation}
	Let us first look at the numerators in 
	\eqref{eq:F_gamma_modified_computation}.
	The number of vertices of type
	$(a+i,1;a+i+1,0)$
	is equal to $N-\ell(\lambda)$, since only $\ell(\lambda)$ paths leave the zeroth column. 
	These vertices correspond 
	to the number $i$ ranging from $0$ to $N-\ell(\lambda)-1$.
	We see that by taking out the prefactor
	$\frac{(\gamma q;q)_{N-\ell(\lambda)}}{(q;q)_{N-\ell(\lambda)}}$
	from the weight of the whole path configuration, we may remove the $\gamma$-modification from the numerators of the weights of
	the second type of vertices in \eqref{eq:F_gamma_modified_computation}.

	It remains to remove the $\gamma$-modification
	from the denominators in \eqref{eq:F_gamma_modified_computation}, 
	and this is achieved by taking out the factor
	\begin{equation*}
		\prod_{i=1}^{N}\frac{1-\xi_0s_0u_j}{1-\xi_0s_0\gamma^{-1}u_j}.
	\end{equation*}
	The resulting vertex weights are the same 
	as the ones entering 
	$\mathsf{F}_\lambda$. Therefore, for fixed $\lambda$
	the partition function of the 
	bottom (red) part in \Cref{fig:Z_S_partition_functions} (right) is equal to 
	\begin{equation*}
		\frac{(\gamma q;q)_{N-\ell(\lambda)}}{(q;q)_{N-\ell(\lambda)}}
		\prod_{j=1}^{N}\frac{1-\xi_0s_0u_j}{1-\xi_0s_0\gamma^{-1}u_j}\,
		\mathsf{F}_\lambda(u_1,\ldots,u_N ).
	\end{equation*}

	For the top (blue) part in \Cref{fig:Z_S_partition_functions} (right)
	we argue in a similar way by
	relating the partition function of the top part to 
	$\mathsf{F}^*_\lambda(v_1,\ldots,v_N )$.
	In the zeroth column, there are two types of vertices with weights depending on $\gamma$
	as
	\begin{equation}
		\label{eq:FStar_gamma_modified_computation}
		\begin{split}
			w^*_{v_j/\xi_0,s_0/\gamma}
			( a+i,1; a+i,1)
			&=
			\frac{v_j/\xi_0-s_0 q^i}{1-s_0\gamma^{-1}v_j/\xi_0}
			,
			\\
			w^*_{v_j/\xi_0,s_0/\gamma}
			( a+i+1,1; a+i,0)
			&=
			\frac{1-s_0^2 \gamma^{-1} q^{i}}{1-s_0\gamma^{-1}v_j/\xi_0}
			.
		\end{split}
	\end{equation}
	Similarly to the bottom part, we first take out the 
	factor 
	$
	\frac{(s_0^2 \gamma^{-1} ;q)_{N-\ell(\lambda)}}{(s_0^2;q)_{N-\ell(\lambda)}}
	$
	which deals with the numerators in the vertex weights
	of the second type in 
	\eqref{eq:FStar_gamma_modified_computation}.
	Then, to compensate for the 
	denominators, we take out a suitable product over $j=1,\ldots,N $. 
	This implies that for fixed $\lambda$ the partition function of the 
	top (blue) part in \Cref{fig:Z_S_partition_functions} (right)
	is equal to 
	\begin{equation*}
		\frac{(\gamma^{-1}s_0^2;q)_{N-\ell(\lambda)}}{(s_0^2;q)_{N-\ell(\lambda)}}
		\prod_{j=1}^{N}
		\frac{1-s_0 v_j/\xi_0}{1-s_0 \gamma^{-1}v_j/\xi_0}
		\,
		\mathsf{F}^*_\lambda(v_1,\ldots,v_N ).
	\end{equation*}
	Putting all together and summing over $\lambda$, we get the first claim.

	The second claim follows from the definition of the stable spin Hall--Littlewood functions
	(\Cref{sub:stable_sHL})
	via inserting infinitely many arrows in the zeroth column.
	Indeed, setting
	$\gamma=0$ corresponds to $a=+\infty$.
\end{proof}

\subsection{Equality of partition functions}

Next we use the Yang--Baxter equation to relate the partition
functions 
$\mathscr{Z}^\gamma_{N}$
and
$\mathscr{S}^\gamma_{N}$.

\begin{proposition}
	\label{prop:equality_Z_S}
	We have 
	\begin{equation*}
		\mathscr{Z}^\gamma_{N}
		(u_1,\ldots,u_N\mid v_1,\ldots,v_N;s_0 )
		=
		\mathscr{S}^\gamma_{N}
		(u_1,\ldots,u_N\mid v_1,\ldots,v_N;s_0 )
	\end{equation*}
	provided that the spectral parameters $u_i,v_j$
	satisfy \eqref{eq:admissible_u_v}.
\end{proposition}
\begin{proof}
	Start with the configuration of the lattice corresponding to 
	$\mathscr{Z}^{\gamma}_{N}$.
	Drag all the cross vertices to the right using the Yang--Baxter
	equation (\Cref{prop:YBE_w_wstar}).
	Condition \eqref{eq:admissible_u_v} ensures that 
	after moving the crosses, the end state of the 
	cross vertices is $(0,0;0,0)$.
	Indeed, keeping the other eventual state $(1,1;0,0)$ of the cross vertex
	introduces infinitely many factors of the form
	$w_{u_i\xi_x,s_x}(0,1;0,1)w_{u_j/\xi_x,s_x}(0,1;0,1)$ into the weight of the path configuration.
	Condition \eqref{eq:admissible_u_v} implies that each of these factors is smaller than $1-\varepsilon$
	in the absolute value. Thus, configurations with the 
	eventual state $(1,1;0,0)$ of the cross vertex do not contribute to the partition function.
	Finally, the fact that $R_{u_iv_j}(0,0;0,0)=1$
	means that we arrive at the partition function $\mathscr{S}^{\gamma}_{N}$, 
	and so the equality of partition functions follows.
\end{proof}

\subsection{Evaluation of the Izergin--Korepin type determinant \texorpdfstring{$\mathscr{Z}_{N}^{\gamma}$}{Z}}

The third and final step of the proof of \Cref{thm:intro_sHL_refined}
consists in an explicit computation (in a determinantal form)
of the partition function
of the six vertex model with the
decorated domain wall boundary conditions.
We compute it by the Lagrange interpolation technique
(similarly to, e.g., \cite{wheeler2015refined}).

\subsubsection{Formulation of the result}

Denote
\begin{equation}
	\label{eq:IK_det_element}
	\mathsf{z}(u,v;s_0):=
	\frac{(1-\gamma)(q-\gamma s_0^2)(1-uv)+(1-q)(1-\gamma\xi_0s_0 u)(1-\gamma \xi_0^{-1}s_0v)}
	{(1-uv)(1-quv)}.
\end{equation}

\begin{proposition}
	\label{prop:Z_N_det}
	We have
	\begin{equation}
		\label{eq:Z_N_det}
		\begin{split}
			&\mathscr{Z}^\gamma_N(u_1,\ldots,u_N\mid v_1,\ldots,v_N;s_0 )
			\\&\hspace{30pt}=
			\prod_{j=1}^{N}
			\frac{1}{(1-s_0\xi_0 u_j)(1-\xi_0^{-1}s_0v_j)}\,
			\frac{\prod_{i,j=1}^{N}(1-qu_iv_j)}{\prod_{1\le i<j\le N}(u_i-u_j)(v_i-v_j)}\,
			\det\left[ \mathsf{z}(u_i,v_j;s_0) \right]_{i,j=1}^{N}.
		\end{split}
	\end{equation}
\end{proposition}

Before proving \Cref{prop:Z_N_det}, 
let us discuss three reductions.

First, when $s_0=0$, this result is established in \cite[Lemma 5]{wheeler2015refined}.

Next, keeping $s_0$ generic and setting
$\gamma=1$ (that is, $a=0$), we have
\begin{equation}
	\label{eq:Z_gamma_1}
	\mathscr{Z}^{\gamma=1}_N(u_1,\ldots,u_N\mid v_1,\ldots,v_N  )
	=
	\frac{(1-q)^N\prod_{i,j=1}^{N}(1-qu_iv_j)}{\prod_{1\le i<j\le N}(u_i-u_j)(v_i-v_j)}\,
	\det\left[ \frac{1}{(1-u_iv_j)(1-qu_iv_j)} \right]_{i,j=1}^{N}.
\end{equation}
The determinant in the right-hand side is the celebrated Izergin--Korepin determinant
\cite{Izergin1987},
\cite[Ch. VII.10]{QISM_book}.

Third, keep $s_0$ generic, and consider now the case $\gamma=0$.
We have
\begin{equation*}
	\mathsf{z}(u,v;s_0)\big\vert_{\gamma=0}=
	\frac{q(1-uv)+(1-q)}
	{(1-uv)(1-quv)}=\frac{1}{1-uv}.
\end{equation*}
Thus, for $\gamma=0$ the
partition function $\mathscr{Z}_N^{\gamma}(u_1,\ldots,u_N\mid v_1,\ldots,v_N; s_0  )$ 
simplifies using the Cauchy determinant 
(e.g.,
\cite[vol. III, p. 311]{Muir_det}) to 
\begin{equation}
	\mathscr{Z}^\gamma_N(u_1,\ldots,u_N\mid v_1,\ldots,v_N;s_0 )
	\Big\vert_{\gamma=0}=\label{eq:Z_N_gamma0_special}
	\prod_{j=1}^{N}\frac{1}{(1-s_0\xi_0u_j)(1-s_0v_j/\xi_0)}
	\prod_{i,j=1}^{N}\frac{1-qu_iv_j}{1-u_iv_j}.
\end{equation}
Formula \eqref{eq:Z_N_gamma0_special} 
agrees with the known 
Cauchy identity for the stable spin Hall--Littlewood functions
recalled in \Cref{prop:stable_Cauchy}.
To see this, use
\eqref{eq:S_N_gamma0} and \Cref{prop:equality_Z_S}.

On the other hand, formula \eqref{eq:Z_N_gamma0_special} for the 
six vertex model partition function may be obtained independently by noting that 
for $a=+\infty$, the weights
$w_{u_j\xi_0,s_0}$
and 
$w^*_{v_j/\xi_0,s_0}$
do not depend on the number of incoming horizontal arrows from the left, 
and
\begin{equation*}
	\sum_{i_2,j_2}R_{uv}(i_1,j_1;i_2,j_2)=\frac{1-quv}{1-uv}
\end{equation*}
for all $i_1,j_1$.
Therefore, the decorated domain wall partition function
factorizes.
This independent derivation of \eqref{eq:Z_N_gamma0_special}
establishes \Cref{prop:Z_N_det} for $\gamma=0$.

\medskip

The proof of \Cref{prop:Z_N_det} in the case of general $\gamma$ occupies the rest of this subsection.
We argue using Lagrange interpolation technique 
(similarly to, e.g., \cite{wheeler2015refined}) which consists of three steps.
First, directly from \Cref{def:Z} we formulate five properties
which the partition function $\mathscr{Z}^\gamma_N$ satisfies.
Then we show that these properties determine a function uniquely.
Finally, we check that the determinantal formula in the 
right-hand side of \eqref{eq:Z_N_det} satisfies the same five properties,
which establishes the desired equality.

\subsubsection{Step 1. Properties of the partition function}
\label{subsub:step1_IK}

Denote
\begin{multline*}
	\widetilde{\mathscr{Z}}^\gamma_N(u_1,\ldots,u_N\mid v_1,\ldots,v_N;s_0  )
	\\:=
	\prod_{j=1}^{N}
	(1-s_0\xi_0 u_j)(1-\xi_0^{-1}s_0v_j)
	\prod_{i,j=1}^{N}(1-u_iv_j)\,
	\mathscr{Z}^\gamma_N(u_1,\ldots,u_N\mid v_1,\ldots,v_N;s_0).
\end{multline*}
Additional factors in $\widetilde{\mathscr{Z}}^\gamma_N$
clear out all the denominators in
the vertex weights $w_{u_i \xi_0,s_0}$, $w^*_{v_j/\xi_0,s_0}$, 
and $R_{u_iv_j}$, $i,j=1,\ldots,N $. 
This means that $\widetilde{\mathscr{Z}}^\gamma_N$ now depends on the spectral parameters
$u_i,v_j$ in a polynomial manner.
As the first step in the proof of \Cref{prop:Z_N_det},
let us list a number of properties of this renormalized partition function
$\widetilde{\mathscr{Z}}^\gamma_N$.

\begin{properties}
	\label{Z_properties}
	\begin{enumerate}[\bf1.\/]
		\item The function 
			$\widetilde{\mathscr{Z}}^\gamma_N(u_1,\ldots,u_N\mid v_1,\ldots,v_N;s_0  )$
			is symmetric separately in each of the two sets of variables 
			$\left\{ u_1,\ldots,u_N  \right\}$ and 
			$\left\{ v_1,\ldots,v_N  \right\}$.
		\item 
			As a function of each single
			variable $u_i$ or $v_j$, 
			$\widetilde{\mathscr{Z}}^\gamma_N(u_1,\ldots,u_N\mid v_1,\ldots,v_N;s_0  )$
			is a polynomial of degree at most $N$.
		\item 
			Setting $u_1=v_1^{-1}$, we have the recurrence
			\begin{equation}
				\label{eq:IK_property_3}
				\begin{split}
					&\widetilde{\mathscr{Z}}^\gamma_N(u_1,u_2,\ldots,u_N\mid v_1,v_2,\ldots,v_N;s_0  )
					\Big\vert_{u_1=v_1^{-1}}
					\\&\hspace{50pt}=
					(1-q)
					(1-s_0\xi_0 \gamma v_1^{-1})
					(1-s_0\xi_0^{-1}\gamma v_1)
					\prod_{j=2}^{N}
					\left( 1-qv_1^{-1}v_j \right)\left( 1-qv_1u_j \right)
					\\&\hspace{100pt}\times
					\widetilde{\mathscr{Z}}^\gamma_{N-1}(u_2,\ldots,u_N\mid v_2,\ldots,v_N;s_0  ).
				\end{split}
			\end{equation}
		\item Under the specialization $u_j=q^{j-1} / (s_0\xi_0 \gamma)$, we have
			\begin{equation}
				\label{eq:IK_property_4}
				\begin{split}
					&\widetilde{\mathscr{Z}}^\gamma_N\left(  
						(s_0\xi_0 \gamma)^{-1},q(s_0\xi_0\gamma)^{-1},\ldots,q^{N-1}(s_0\xi_0\gamma)^{-1}
						\mid v_1,\ldots,v_N;s_0 
					\right)
					\\&\hspace{80pt}=
					q^{N^2}
					\prod_{i,j=1}^{N}\left( 1-v_i\frac{q^{j-1}}{s_0\xi_0\gamma} \right)
					\prod_{j=1}^{N}(1-\gamma q^{-j+1})(1-s_0^2 \gamma q^{-j}).
				\end{split}
			\end{equation}
		\item For $N=1$, we have
			\begin{equation}
				\label{eq:IK_property_5}
				\widetilde{\mathscr{Z}}^\gamma_1(u\mid v;s_0  )
				=(1-\gamma)(q-\gamma s_0^2)(1-uv)+(1-q)(1-\gamma\xi_0s_0 u)(1-\gamma \xi_0^{-1}s_0v).
			\end{equation}
	\end{enumerate}
\end{properties}

Properties \textbf{3}, \textbf{4}, and \textbf{5} follow from the definition
of the renormalized partition function 
$\widetilde{\mathscr{Z}}^\gamma_N$. Namely, for \textbf{3}, 
setting $u_1=v_1^{-1}$ makes the renormalized weight $(1-u_1v_1)R_{u_1v_1}(1,1;1,1)$
disappear, and so the nontrivial behavior of the paths 
reduces to a square of size $N-1$.
The factors in the right-hand side of \eqref{eq:IK_property_3} 
come from the ``frozen'' vertices.
For \textbf{4}, specializing the $u_j$ variables 
forces the configurations in the zeroth column to 
be of the form $(a-j,1;a-j+1,0)$, where $j=1,\ldots,N $ (here we assume $a\ge N$ which
does not restrict the generality since the desired identity is between rational 
functions in $\gamma=q^a$).
This requires all vertices inside the square to be of the type
$(1,1;1,1)$, and thus we get the right-hand side of \eqref{eq:IK_property_4}.
For \textbf{5}, identity \eqref{eq:IK_property_5} is straightforward.

Let us now prove properties \textbf{1} and \textbf{2}.
\begin{proof}[Proof of symmetry]
	Symmetry follows from a number of Yang--Baxter equations
	of a type different from the one in \Cref{prop:YBE_w_wstar}.
	Introduce vertex weights $r_z(i_1,j_1;i_2,j_2)$, $i_1,i_2,j_1,j_2\in \left\{ 0,1 \right\}$,
	given in \Cref{fig:r_table}.

	\begin{figure}[htpb]
		\centering
		\includegraphics{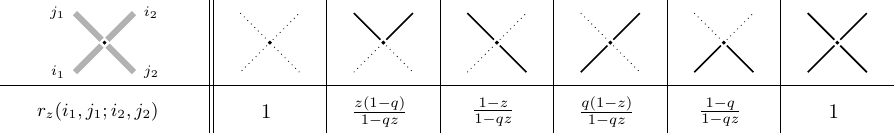}
		\caption{The vertex weights $r_z$ employed in the proof of symmetry of 
		$\widetilde{\mathscr{Z}}^\gamma_N$.}
		\label{fig:r_table}
	\end{figure}
	\begin{figure}[htpb]
		\centering
		\includegraphics[width=\textwidth]{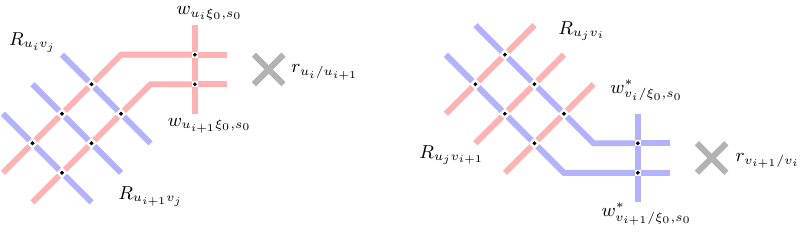}
		\caption{Interchanging spectral parameters
		$u_i\leftrightarrow u_{i+1}$ (left) or
		$v_i\leftrightarrow v_{i+1}$ (right)
		by means of Yang--Baxter equations involving
		weights $r_z$.}
		\label{fig:YBE_for_IK}
	\end{figure}
	
	The fact that we can interchange the spectral parameters $u_i,u_{i+1}$ follows from
	two
	Yang--Baxter equations. One is satisfied by the weights
	$r_{u_i/u_{i+1}}, w_{u_i\xi_0,s_0}, w_{u_{i+1}\xi_0,s_0}$,
	and the other one by 
	$r_{u_i/u_{i+1}}, R_{u_iv_j}, R_{u_{i+1}v_j}$.
	These equations have a form similar to \eqref{eq:YBE_w_wstar} and are illustrated
	in \Cref{fig:YBE_for_IK} (left). This allows to take a cross vertex of weight $r_{u_i/u_{i+1}}$
	and drag it throughout the lattice, interchanging the parameters $u_i,u_{i+1}$ everywhere.
	The fact that we can interchange $v_i,v_{i+1}$ follows similarly,
	see \Cref{fig:YBE_for_IK} (right).
	All Yang--Baxter equations involved are verified in a straightforward
	way. 
\end{proof}

\begin{proof}[Proof of polynomiality of degree $\le N$]
	Each variable $u_i$ and $v_j$ enters into $N+1$ vertices --- $N$
	cross vertices in the square, and one vertex in the zeroth column (cf. \Cref{fig:Z_S_partition_functions}, left).
	After renormalization, all vertex weights become linear in their respective spectral parameters. Thus, the
	degree of $\widetilde{\mathscr{Z}}^\gamma_N$ in each variable $u_i$ or $v_j$ is at most $N+1$.
	It remains to show that the coefficient by the $(N+1)$-st power is zero.

	Due to symmetry, it suffices to consider only $u_N$ (the case of $v_N$ is analogous).
	In order to get a nonzero coefficient by $u_N^{N+1}$, 
	the configuration of paths inside
	the $N\times N$ square
	should avoid weights
	$R_{u_Nv_j}(1,1;0,0)$ because their renormalized weight equals $1-q$.
	This implies that in the zeroth column, 
	the vertex containing the spectral parameter $u_N$
	would be of type $(g,1;g+1,0)$. 
	The renormalized weight of the latter vertex is $1-q^{g+1}$ which does not depend on $u_N$.
	We see that it is impossible to get a nonzero coefficient by $u_N^{N+1}$, so the 
	total degree is at most $N$.
\end{proof}

\subsubsection{Step 2. Uniqueness}
\label{subsub:step2_IK}

The fact that \Cref{Z_properties}
determine $\widetilde{\mathscr{Z}}^\gamma_N$
uniquely follows from Lagrange interpolation. 
For the reader's convenience, let us reproduce the necessary statement which
closely follows \cite[Appendix B]{wheeler2015refined}.

\begin{lemma}
	\label{lemma:Lagrange_IK_uniqueness}
	Let $f_N(u_1,\dots,u_N)$, $N \geq 1$, be symmetric polynomials
	in $(u_1,\dots,u_N)$ of degree at most $N$ in each $u_j$.
	Suppose that they satisfy recurrence
	\begin{align*}
		f_N(u_1,\dots,u_{N-1},t_i) 
		&= 
		C^{(i)}_N
		f_{N-1}(u_1,\dots,u_{N-1}),
		\qquad
		1 \leq i \leq N,
	\end{align*}
	for a suitable set of distinct points $\left\{ t_1,\dots,t_{N} \right\}$,
	where $C^{(i)}_N$ are some coefficients.
	Let also $f_N(u_1^0,\ldots,u_N^0)=C^{(0)}_N$ for some point
	$(u_1^0,\ldots,u_N^0 )$ and some constant $C^{(0)}_N$.
	
	If another family of symmetric polynomials $g_N(u_1,\dots,u_N)$, $N \geq 1$
	satisfies all of these conditions and $f_1(u) = g_1(u)$, then 
	\begin{align*}
		f_N(u_1,\dots,u_N) = g_N(u_1,\dots,u_N), 
		\quad
		\text{for all}\ N \geq 1.
	\end{align*}
\end{lemma}
\begin{proof}
	By induction, since $f_1=g_1$, we
	assume $f_{N-1}(u_1,\ldots,u_{N-1} )=g_{N-1}(u_1,\ldots,u_{N-1} )$.
	Then using the recurrence we see that 
	$f_N(u_1,\ldots,u_{N-1},t_j)=g_N(u_1,\ldots,u_{N-1},t_j)$
	for $N$ distinct points. 
	This, together with the fact that the degree of $f_N$ in $u_N$ is at most $N$,
	implies that $f_N$ and $g_N$ coincide up to some constant factor (independent of $u_N$).
	Moreover, this constant factor cannot depend on $u_1,\ldots,u_{N-1} $ either due to the symmetry
	of the polynomials.
	Finally, since $f_N$ and $g_N$ coincide at a fixed point $(u_1^0,\ldots,u_N^0 )$, 
	we get the claim.
\end{proof}

\subsubsection{Step 3. Verification of properties for the determinantal formula}
\label{subsub:step3_IK}

It remains to check that suitably normalized determinant in the 
right-hand side of \eqref{eq:Z_N_det}
satisfies \Cref{Z_properties}.
This renormalization has the form
\begin{equation}
	\label{eq:Z_N_det_RHS}
	\frac{\prod_{i,j=1}^{N}(1-u_iv_j)(1-qu_iv_j)}{\prod_{1\le i<j\le N}(u_i-u_j)(v_i-v_j)}\,
	\det\left[ \mathsf{z}(u_i,v_j;s_0) \right]_{i,j=1}^{N},
\end{equation}
where $\mathsf{z}(u,v;s_0)$ is given by \eqref{eq:IK_det_element}.
Symmetry (property \textbf{1}) is straightforward: permuting 
any pair $u_i,u_j$ changes the sign of the determinant as well as the 
Vandermonde $\prod_{i<j}(u_i-u_j)$ in the denominator.
The fact that \eqref{eq:Z_N_det_RHS} is a polynomial of degree $N$ 
(property \textbf{2})
follows because we can rewrite
\begin{equation*}
	\eqref{eq:Z_N_det_RHS}=
	\frac{1}{\prod_{1\le i<j\le N}(u_i-u_j)(v_i-v_j)}\,
	\det
	\Bigl[ 
		\mathsf{z}(u_i,v_j;s_0)\prod_{l=1}^{N}(1-u_iv_l)(1-qu_iv_l) 
	\Bigr]_{i,j=1}^{N},
\end{equation*}
and each element of the determinant is a polynomial in $u_i$ of degree $2N-1$.
Dividing by the Vandermonde 
which has degree $N-1$ in each $u_i$ yields degree $N$.

Property \textbf{5} is straightforward.

For \textbf{3}, we multiply the first row of the determinant
by $1-u_1v_1$,
and note that setting $u_1=1/v_1$ eliminates
all elements in the first row except 
\begin{equation*}
	(1-u_1v_1)\,\mathsf{z}(u_1,v_1;s_0)\big\vert_{u_1=v_1^{-1}}=
	(1-\gamma s_0v_1/\xi_0)(1-\gamma s_0\xi_0/v_1).
\end{equation*}
This means that the 
determinant of size $N$ reduces to a similar determinant of size $N-1$
which is equal to $\widetilde{\mathscr{Z}}_{N-1}^{\gamma}$.

Finally, let us get property \textbf{4}. Denote $u_0=(s_0\xi_0\gamma)^{-1}$.
\begin{lemma}
	\label{lemma:u0_det_computation}
	The determinantal expression in \eqref{eq:Z_N_det_RHS},
	specialized at $u_i=u_0 q^{i-1}$, $i=1,\ldots,N $, 
	equals the right-hand side of \eqref{eq:IK_property_4}.
\end{lemma}
\begin{proof}
	Take the $N\times N$ matrix $\mathsf{A}=[\mathsf{z}(u_0 q^{j-1},v_i;s_0)]_{i,j=1}^{N}$
	whose determinant we need to compute.
	We employ the method suggested in \cite[Section 2.6]{krattenthaler1999advanced},
	and consider the LU decomposition $\mathsf{A}=\mathsf{L}\mathsf{U}$, 
	where $\mathsf{L}$ is a lower triangular matrix with ones on the main diagonal, and
	$\mathsf{U}$ is an upper triangular matrix. (The fact that $\mathsf{A}$ 
	is nondegenerate so that to 
	admit an LU decomposition follows from the computations below.)
	Denote 
	\begin{equation}
		\label{eq:f_Q_defs}
		f_k(v) := \prod_{r=1}^{k} (q^r v-\gamma s_0\xi_0),
		\qquad 
		Q_k^j(v):=\frac{f_k(v)}{\prod_{r=1, \, r\ne j}^{k}(v-v_r)}.
	\end{equation}
	We claim that 
	\begin{equation*}
		(\mathsf{L}^{-1})_{ij}=\begin{cases}
			0,&i<j;\\
			1,&i=j;\\
			-Q_{i-1}^j(v_j)/Q_{i-1}^j(v_i),&i>j.
		\end{cases}
	\end{equation*}
	To see this, consider the product $\mathsf{L}^{-1}\mathsf{A}$ with $\mathsf{L}^{-1}$ being the candidate 
	above.
	We would like to show that $\mathsf{L}^{-1}\mathsf{A}$ is upper triangular, that is,
	\begin{equation}
		\label{eq:LU_decomp_proof}
		\begin{split}
			\frac{f_{i-1}(v_i)}{\prod_{r=1}^{i-1}(v_i-v_r)}\,\mathsf{z}(u_0q^{j-1},v_i;s_0)
			&=
			\sum_{k=1}^{i-1}
			\frac{Q_{i-1}^{k}(v_{k})}{v_i-v_{k}}\,\mathsf{z}(u_0 q^{j-1},v_{k};s_0)
			\\&=
			-
			\sum_{k=1}^{i-1}
			\frac{f_{i-1}(v_{k})}{\prod_{r=1,\,r \ne k}^{i}(v_k-v_r)}\,\mathsf{z}(u_0 q^{j-1},v_{k};s_0),
		\end{split}
	\end{equation}
	for all $i>j$
	(here the second equality is just a simplification using \eqref{eq:f_Q_defs}).
	To show \eqref{eq:LU_decomp_proof} we use the Residue Theorem.
	Consider the following function:
	\begin{equation*}
		B(w):=\frac{f_{i-1}(w)}{\prod_{r=1}^{i}(w-v_r)}\,\mathsf{z}(q^{j-1}u_0,w;s_0),\qquad w\in \mathbb{C}.
	\end{equation*}
	It is a rational function with possible poles at $w=v_1,\ldots,v_i $
	and
	$w=q^{1-j}/u_0,q^{-j}/u_0$ coming from $\mathsf{z}$
	\eqref{eq:IK_det_element}.
	However, since $i-1\ge j$, the zeroes coming from 
	$f_{i-1}(w)$ \eqref{eq:f_Q_defs}
	eliminate the two latter poles. 
	Moreover, $B(w)$ is decaying as $O(1/w^2)$ at infinity and thus has zero residue there.
	We see that \eqref{eq:LU_decomp_proof}
	is simply an equality between the residue of $B(w)$ at $w=v_i$ 
	and the sum of minus residues of $B(w)$ at all $w=v_k$, $k=1,\ldots,i-1 $.
	This establishes \eqref{eq:LU_decomp_proof}.

	Having the upper triangularity of $\mathsf{L}^{-1}\mathsf{A}$,
	we can compute the determinant as the product of diagonal elements:
	\begin{equation*}
		\begin{split}
			\det \mathsf{A}=\det (\mathsf{L}^{-1}\mathsf{A})
			&=
			\prod_{i=1}^{N}
			\left(
				\mathsf{z}(u_0q^{i-1},v_{i};s_0)-
				\sum_{k=1}^{i-1}
				\frac{Q_{i-1}^k(v_k)}{Q_{i-1}^{k}(v_i)}\,\mathsf{z}(u_0q^{i-1},v_k;s_0)
			\right)
			\\
			&=
			\frac{\prod_{1\le r<i\le N}(v_i-v_r)}{\prod_{i=1}^{N}f_{i-1}(v_i)}
			\prod_{i=1}^{N}
			\frac{1}{2\pi\mathbf{i}}
			\oint_{c_i}\frac{f_{i-1}(w)\,\mathsf{z}(u_0q^{i-1},w;s_0)}{\prod_{r=1}^{i}(w-v_r)}\,dw,
		\end{split}
	\end{equation*}
	where the contour $c_i$ encircles $v_1,\ldots,v_i $.
	Each $i$-th integrand is regular at infinity and has a
	single pole outside the contour $c_i$.
	This pole comes from $\mathsf{z}$ and is at $w=q^{-i}/u_0$.
	Taking the minus residues at these points gives a product formula
	for the determinant of $\mathsf{A}$. 
	Putting all together, we arrive at the right-hand side of \eqref{eq:IK_property_4}.
\end{proof}

This completes the proof of \Cref{prop:Z_N_det}.
Combining \Cref{prop:Z_N_det,prop:S_as_refined}
we have established \Cref{thm:intro_sHL_refined}.
Note that the parameter $s_0$
in \eqref{eq:Z_N_det} must be replaced by $s_0/\gamma$ to match the refined Cauchy sum, hence 
we get a slightly different determinant
in \Cref{thm:intro_sHL_refined}.

\section{Determinantal identity. Proof of \texorpdfstring{\Cref{thm:intro_det_identity}}{the second main result}}
\label{sec:det_identity_proof}

We will now present an alternative expression for the determinant in the right-hand side of the 
refined Cauchy identity.
Recall the function $\mathsf{z}(u,v;s_0)$ defined by \eqref{eq:IK_det_element},
and also denote
\begin{equation}
	\label{eq:M_Mt_definition}
	\begin{split}
		&
		\mathsf{M}_i(v;\gamma^{-1}s_0):=
		\xi_0^{2i-N}
		v^{N-i-1}
		\\&\hspace{20pt}\times
		\left\{ 
			\left( 1-s_0\xi_0^{-1}v \right)
			\left( v\xi_0^{-1}-s_0 \right)
			\prod_{l=1}^{N}\frac{1-qv u_l}{1-v u_l}
			-
			\gamma^{-1}q^{N-i}
			\left( \gamma-s_0\xi_0^{-1}v \right)
			\left( \gamma q v \xi_{0}^{-1}-s_0 \right)
		\right\};
		\\
		&
		\widetilde{\mathsf{M}}_i(u;\gamma^{-1}s_0)
		:=
		\\&\hspace{20pt}
			u^{N-i-1}
			\left\{
				\left( 1-s_0\xi_0 u \right)
				\left( u-\frac{s_0}{\xi_0} \right)
				\prod_{l=1}^{N}\frac{1-q uv_l}{1-uv_l}
				-
				\gamma^{-1}q^{N-i}
				(\gamma-s_0\xi_0 u)
				\left( \gamma q u-\frac{s_0}{\xi_0} \right)
			\right\}.
	\end{split}
\end{equation}

The following result implies \Cref{thm:intro_det_identity} from the Introduction.
\begin{theorem}
	\label{thm:Cuenca_style_formula}
	We have
	\begin{equation}
		\label{eq:Cuenca_style_formula}
		\begin{split}
			&\frac{\prod_{i,j=1}^{N}(1-qu_iv_j)}{\prod_{1\le i<j\le N}(u_i-u_j)(v_i-v_j)}\,
			\det\left[
				\mathsf{z}(u_i,v_j;\gamma^{-1}s_0)
			\right]_{i,j=1}^{N}
			\\
			&
			\hspace{140pt}=
			\frac{
				\det
				\bigl[ \mathsf{M}_i(v_j;\gamma^{-1}s_0) \bigr]_{i,j=1}^{N}
			}
			{\prod_{1\le i<j\le N}(v_i-v_j)}
			=
			\frac{
				\det
				\bigl[ \widetilde{\mathsf{M}}_i(u_j;\gamma^{-1}s_0) \bigr]_{i,j=1}^{N}
			}
			{\prod_{1\le i<j\le N}(u_i-u_j)}
			.
		\end{split}
	\end{equation}
\end{theorem}

\begin{remark}
	Recall that the right-hand side of the refined Cauchy identity
	in \Cref{thm:intro_sHL_refined}
	is equal to
	\begin{equation*}
		\prod_{j=1}^{N}
		\frac{1}{(1-s_0\xi_0 u_j)(1-\xi_0^{-1}s_0v_j)}\,
		\frac{\prod_{i,j=1}^{N}(1-qu_iv_j)}{\prod_{1\le i<j\le N}(u_i-u_j)(v_i-v_j)}
		\,
		\det\left[
			\mathsf{z}(u_i,v_j;\gamma^{-1}s_0)
		\right]_{i,j=1}^{N}.
	\end{equation*}
	This is the reason for using the parameter $\gamma^{-1}s_0$ in 
	\eqref{eq:M_Mt_definition}--\eqref{eq:Cuenca_style_formula}.
\end{remark}

A feature of the alternative determinantal expressions in
\Cref{thm:Cuenca_style_formula} is that
in them
the dependence on the variables $u_j$ or $v_j$, respectively,
has a product form. 

\medskip

The proof of \Cref{thm:Cuenca_style_formula} occupies the rest of this section.
First, observe that the equality of the two determinantal expressions,
\begin{equation}
	\label{eq:Cuenca_style_proof1}
		\frac{
			\det
			\bigl[ \mathsf{M}_i(v_j;\gamma^{-1}s_0) \bigr]_{i,j=1}^{N}
		}
		{\prod_{1\le i<j\le N}(v_i-v_j)}
		=
		\frac{
			\det
			\bigl[ \widetilde{\mathsf{M}}_i(u_j;\gamma^{-1}s_0) \bigr]_{i,j=1}^{N}
		}
		{\prod_{1\le i<j\le N}(u_i-u_j)},
\end{equation}
readily follows from the symmetry of the function $\mathsf{z}$:
\begin{equation*}
	\mathsf{z}(u,v;s_0)
	=
	\mathsf{z}(v\xi_0^{-2},u\xi_0^{2};s_0).
\end{equation*}
Therefore, in the proof we can freely pass between the two expressions
in \eqref{eq:Cuenca_style_proof1},
depending on convenience.

Arguing as in the proof of
\Cref{prop:Z_N_det} from \Cref{subsub:step3_IK}, it suffices to check that the following
function
\begin{equation}
	\label{eq:M_det}
	\mathscr{M}_N^\gamma(u_1,\ldots,u_N\mid v_1,\ldots,v_N ;s_0 )=
	\prod_{i,j=1}^{N}(1-u_iv_j)\,
		\frac{
			\det
			\bigl[ \mathsf{M}_i(v_j;s_0) \bigr]_{i,j=1}^{N}
		}
		{\prod_{1\le i<j\le N}(v_i-v_j)}
\end{equation}
satisfies \Cref{Z_properties}.
Here we are using the normalization from \eqref{eq:Z_N_det_RHS},
and also have replaced the parameter $\gamma^{-1}s_0$
by $s_0$ to directly match the properties.

Properties \textbf{1} (symmetry in $u_i$ and $v_j$ separately) and \textbf{5} (evaluation for $N=1$) are straightforward
from \eqref{eq:M_det}.
For \textbf{2}, observe that 
$\mathsf{M}_i(v;s_0)\prod_{l=1}^{N}(1-v u_l)$ is linear in each $u_j$ separately, 
which implies that
$\mathscr{M}_N^\gamma$
is a polynomial in each $u_j$ of degree $N$.
Using the second determinant in \eqref{eq:Cuenca_style_proof1}
shows polynomialily in each $v_j$, too.

Let us now establish property \textbf{3}:
\begin{lemma}
	\label{lemma:M_det_3}
	Setting $u_N=v_N^{-1}$, we have the recurrence
	\begin{equation*}
		\begin{split}
			&\mathscr{M}^\gamma_N(u_1,u_2,\ldots,u_N\mid v_1,v_2,\ldots,v_N;s_0  )
			\Big\vert_{u_N=v_N^{-1}}
			\\&\hspace{60pt}=
			(1-q)
			(1-s_0\xi_0 \gamma v_N^{-1})
			(1-s_0\xi_0^{-1}\gamma v_N)
			\prod_{j=2}^{N}
			\left( 1-qv_N^{-1}v_j \right)\left( 1-qv_Nu_j \right)
			\\&\hspace{120pt}\times
			\mathscr{M}^\gamma_{N-1}(u_1,\ldots,u_{N-1}\mid v_1,\ldots,v_{N-1};s_0  ).
		\end{split}
	\end{equation*}
\end{lemma}
\begin{proof}
	Note that \eqref{eq:IK_property_4} has the substitution $u_1=v_1^{-1}$, but due to symmetry this is the same property.
	To establish the claim, observe that the last column of the matrix 
	\begin{equation}
		\label{eq:Cuenca_style_proof4}
		\Bigl[\mathsf{M}_i(v_j;s_0)\prod_{l=1}^{N}(1-v_ju_l)\Bigr]_{i,j=1}^{N}
	\end{equation}
	simplifies after substituting $u_N=v_N^{-1}$, and its $i$-th element becomes
	\begin{equation}
		\label{eq:Cuenca_style_proof4_5}
		\xi_0^{2i-N-1}v_N^{N-i}\cdot
		\underbrace{(1-q)(1-\gamma s_0v_N/\xi_0)(1-\xi_0\gamma s_0 v_N^{-1})\prod_{l=2}^{N}(1-q u_l/u_N)}_{\text{factors out}}.
	\end{equation}
	Taking out the factors indicated above (note that they match the right-hand side of the claim), 
	perform the following row operations to the matrix. 
	Subtract each $i$-th row multiplied by $v_N/\xi_0^2$ from the $(i-1)$-st row,
	$i=2,\ldots,N $.
	These operations do not change the determinant, 
	but make the matrix block-diagonal, with zeroes in the last column
	everywhere except the main diagonal. The $(N,N)$-th entry of the matrix
	is given by \eqref{eq:Cuenca_style_proof4_5} with $i=N$.

	After these transformations, one readily sees that the submatrix of
	\eqref{eq:Cuenca_style_proof4} formed by the first $N-1$ rows and columns
	has determinant proportional to $\mathscr{M}^{\gamma}_{N-1}$. Indeed, the matrix elements 
	after the transformations are equal to 
	\begin{equation*}
		-\frac{(v_N-v_j)(v_N-q v_j)}{\xi_0 v_N}
		\,\mathsf{M}_i^{(N-1)}(v_j;s_0)\prod_{l=1}^{N-1}(1-v_ju_l),\qquad 
		i,j=1,\ldots,N-1,
	\end{equation*}
	where 
	by $\mathsf{M}_i^{(N-1)}$ we mean the matrix element involved in the function $\mathscr{M}^{\gamma}_{N-1}$
	of the smaller rank.
	Taking into account 
	the Vandermonde in the $v_j$'s in the denominator, we see that
	this implies the claim.
\end{proof}

The proof of property \textbf{4} requires to compute a nontrivial determinant:
\begin{lemma}
	\label{lemma:M_det_4}
	We have
	\begin{equation*}
		\begin{split}
			&\mathscr{M}^\gamma_N\left(  
				(s_0\xi_0 \gamma)^{-1},q(s_0\xi_0\gamma)^{-1},\ldots,q^{N-1}(s_0\xi_0\gamma)^{-1}
				\mid v_1,\ldots,v_N;s_0 
			\right)
			\\&\hspace{120pt}=
			q^{N^2}
			\prod_{i,j=1}^{N}\left( 1-v_i\frac{q^{j-1}}{s_0\xi_0\gamma} \right)
			\prod_{j=1}^{N}(1-\gamma q^{-j+1})(1-s_0^2 \gamma q^{-j}).
		\end{split}
	\end{equation*}
\end{lemma}
\begin{proof}
	Substituting $u_j=q^{j-1}/(s_0\xi_0 \gamma)$, we have
	\begin{equation*}
		\prod_{l=1}^{N}\frac{1-q v u_l}{1-vu_l}=
		\frac{s_0\xi_0\gamma-q^{N}v}{s_0\xi_0\gamma-v},
	\end{equation*}
	which simplifies the matrix elements $\mathsf{M}_i$ as follows:
	\begin{equation*}
		\mathsf{M}_i(v;s_0)=
		\xi_0^{2i-N}
		v^{N-i-1}
		\left\{ 
			\left( 1-\gamma s_0\xi_0^{-1}v \right)
			\left(q^{N}\xi_0^{-1}v - s_0\gamma \right)
			-
			\gamma q^{N-i}
			\left( 1-s_0\xi_0^{-1}v \right)
			\left( q v \xi_{0}^{-1}-s_0 \right)
		\right\}.
	\end{equation*}
	We see that $\mathsf{M}_i(v;s_0)$ is a polynomial in $v$ which has two or three nonzero terms, and 
	\begin{equation*}
		\deg \mathsf{M}_i(v;s_0)=\begin{cases}
			N-1,&i=1;\\
			N+1-i,&2\le i\le N.
		\end{cases}
	\end{equation*}
	This suggests the following representation of the matrix:
	\begin{equation*}
		\left[ \mathsf{M}_i(v_j;s_0) \right]_{i,j=1}^{N}
		=
		\left[ 
			\begin{matrix}
			a_1  & b_1     &       &  0 \\
			c_1  & a_2     &  b_2      & \\
			0& \ddots & \ddots &   \\  & \ddots & \ddots & b_{N-1} \\
			0  &        & c_{N-1}     & a_N 
		\end{matrix}
		\right]
		\cdot
		\left[ v_j^{N-i} \right]_{i,j=1}^{N},
	\end{equation*}
	where the elements of the tridiagonal matrix are given by
	\begin{equation}
		\begin{split}
			a_i&=\xi ^{2 i-N-1} 
			\left(
				q^N(1-\gamma q^{1-i})
				+s_0^2 \gamma\left( \gamma-q^{N-i} \right)
			\right);\\
			b_i&=\gamma s_0 \xi_0^{2i-N}\left( q^{N-i}-1 \right)
			;\\
			c_i&=\gamma s_0\xi_0^{2i-N}q^{N-i}(1-q^{i})
			.
		\end{split}
		\label{eq:abc_tridiagonal}
	\end{equation}
	It now remains to check that the determinant of this tridiagonal matrix is equal to 
	\begin{equation}
		\label{eq:abc_tridiagonal_det}
		q^{N^2}
		\prod_{j=1}^{N}(1-\gamma q^{-j+1})(1-s_0^2 \gamma q^{-j}).
	\end{equation}
	First, observe that the powers of $\xi_0$ in \eqref{eq:abc_tridiagonal}
	can be omitted by taking a conjugation of our tridiagonal matrix by the antidiagonal matrix
	with the entries $\xi_0^i$ on the side diagonal. In what follows we thus assume that $\xi_0=1$.
	Let us also multiply $b_i$ by $s_0$ and divide $c_i$ by $s_0$, this does not change the 
	determinant. Therefore, now we have (by reusing the same notation)
	\begin{equation}
		\label{eq:abc_tridiagonal_2}
		a_i=
			q^N(1-\gamma q^{1-i})
			+s_0^2 \gamma\left( \gamma-q^{N-i} \right)
			,\qquad 
		b_i=\gamma s_0^2 \left( q^{N-i}-1 \right)
		,\qquad 
		c_i=q^{N-i}(1-q^{i}),
	\end{equation}
	and need to compute the $N\times N$ tridiagonal determinant formed by these quantities.
	Denote by $T_N$ the tridiagonal matrix corresponding to the entries
	\eqref{eq:abc_tridiagonal_2}.
	
	The form of the coefficients \eqref{eq:abc_tridiagonal_2} suggests that the 
	eigenvectors of 
	$T_N$ can be matched to some $q$-hypergeometric orthogonal 
	polynomials on the finite lattice $\left\{ 1,2,\ldots,N  \right\}$. 
	The right family of orthogonal polynomials
	can be guessed by looking at 
	the eigenvectors of, say, $T_5$, and 
	computing an orthogonality measure on $\left\{ 1,2,\ldots,N  \right\}$ for them with a computer algebra 
	system. One sees that 
	this orthogonality measure is the $q$-deformed binomial distribution, and so the orthogonal polynomials
	are the \emph{$q$-Krawtchouk} ones. See, e.g., \cite[Section 3.15]{Koekoek1996}, for their definition.

	We would not make direct use of the definition or properties of the $q$-Krawtchouk polynomials.
	Instead, define
	\begin{align*}
		F_i(k)
		&:={}_3\phi_2
		\bigl( 
			q^{i-N},q^{k-N},s_0^2 q^{-i};q^{1-N},0;q,q
		\bigr)
		\\&=
		\sum_{r=0}^{N-1}
		\frac{(q^{i-N};q)_r(q^{k-N};q)_r(s_0^2q^{-i};q)_r}
		{(q^{1-N};q)_r}\,
		\frac{q^r}{(q;q)_r},
	\end{align*}
	where $_3\phi_2$ is the $q$-hypergeometric series, and in the second line we have explicitly expanded its definition.
	Note that the series it terminating.
	
	With the notation \eqref{eq:abc_tridiagonal_2}, we have
	\begin{equation}
		\label{eq:abc_tridiagonal_eigenvalues}
		a_k F_i(k)+b_k F_i(k+1)+c_{k-1}F_i(k-1)=
		q^N(1-\gamma q^{i-N})(1-\gamma s_0^2 q^{-i})F_i(k)
	\end{equation}
	for all $i,k=1,\ldots,N $.
	Indeed, this follows
	from term-by-term manipulations with the terminating $q$-hypergeometric series.
	These manipulations are straightforward and we omit them.
	
	Identity \eqref{eq:abc_tridiagonal_eigenvalues} implies that $F_i(\cdot)$ are eigenvectors of $T_N$.
	Therefore, the
	desired determinant \eqref{eq:abc_tridiagonal_det}
	is the product of the eigenvalues of the $F_i(\cdot)$'s
	which appear in the right-hand side of 
	\eqref{eq:abc_tridiagonal_eigenvalues}.
	This completes the proof.
\end{proof}

The last statement completes checking of 
\Cref{Z_properties}
for the function
$\mathscr{M}^\gamma_N$
\eqref{eq:M_det},
and (combined with \Cref{lemma:Lagrange_IK_uniqueness})
establishes \Cref{thm:Cuenca_style_formula}.
The latter implies \Cref{thm:intro_det_identity} from the Introduction.

\section{Degeneration to interpolation Hall--Littlewood polynomials}
\label{sec:interp_HL}

In this section we specialize our results 
to interpolation Hall--Littlewood polynomials.
In \Cref{sub:interp_Mac_definition,sub:homog_Macd_polys,sub:HL_deg_of_interpolation,sub:refined_Cauchy_interp_HL}
we recall the definition of the interpolation Macdonald polynomials, 
their Hall--Littlewood degeneration, and the refined Cauchy
identity from \cite{cuenca2018interpolation}. 
Then in \Cref{sub:intHL_degen_all} 
we explain how to get the same results by specializing our spin Hall--Littlewood
statements.
Finally, in \Cref{sub:Olshanski_from_us}
we discuss another Cauchy type identity \cite[Proposition 9.4]{olshanski2019interpolation}
for interpolation Hall--Littlewood polynomials.

\subsection{Interpolation Macdonald polynomials}
\label{sub:interp_Mac_definition}

Let us recall interpolation Macdonald polynomials from
\cite{knop1996symmetric}, \cite{okounkov1997binomial},
\cite{sahi1996interpolation}.
We mostly follow the notation of 
\cite{olshanski2019interpolation},
\cite{cuenca2018interpolation}
(and also \cite{Macdonald1995} when talking about homogeneous polynomials).

Fix the number of variables $N$ and two parameters $\qm,\tm\in(0,1)$
(we use different font so that these Macdonald parameters are not confused with our 
main quantization parameter $q$).
For every $\lambda \in \mathrm{Sign}_N$
there exists a unique (up to a scalar factor)
symmetric polynomial $I_{\lambda}(u_1,\ldots,u_N; \qm,\tm )$ (depending on the parameters $\qm,\tm$)
such that \cite{Sah1994spectrum}, \cite{okounkov_newton_int}:
\begin{itemize}
	\item The degree of $I_{\lambda}$ is $|\lambda|$;
	\item For all $\mu\in \mathrm{Sign}_N$ with $\mu\ne \lambda$, $|\mu|\le |\lambda|$, 
		we have
		$I_{\lambda}(\qm^{-\mu_1},\qm^{-\mu_2}t,\ldots,\qm^{-\mu_N}t^{N-1} ;\qm,\tm)=0$;
	\item 
		We have $I_{\lambda}(\qm^{-\lambda_1},\qm^{-\lambda_2}\tm,\ldots,\qm^{-\lambda_N}\tm^{N-1} ;\qm,\tm)\ne 0$.
\end{itemize}
We fix normalization so that the 
top component in $I_{\lambda}(u_1,\ldots,u_N;\qm,\tm )$
with respect to the lexicographic ordering is equal to $u_1^{\lambda_1}\ldots u_N^{\lambda_N} $.

\subsection{Homogeneous Macdonald polynomials}
\label{sub:homog_Macd_polys}

The top, degree $|\lambda|$, homogeneous component of $I_{\lambda}(u_1,\ldots,u_N;\qm,\tm )$ is the 
symmetric Macdonald polynomial $P_{\lambda}(u_1,\ldots,u_N;\qm,\tm )$
\cite[Ch VI.4]{Macdonald1995}. 
We recall the Cauchy identity for Macdonald polynomials
which holds when $|u_iv_j|<1$ for all $i,j=1,\ldots,N $:
\begin{equation}
	\label{eq:Macdonald_Cauchy}
	\sum_{\lambda\in \mathrm{Sign}_N}
	P_\lambda(u_1,\ldots,u_N;\qm,\tm )\,
	Q_\lambda(v_1,\ldots,v_N;\qm,\tm )=
	\prod_{i,j=1}^{N}\frac{(\tm u_iv_j;\qm)_{\infty}}{(u_iv_j;\qm)_{\infty}}.
\end{equation}
Here $Q_\lambda(\cdot\,;\qm,\tm)=b_\lambda(\qm,\tm)P_\lambda(\cdot\,;\qm,\tm)$ 
are the dual Macdonald polynomials which are proportional to the original ones.
The coefficients $b_\lambda$ are explicit \cite[VI.(6.19)]{Macdonald1995},
but we will need an expression for them only in a particular case.
We refer to 
\cite[Ch. VI]{Macdonald1995}
for alternative characterizations and more properties of Macdonald polynomials. 

\subsection{Hall--Littlewood degeneration of interpolation Macdonald polynomials}
\label{sub:HL_deg_of_interpolation}

When $\qm=0$, the Macdonald polynomials become the Hall--Littlewood symmetric polynomials
(depending on $\tm$)
which we denote by 
$P_\lambda^{HL}=P_\lambda^{HL}(\cdot;\tm)$,
$Q_\lambda^{HL}=Q_\lambda^{HL}(\cdot;\tm)$.
We have the following explicit symmetrization formula:
\begin{equation}
	\label{eq:Q_HL_polynomial}
	\begin{split}
		Q^{HL}_\lambda(u_1,\ldots,u_N;\tm )&=
		b_\lambda(0,\tm)\,P^{HL}_\lambda(u_1,\ldots,u_N;\tm )\\&=
		\frac{(1-\tm)^N}{(\tm;\tm)_{N-\ell(\lambda)}}
		\sum_{\sigma\in S_N}
		\sigma\left( \prod_{1\le i<j\le N}\frac{u_i-\tm u_j}{u_i-u_j}\prod_{i=1}^{N}v_i^{\lambda_i} \right),
	\end{split}
\end{equation}
and $b_\lambda(0,\tm)=\prod_{r\ge1}(\tm;\tm)_{m_r(\lambda)}$.

As shown in 
\cite[Theorem 5.11]{cuenca2018interpolation},
under a suitable
Hall--Littlewood degeneration the Macdonald
interpolation polynomials $I_\lambda$ also take an explicit form.
Denote
\begin{equation}
	\label{eq:F_HL_from_I_Mac}
	F_\lambda^{HL}(u_1,\ldots,u_N;\tm ):=\lim_{\qm\to0}
	I_\lambda(u_1,\ldots,u_N;1/\qm,1/\tm ),
\end{equation}
then
\begin{equation}
	\label{eq:F_HL_interp}
	F_\lambda^{HL}(u_1,\ldots,u_N;\tm )=
	(1-\tm)^N\prod_{r\ge0}\frac1{(\tm;\tm)_{m_r(\lambda)}}
	\sum_{\sigma\in S_N}
	\sigma\left( 
		\prod_{i<j}\frac{u_i-\tm u_j}{u_i-u_j}\prod_{i=1}^{\ell(\lambda)}
		u_i^{\lambda_i}(1 - \tm^{1-N}u_i^{-1})
	\right).
\end{equation}
The top degree homogeneous component in $F_\lambda^{HL}$ is the 
Hall--Littlewood polynomial $P_\lambda^{HL}$.
While the interpolation property 
formulated in \Cref{sub:interp_Mac_definition}
does not determine the polynomials 
uniquely in the degeneration \eqref{eq:F_HL_from_I_Mac},
it is still convenient to 
refer to $F_\lambda^{HL}$ \eqref{eq:F_HL_interp} as 
the \emph{interpolation Hall--Littlewood polynomials}.

\begin{remark}
	The homogeneous Macdonald polynomials 
	$P_\lambda(\cdot\,;\qm,\tm)$
	are invariant under the change $(\qm,\tm)\to(1/\qm,1/\tm)$
	\cite[VI.(4.14)(iv)]{Macdonald1995} while the interpolation Macdonald polynomials are not.
	To perform the Hall--Littlewood degeneration in 
	\eqref{eq:F_HL_from_I_Mac} one needs the inverse parameters $(1/\qm,1/\tm)$.
\end{remark}

\subsection{Refined Cauchy identity for interpolation Hall--Littlewood polynomials}
\label{sub:refined_Cauchy_interp_HL}

The following refined Cauchy identity holds
for the interpolation Hall--Littlewood polynomials.
If $|u_iv_j|<1$ for all $i,j=1,\ldots,N $, then
\cite[Proposition 4.2]{cuenca2018interpolation}
\begin{equation}
	\label{eq:Cuenca_refined_Cauchy}
	\begin{split}
		&\sum_{\lambda\in \mathrm{Sign}_N}
		(\chi \tm^{1-N};\tm)_{N-\ell(\lambda)}\,
		F_\lambda^{HL}(u_1,\ldots,u_N ;\tm)\,Q_\lambda^{HL}(v_1,\ldots,v_N ;\tm)
		\\&\hspace{6pt}=
		\frac1
		{\prod_{1\le i<j\le N}(u_i-u_j)}
		\det
		\left[ 
			u_j^{N-i-1}
			\left\{
				\left( u_j-\tm^{1-N} \right)
				\prod_{l=1}^{N}\frac{1-\tm u_jv_l}{1-u_jv_l}
				+
				\tm^{1-i}
				\left( 1-\chi u_j\right)
			\right\}
		\right]_{i,j=1}^N.
	\end{split}
\end{equation}
This identity is proven in \cite{cuenca2018interpolation} by using the
symmetrization formulas 
\eqref{eq:Q_HL_polynomial}, \eqref{eq:F_HL_interp}
for the functions in the left-hand side of \eqref{eq:Cuenca_refined_Cauchy}, 
and employing direct manipulations with the summands to produce the determinantal expression.

\begin{remark}
	\label{rmk:symm_function_algebra_remark_for_refined_Cauchy}
	The statement \cite[Proposition 4.2]{cuenca2018interpolation}
	is more general in that the number of the variables
	$v_j$ is allowed to be arbitrary, not necessarily the same $N$ as the 
	number of the $u_i$ variables.
	However, both sides of \eqref{eq:Cuenca_refined_Cauchy}
	are symmetric polynomials in $v_j$ which satisfy the stability property
	of the form $f_N(v_1,\ldots,v_{N-1},v_N )\vert_{v_N=0}=
	f_{N-1}(v_1,\ldots,v_{N-1} )$.
	This means that identity with $N$ variables extends to an identity between
	elements of the algebra of symmetric functions \cite[Ch. I.2]{Macdonald1995},
	so that the number of the variables $v_j$ can be arbitrary.

	This extension is a feature of the type of the determinant in the right-hand side
	of \eqref{eq:Cuenca_refined_Cauchy}, and is not directly possible for the
	Izergin--Korepin form of the determinant which we present 
	in \Cref{prop:upgrade_of_Cuenca_det_identity} below.
\end{remark}

As noted in \cite[Section 4.5]{cuenca2018interpolation}, by taking the 
top degree components in \eqref{eq:Cuenca_refined_Cauchy}
and using the refined Cauchy identity for the 
Hall--Littlewood polynomials 
\cite{warnaar2008bisymmetric},
\cite[Theorem 4]{wheeler2015refined},
one arrives at
the following nontrivial determinantal identity:
\begin{equation}
	\label{eq:WZJ_Cuenca_det_identity}
	\begin{split}
		&\frac{\prod_{i,j=1}^{N}(1-\tm u_iv_j)}{\prod_{1\le i<j\le N}(u_i-u_j)(v_i-v_j)}\,
		\det\left[
			\frac{1-\chi \tm^{1-N}+(\chi \tm^{1-N}-\tm)u_iv_j }{(1-u_iv_j)(1-\tm u_iv_j)}
		\right]_{i,j=1}^{N}
		\\
		&
		\hspace{80pt}=
		\frac{1}
		{\prod_{1\le i<j\le N}(u_i-u_j)}
		\,\det
		\left[ u_j^{N-i}
		\left\{ \prod_{l=1}^{N}\frac{1-\tm u_jv_l}{1-u_jv_l}-\tm^{1-i}\chi \right\}
		\right]_{i,j=1}^{N}
		.
	\end{split}
\end{equation}
In the next subsection
we use this determinantal identity
as the key to understand how our results (\Cref{thm:intro_sHL_refined,thm:intro_det_identity})
degenerate to \eqref{eq:Cuenca_refined_Cauchy}, \eqref{eq:WZJ_Cuenca_det_identity}.
In the process we also observe how the interpolation Hall--Littlewood polynomials
arise as a particular case of our inhomogeneous spin Hall--Littlewood rational functions.

\subsection{From spin Hall--Littlewood to interpolation Hall--Littlewood}
\label{sub:intHL_degen_all}

Observe that identity \eqref{eq:WZJ_Cuenca_det_identity}
is a particular case of \Cref{thm:intro_det_identity}
when $s_0=0$, $\xi_0$ is arbitrary (it vanishes from the formula when $s_0=0$),
$q=\tm$,
and $\gamma=q^{-N}\chi$.
This suggests the following generalization of the determinantal identity
\eqref{eq:WZJ_Cuenca_det_identity} which
incorporates the right-hand side of
\eqref{eq:Cuenca_refined_Cauchy}:

\begin{proposition}
	\label{prop:upgrade_of_Cuenca_det_identity}
	We have
	\begin{equation}
		\label{eq:Cuenca_upgrade}
		\begin{split}
			&
			\frac1
			{\prod_{1\le i<j\le N}(u_i-u_j)}
			\det
			\left[ 
				u_j^{N-i-1}
				\left\{
					\left( u_j-\tm^{1-N} \right)
					\prod_{l=1}^{N}\frac{1-\tm u_jv_l}{1-u_jv_l}
					+
					\tm^{1-i}
					\left( 1-\chi u_j\right)
				\right\}
			\right]_{i,j=1}^N
			\\
			&\hspace{30pt}
			=
			\frac{\prod_{i,j=1}^{N}(1-\tm u_iv_j)}{\prod_{1\le i<j\le N}(u_i-u_j)(v_i-v_j)}\,
			\det\left[
				\frac{(1-\tm ^{-N}\chi)\tm(1-u_iv_j)+(1-\tm)(1-\tm^{1-N}v_j)}
				{(1-u_iv_j)(1-\tm u_iv_j)}
			\right]_{i,j=1}^{N}
			.
		\end{split}
	\end{equation}
\end{proposition}
\begin{proof}
	Start with identity \eqref{eq:intro_det_identity}
	from \Cref{thm:intro_det_identity}. Rename the parameter $q$ by $\tm$.
	Then change the parameters $s_0,\xi_0$
	to their combinations
	$s_0\xi_0$ and $s_0/\xi_0$.
	In particular, $s_0^2=(s_0\xi_0)(s_0/\xi_0)$.
	After this, specialize 
	\begin{equation}
		\label{eq:a_b_chi_specialization}
		s_0\xi_0\to 0,\qquad 
		s_0/\xi_0\to \tm^{1-N},\qquad 
		\gamma=\tm^{-N}\chi.
	\end{equation}
	The resulting determinantal identity is equivalent to
	\eqref{eq:Cuenca_upgrade}.
\end{proof}

Let us now look at the spin Hall--Littlewood functions 
under the same degeneration as in the proof of 
\Cref{prop:upgrade_of_Cuenca_det_identity}.

\begin{proposition}
	\label{prop:sHL_to_iHL}
	Fix $\lambda\in \mathrm{Sign}_N$.	
	When we rename $q$ to $\tm$, set $s_x=0$, $\xi_x=1$ for all $x\ge1$, and 
	specialize $s_0$ and $\xi_0$
	as in 
	\eqref{eq:a_b_chi_specialization}, we have
	\begin{equation*}
		\begin{split}
			&
			\xi_0^{-\ell(\lambda)}
			\mathsf{F}_\lambda(u_1,\ldots,u_N )\to 
			\prod_{r\ge0}(\tm;\tm)_{m_r(\lambda)}\cdot
			F_\lambda^{HL}(u_1,\ldots,u_N;\tm )
			;
			\\
			&
			\xi_0^{\ell(\lambda)}
			\mathsf{F}^*_\lambda(v_1,\ldots,v_N )\to 
			\prod_{j=1}^{N}\frac{1}{1-v_j\tm^{1-N}}
			\prod_{r\ge 1}\frac1{(\tm;\tm)_{m_r(\lambda)}}\cdot
			Q_\lambda^{HL}(v_1,\ldots,v_N;\tm ).
		\end{split}
	\end{equation*}
\end{proposition}
\begin{proof}
	This follows from explicit symmetrization formulas
	\eqref{eq:F_symmetrization},
	\eqref{eq:F_star_explicit},
	\eqref{eq:F_HL_interp},
	\eqref{eq:Q_HL_polynomial}
	for 
	the functions $\mathsf{F}_\lambda,\mathsf{F}_\lambda^*,F_\lambda^{HL}$,
	and $Q_\lambda^{HL}$, respectively.
	Namely, for the first claim observe using notation \eqref{eq:phi_notation}:
	\begin{equation*}
		\begin{split}
			\varphi_0(u)&=\frac{1-q}{1-s_0\xi_0 u}\to 1-\tm;\\
			\xi_0^{-1}\varphi_k(u)&=\frac{1-q}{1-s_k\xi_k u}
			\frac{u- s_0/\xi_0}{1-s_0\xi_0u}
			\prod_{j=1}^{k-1}\frac{\xi_j u-s_j}{1-s_j\xi_j u}
			\to
			(1-\tm)u^{k-1}(u-\tm^{1-N}),\qquad k>0.
		\end{split}
	\end{equation*}
	For the second claim we have
	\begin{equation*}
		\begin{split}
			\varphi_0(v)\Big\vert_{\text{$\xi_x\to \xi_x^{-1}$}}
			&=\frac{1-q}{1-v s_0/\xi_0}\to \frac{1-\tm}{1-v\tm^{1-N}};\\
			\xi_0\varphi_k(v)
			\Big\vert_{\text{$\xi_x\to \xi_x^{-1}$}}
			&=\frac{1-q}{1-vs_k/\xi_k}
			\frac{v-s_0 \xi_0}{1-vs_0/\xi_0}
			\prod_{j=1}^{k-1}\frac{v/\xi_j-s_j}{1-vs_j/\xi_j}
			\to
			\frac{(1-\tm)v^{k}}
			{1-v\tm^{1-N}}
			,\qquad k>0.
		\end{split}
	\end{equation*}
	It remains to put together all the prefactors, and we are done.
\end{proof}

The degeneration in \Cref{prop:sHL_to_iHL}
produces an alternative proof of the refined Cauchy identity 
for the interpolation Hall--Littlewood polynomials:
\begin{proof}[New proof of {\cite[Proposition 4.2]{cuenca2018interpolation}}]
	As explained in \Cref{rmk:symm_function_algebra_remark_for_refined_Cauchy}, 
	we may assume that the number of the variables $v_j$ is the same $N$
	as the number of the $u_i$ variables. That is, we need to establish \eqref{eq:Cuenca_refined_Cauchy}.
	Starting with our refined Cauchy identity
	\eqref{eq:intro_sHL_refined} from \Cref{thm:intro_sHL_refined}
	and specializing the parameters as 
	\begin{equation}
		\label{eq:ab_spec_in_functions}
		q=\tm,\qquad
		\textnormal{$s_x=0$, $\xi_x=1$ for all $x\ge1$},\qquad 
		s_0\xi_0\to 0,\qquad 
		s_0/\xi_0\to \tm^{1-N},\qquad 
		\gamma=\tm^{-N}\chi,
	\end{equation}
	we obtain for its terms using \Cref{prop:sHL_to_iHL}:
	\begin{multline*}
		\frac{(\gamma q;q)_{m_0(\lambda)}(\gamma^{-1}s_0^2;q)_{m_0(\lambda)}}{(q;q)_{m_0(\lambda)}(s_0^2;q)_{m_0(\lambda)}}
		\,
		\mathsf{F}_{\lambda}(u_1,\ldots,u_N )\,
		\mathsf{F}^*_{\lambda}(v_1,\ldots,v_N )
		\\
		\to
		\frac{(\tm^{1-N}\chi;\tm)_{m_0(\lambda)}}{\prod_{j=1}^{N}(1-v_j\tm^{1-N})}
		\,
		F_\lambda^{HL}(u_1,\ldots,u_N ;\tm)\,Q_\lambda^{HL}(v_1,\ldots,v_N ;\tm).
	\end{multline*}
	Specializing
	the right-hand side of \eqref{eq:intro_sHL_refined}
	in the same way and rewriting the Izergin--Korepin type determinant
	using \eqref{eq:Cuenca_upgrade} (which follows from \Cref{thm:intro_det_identity}),
	we arrive
	at the desired statement.
\end{proof}

The vanishing property definition of the interpolation
Macdonald polynomials $I_\lambda$ recalled in \Cref{sub:interp_Mac_definition}
implies some vanishing of their Hall--Littlewood degeneration
$F_\lambda^{HL}$~\eqref{eq:F_HL_from_I_Mac}.
Namely, one can check that
$F_\lambda^{HL}$ mush vanish at
all points 
\begin{equation*}
	X_m:=(0,0,\ldots,0,\tm^{-m},\tm^{-m-1},\ldots,\tm^{1-N}),
	\qquad 
	0\le m\le \ell(\lambda)+\#\{i\colon \lambda_i\ge2\}.
\end{equation*}
Note that 
unlike in the Macdonald case,
these vanishing properties are not sufficient to uniquely determine $F_\lambda^{HL}$.

One might ask whether the vanishing $F_\lambda^{HL}(X_m;\tm)=0$
could be deduced from the vertex model
interpretation of $F_\lambda^{HL}$
following from \Cref{prop:sHL_to_iHL}.
Here let us suggest a possible argument without working out all the details. 
Fix $N,\lambda$, and let $\ell=\ell(\lambda)$,
$k=\#\{i\colon \lambda_i\ge2\}$ (clearly, $\ell\ge k$).
In each up-right path ensemble
contributing to the partition function $F_\lambda^{HL}$,
$\ell$ horizontal arrows must exit the zeroth column, 
and $k$ horizontal arrows must exit the first column.
Therefore, the partition function $F_\lambda^{HL}$ can be refined as follows:
\begin{equation}
	\label{eq:F_vanishing_to_conj}
	F_\lambda^{HL}(u_1,\ldots,u_N;\tm)=
	\sum_{I,J}W_{I,J}\prod_{i\in I}(u_i-\tm^{1-N+I^c_{<i}})
	\prod_{j\in J}u_j,
\end{equation}
where $I,J\subseteq\{1,\ldots,N \}$,
$|I|=\ell$,
$|J|=k$,
$I^c_{<i}=\#\{r\colon r\notin I,\ r<i\}$,
and the sets $I,J$ should satisfy ordering:
\begin{equation*}
	\textnormal{If }I=\{i_1<\ldots<i_\ell \},
	\
	J=\{j_1<\ldots<j_k \},
	\ \textnormal{then}\ 
	i_1\le j_1,\ 
	\ldots
	,\ 
	i_k\le j_k.
\end{equation*}
Namely, $I$ and $J$ record the locations where horizontal paths exit the zeroth
and the first columns, respectively.
Finally, $W_{I,J}$ are the refined partition functions corresponding
to summing over all possible path configurations beyond 
the first column. 

For example, when $N=3$, $\ell=2$, and $k=1$, we get
\begin{equation}
	\begin{split}
		F_{\lambda}^{HL}(u_1,u_2,u_3;\tm)
		&=
		(u_1-\tm^{-2})(u_2-\tm^{-2})
		\left[ u_1 W_{ \left\{ 1,2 \right\},\{1\} }
			+ 
			u_2 W_{ \left\{ 1,2 \right\},\{2\} }+
			u_3 W_{ \left\{ 1,2 \right\},\{3\} }
		\right]
		\\
		&\hspace{20pt}+
		(u_1-\tm^{-2})(u_3-\tm^{-1})
		\left[ u_1 W_{ \left\{ 1,3 \right\},\{1\} }
			+ 
			u_2 W_{ \left\{ 1,3 \right\},\{2\} }+
			u_3 W_{ \left\{ 1,3 \right\},\{3\} }
		\right]
		\\
		&\hspace{20pt}+
		(u_2-\tm^{-1})(u_3-\tm^{-1})
		\left[ 
			u_2 W_{ \left\{ 2,3 \right\},\{2\} }+
			u_3 W_{ \left\{ 3,3 \right\},\{3\} }
		\right]
	\end{split}
	\label{eq:F_vanishing_to_conj2}
\end{equation}

The following conjecture would ensure the desired vanishing
due to the symmetry of the function $F_\lambda^{HL}$
in the $u_j$'s:
\begin{conjecture}
	For each $m=0,1,\ldots,\ell+k $, there exists a permutation
	of the entries of $X_m$ 
	such that assigning $(u_1,\ldots,u_N )$ to this permutation of 
	the entries of $X_m$
	leads to vanishing of all the terms in \eqref{eq:F_vanishing_to_conj}.
\end{conjecture}

Let us check this conjecture for the above example
\eqref{eq:F_vanishing_to_conj2}. With our $\ell,k$ we need to check all values 
$m=0,1,2,3$.
For $X_3=(0,0,0)$, any permutation leads to the vanishing.
For $X_2=(0,0,\tm^{-2})$, we set $u_1=\tm^{-2}$ and $u_2=u_3=0$.
For $X_1=(0,\tm^{-1},\tm^{-2})$ and $X_0=(1,\tm^{-1},\tm^{-2})$, 
we set $u_1=\tm^{-2}$, $u_2=\tm^{-1}$, and $u_3=0$ or $1$, respectively.
We see that these assignments of the variables indeed lead to the vanishing of the 
function $F_\lambda^{HL}$ written in the refined form \eqref{eq:F_vanishing_to_conj2}.

\subsection{Cauchy type identity for interpolation functions}
\label{sub:Olshanski_from_us}

The interpolation Hall--Littlewood polynomials $F_\lambda^{HL}$
satisfy another Cauchy type identity with product right-hand side.
This identity was established in \cite[Proposition 9.4]{olshanski2019interpolation}
and reads
\begin{equation}
	\label{eq:Olshanski_Cauchy}
	\sum_{\mu} 
	F_\mu^{HL}(u_1,\ldots,u_N;\tm)\,
	G_\mu^{HL}(y_1,\ldots,y_K ;\tm)
	=
	\prod_{j=1}^{K}\frac{y_j-\tm^{1-N}}{y_j-\tm}
	\prod_{i=1}^{N}\prod_{j=1}^{M}\frac{y_j-\tm u_i}{y_j-u_i},
\end{equation}
where the sum is over signatures with $\ell(\mu)\le \min(N,K)$
(the number of zeroes does not matter),
and $G_\lambda^{HL}$ is another family of functions.
In fact, these functions are 
symmetric formal power series in the $y_j^{-1}$'s.
We refer to \cite[Section 3.2 and Section 9.2]{olshanski2019interpolation} 
for details.
An explicit combinatorial formula for $G_\lambda^{HL}$ is given in
\cite[Lemma 9.3]{olshanski2019interpolation}.
It may be interpreted as a representation of $G_\lambda^{HL}$
as a partition function of a certain vertex model.

Let us now connect 
identity \eqref{eq:Olshanski_Cauchy} to
the Cauchy identity for the spin Hall--Littlewood 
functions \cite[Corollary 4.13]{BorodinPetrov2016inhom}. Let us recall the latter:
\begin{equation}
	\label{eq:sHL_Cauchy_S5}
	\sum_{\lambda\in \mathrm{Sign}_N}
	\mathsf{F}_\lambda(u_1,\ldots,u_N )
	\,
	\mathsf{G}^*_\lambda(v_1,\ldots,v_K )=
	\frac{(q;q)_N}{\prod_{i=1}^{N}(1-s_0\xi_0 u_i)}\,
	\prod_{i=1}^{N}\prod_{j=1}^{K}\frac{1-q u_iv_j}{1-u_iv_j}.
\end{equation}
The functions $\mathsf{G}_\lambda^*$ are \cite[Theorem 4.14.2]{BorodinPetrov2016inhom}
given by the following symmetrization formula
(for $K\ge \ell(\lambda)$, otherwise the function vanishes):
\begin{multline*}
	\mathsf{G}_\lambda^*
	(v_1,\ldots,v_K )
	=\frac{(q;q)_N}{(q;q)_{m_0(\lambda)}(q;q)_{K-\ell(\lambda)}}
	\prod_{r\ge1}\frac{(s_r^2;q)_{m_r(\lambda)}}{(q;q)_{m_r(\lambda)}}
	\sum_{\sigma\in S_K}
	\sigma\Biggl(
		\prod_{1\le i<j\le K}
		\frac{v_i-q v_j}{v_i-v_j}
		\\\times
		\prod_{i=1}^{\ell(\lambda)}\frac{v_i}{v_i-s_0 \xi_0}
		\prod_{i=\ell(\lambda)+1}^{K}(1-v_i q^{m_0(\lambda)}s_0/\xi_0)
		\prod_{j=1}^{K}
		\left( 
			\varphi_{\lambda_j}(v_j)
			\Big\vert_{\text{$\xi_x\to \xi_x^{-1}$}}
		\right)
	\Biggr),
\end{multline*}
where $\lambda\in \mathrm{Sign}_N$.
Similarly to the second part of \Cref{prop:sHL_to_iHL}, we obtain:
\begin{lemma}
	\label{lemma:G_lambda_spec}
	Specializing the parameters as in \eqref{eq:ab_spec_in_functions},
	we have for all $\lambda\in \mathrm{Sign}_N$:
	\begin{multline*}
		\xi_0^{\ell(\lambda)}\mathsf{G}_\lambda^*
		(v_1,\ldots,v_K )\to
		\frac{(1-\tm)^K(\tm;\tm)_N}{(\tm;\tm)_{K-\ell(\lambda)}}
		\prod_{j=1}^{K}\frac1{1-v_j \tm^{1-N}}
		\prod_{r\ge 0}\frac{1}{(\tm;\tm)_{m_r(\lambda)}}
		\\\times
		\sum_{\sigma\in S_K}
		\sigma\Biggl(
			\prod_{1\le i<j\le K}
			\frac{v_i-\tm v_j}{v_i-v_j}
			\prod_{i=\ell(\lambda)+1}^{K}(1-v_i \tm^{1-\ell(\lambda)})
			\prod_{j=1}^{K}
			v_j^{\lambda_j}
		\Biggr).
	\end{multline*}
\end{lemma}
From \Cref{lemma:G_lambda_spec} and 
comparing Cauchy identities \eqref{eq:Olshanski_Cauchy} and \eqref{eq:sHL_Cauchy_S5},
we get a symmetrization formula for 
the functions $G_\lambda^{HL}$.
\begin{proposition}
	\label{prop:dual_iHL_symmetrization}
	The symmetrization formula
	for the functions $G_\lambda^{HL}$ dual to the interpolation
	Hall--Littlewood polynomials has the form:
	\begin{equation}
		\label{eq:dual_iHL_symmetrization}
		\begin{split}
			G_\lambda^{HL}(y_1,\ldots,y_K;\tm )
			=\frac{(1-\tm)^{K}}{(\tm;\tm)_{K-\ell(\lambda)}}
			\sum_{\sigma\in S_K}
			\sigma\Biggl(
				\prod_{1\le i<j\le K}
				\frac{y_j-\tm y_i}{y_j-y_i}
				\prod_{j=1}^{K}
				y_j^{-\lambda_j}\,
				\frac{y_j-\tm^{1-\ell(\lambda)}\mathbf{1}_{\lambda_j=0}}
				{y_j-\tm}
			\Biggr).
		\end{split}
	\end{equation}
\end{proposition}
In particular, this function does not depend on $N$ (where $\lambda\in \mathrm{Sign}_N$). Note that the right-hand side of 
\eqref{eq:dual_iHL_symmetrization}
is indeed a power series in the $y_i^{-1}$'s, because the 
only fraction that needs to be expanded into an 
infinite series is
$
\frac{y_j-\tm^{1-\ell(\lambda)}\mathbf{1}_{\lambda_j=0}}{y_j-\tm}
=
\frac{1-y_j^{-1}\tm^{1-\ell(\lambda)}\mathbf{1}_{\lambda_j=0}}{1-y_j^{-1}\tm}
$, 
since the denominators in 
$\frac{y_j-\tm y_i}{y_j-y_i}$ are cleared by symmetrization.
\begin{proof}[Proof of \Cref{prop:dual_iHL_symmetrization}]
	The spin Hall--Littlewood Cauchy identity
	\eqref{eq:sHL_Cauchy_S5} and \Cref{lemma:G_lambda_spec}
	imply that 
	the candidate functions $G_\lambda^{HL}$
	given by the right-hand side of 
	\eqref{eq:dual_iHL_symmetrization}
	satisfy the Cauchy identity \eqref{eq:Olshanski_Cauchy}
	together with the functions $F_\lambda^{HL}$.
	Orthogonality of the functions $F_\lambda^{HL}$
	(a consequence of \eqref{eq:torus_orthogonality} and the specialization
	in \Cref{prop:sHL_to_iHL})
	implies that identity \eqref{eq:Olshanski_Cauchy} uniquely determines the 
	coefficients by each individual $F_\lambda^{HL}(u_1,\ldots,u_N ;\tm)$.
	This completes the proof.
\end{proof}

\section{Schur expansion and measures on partitions}
\label{sec:expansions_IK}

The aim of this section is to write down an expansion 
of the determinant in the right-hand side of the refined spin Hall--Littlewood
Cauchy identity, 
\begin{equation}
	\label{eq:IK_gamma_det}
	\mathscr{Z}_N^\gamma(u_1,\ldots,u_N\mid v_1,\ldots,v_N ;s_0 )\,
	\frac{\prod_{j=1}^{N}(1-s_0\xi_0u_j)(1-s_0\xi_0^{-1}v_j)}{\prod_{i,j=1}^{N}(1-q u_iv_j)^{-1}}
	=
	\frac{\det\left[
		\mathsf{z}(u_i,v_j;s_0)
	\right]_{i,j=1}^{N}}{\prod_{1\le i<j\le N}(u_i-u_j)(v_i-v_j)},
\end{equation}
where
$\mathsf{z}(u,v;s_0)$ is given by \eqref{eq:IK_det_element},
in terms of the Schur symmetric polynomials.
This expansion is formulated as \Cref{thm:Schur_expansion}
in \Cref{sub:Schur_expansion}.
Then in \Cref{sub:measures_on_partitions} we discuss connections of our formulas to 
expectations of certain observables of 
Schur and Macdonald measures on partitions.

\subsection{Schur expansion}
\label{sub:Schur_expansion}

Since \eqref{eq:IK_gamma_det} is symmetric in $v_1,\ldots,v_N $, 
let us look for its (Taylor series) expansion
into the Schur polynomials
\begin{equation}
	\label{eq:Schur_poly_def}
	s_\lambda(v_1,\ldots,v_N )=\frac{\det[v_i^{\lambda_j+N-j}]_{i,j=1}^{N}}{\prod_{1\le i<j\le N}(v_i-v_j)}.
\end{equation}
This expansion would look as
\begin{equation}
	\label{eq:IK_gamma_Schur}
	\frac{\det\left[
		\mathsf{z}(u_i,v_j;s_0)
	\right]_{i,j=1}^{N}}{\prod_{1\le i<j\le N}(u_i-u_j)(v_i-v_j)}	
	=
	\sum_{\lambda\in \mathrm{Sign}_N}
	\mathsf{C}_\lambda(u_1,\ldots,u_N;s_0 )\,s_\lambda(v_1,\ldots,v_N ),
\end{equation}
where $\mathsf{C}_\lambda$ are some symmetric functions of the $u_i$'s (treated as some yet unknown coefficients).
Here we assume that the infinite series in 
\eqref{eq:IK_gamma_Schur} converges. After the computation of the $\mathsf{C}_\lambda$'s,
we will see that the series indeed converges for $|u_iv_j|<1$ for all $i,j$,
and this would justify the computation of the $\mathsf{C}_\lambda$'s.

We will employ the well-known torus orthogonality relation
for the Schur polynomials:
\begin{proposition}
	\label{lemma:Schur_orthogonality}
	We have for any $\lambda,\mu\in \mathrm{Sign}_N$:
	\begin{equation*}
		\frac{1}{N!(2\pi \mathbf{i})^N}
		\oint \frac{dv_1}{v_1}
		\ldots
		\oint \frac{dv_N}{v_N}
		\prod_{1\le i< j\le N}(v_i-v_j)(v_i^{-1}-v_j^{-1})\,
		s_\lambda(v_1,\ldots,v_N )\,s_\mu(v_1^{-1},\ldots,v_N^{-1} )
		=
		\mathbf{1}_{\lambda=\mu},
	\end{equation*}
	where each integration is over the positively oriented unit circle.
\end{proposition}
\begin{proof}[Idea of proof]
	While the claim
	is a degeneration of \eqref{eq:torus_orthogonality} as $s_x\equiv 0$, $\xi_x\equiv 1$, $q=0$, 
	it is not hard to prove it independently.
	Namely, 
	expand both determinants in the definition of the Schur polynomials 
	\eqref{eq:Schur_poly_def} as sums of $N!$ terms, and use the orthogonality
	$\frac{1}{2\pi\mathbf{i}}\oint v^{n-1}dv=\mathbf{1}_{n=0}$ for each single variable.
\end{proof}

We have using \Cref{lemma:Schur_orthogonality}:
\begin{equation}
	\label{eq:Schur_main_computation}
	\begin{split}
		&
		\mathsf{C}_\lambda(u_1,\ldots,u_N;s_0 )
		\\&
		\hspace{10pt}
		=
		\frac{1}{N!(2\pi \mathbf{i})^N}
		\oint \frac{dv_1}{v_1}
		\ldots
		\oint \frac{dv_N}{v_N} \,
		\prod_{1\le i< j\le N}(v_i-v_j)(v_i^{-1}-v_j^{-1})
		\\&\hspace{160pt}\times
		\frac{\det\left[
			\mathsf{z}(u_i,v_j;s_0)
		\right]_{i,j=1}^{N}}{\prod_{1\le i<j\le N}(u_i-u_j)(v_i-v_j)}
		\,
		s_\lambda(v_1^{-1},\ldots,v_N^{-1} )
		\\&
		\hspace{10pt}
		=
		\frac{1}{\prod_{i<j}(u_i-u_j)}
		\frac{1}{N!(2\pi \mathbf{i})^N}
		\oint dv_1
		\ldots
		\oint dv_N
		\det\bigl[ v_i^{-(\lambda_j+N-j)-1} \bigr]_{i,j=1}^{N}
		\det\left[
			\mathsf{z}(u_i,v_j;s_0)
		\right]_{i,j=1}^{N}
		\\&
		\hspace{10pt}
		=
		\frac{1}{\prod_{i<j}(u_i-u_j)}
		\det
		\left[ 
			\frac{1}{2\pi\mathbf{i}}
			\oint
			\frac{\mathsf{z}(u_i,v;s_0)\,dv}{v^{\lambda_j+N-j+1}}
		\right]_{i,j=1}^{N}.
	\end{split}
\end{equation}
In the last equality 
we have expanded both determinants as sums of $N!$ terms, and wrote the multiple integration
as a product of single integrations. This allows to cancel $N!$ from the denominator, and 
write the result as a determinant of single integrals.

Next, one readily sees that for any $k\ge0$ we have
\begin{equation*}
	\frac{1}{2\pi\mathbf{i}}
	\oint
	\frac{\mathsf{z}(u,v;s_0)\,dv}{v^{k+1}}
	=
	u^k
	\left\{ 
		1-\gamma q^{k+1}
		+
		s_0^2\gamma(\gamma-q^k)
		-
		s_0\gamma\left(
			(1-q^k)(u \xi_0)^{-1}+
			(1-q^{k+1})u\xi_0
		\right)
	\right\}
	,
\end{equation*}
since the integration contour contains the single pole $v=0$.
This leads to the first expression for 
the coefficients $\mathsf{C}_\lambda(u_1,\ldots,u_N ;s_0)$ appearing in the expansion \eqref{eq:IK_gamma_Schur}:
\begin{equation}
	\label{eq:C_lambda_det_function}
	\mathsf{C}_\lambda(u_1,\ldots,u_N;s_0 )=
	\frac{\det[\mathsf{c}_{\lambda_i+N-i}(u_j;s_0)]_{i,j=1}^{N}}{\prod_{1\le i<j\le N}(u_i-u_j)},
\end{equation}
where
\begin{equation}
	\label{eq:small_c_function}
	\mathsf{c}_k(u;s_0)=
	u^k
	\left\{ 
		1-\gamma q^{k+1}
		+
		s_0^2\gamma(\gamma-q^k)
		-
		s_0\gamma\left(
			(1-q^k)(u \xi_0)^{-1}+
			(1-q^{k+1})u\xi_0
		\right)
	\right\}.
\end{equation}
In particular, we see that 
$\mathsf{C}_\lambda(u_1,\ldots,u_N ;s_0 )$ is a symmetric polynomial in $u_1,\ldots,u_N $.
For $s_0=\gamma=0$ we have $\mathsf{C}_\lambda=s_\lambda$.

Let us get another expression for $\mathsf{C}_\lambda$ using the 
other determinantal expression for \eqref{eq:IK_gamma_det}
following from \Cref{thm:Cuenca_style_formula}.
Using orthogonality (\Cref{lemma:Schur_orthogonality})
in the same way as 
in the computation \eqref{eq:Schur_main_computation}, we have
\begin{equation}
	\label{eq:C_lambda_Jacobi_Trudi_type}
	\begin{split}
		&\mathsf{C}_\lambda(u_1,\ldots,u_N;s_0 )
		=
		\det
		\Biggl[ 
			\frac{1}{2\pi \mathbf{i}}
			\oint
			\frac{
			\xi_0^{2i-N}\, dv}{v^{\lambda_j+i-j+2}}
			\\&\hspace{40pt}\times
			\Biggl\{ 
				\frac{\left( 1-\gamma s_0\xi_0^{-1}v \right)
				\left( v\xi_0^{-1}-\gamma s_0 \right)}
				{\prod_{l=1}^{N}(1-v u_l)}
				-
				\gamma q^{N-i}\,
				\frac{\left( 1-s_0\xi_0^{-1}v \right)
				\left( q v \xi_{0}^{-1}-s_0 \right)}{
				\prod_{l=1}^{N}(1-qv u_l)}
			\Biggr\}
		\Biggr]_{i,j=1}^N,
	\end{split}
\end{equation}
where the integration contour is around $0$ and encircles no other poles.
The result of the integration in \eqref{eq:C_lambda_Jacobi_Trudi_type} 
can be expressed in terms of the complete homogeneous
symmetric polynomials
\begin{equation*}
	h_k=h_k(u_1,\ldots,u_N )=\frac{1}{2\pi \mathbf{i}}\oint 
	\frac{dv}{v^{k+1}}\frac{1}{\prod_{l=1}^{N}(1-vu_l)}
	=\sum_{1\le i_1\le \ldots\le i_k\le N }u_{i_1}\ldots u_{i_k} .
\end{equation*}
In this way we obtain:
\begin{equation}
	\label{eq:C_lambda_H}
	\begin{split}
		&\mathsf{C}_\lambda(u_1,\ldots,u_N ;s_0)=
		\det\Bigl[
			-\gamma s_0\xi_0^{-1}(1-q^{\lambda_j+N-j})h_{\lambda_j+i-j-1}
			\\&\hspace{30pt}
			+
			\left( 1+\gamma^2s_0^2-\gamma(q+s_0^2)q^{\lambda_j+N-j} \right)
			h_{\lambda_j+i-j}
			-\gamma s_0 \xi_0(1-q^{\lambda_j+N-j+1})
			h_{\lambda_j+i-j+1}
		\Bigr]_{i,j=1}^N,
	\end{split}
\end{equation}
where $h_k=h_k(u_1,\ldots,u_N )$ are the complete homogeneous symmetric polynomials.

\begin{remark}
	\label{rmk:det_det_vs_JT}
	Comparing formulas \eqref{eq:C_lambda_det_function} and
	\eqref{eq:C_lambda_H} for the functions $\mathsf{C}_\lambda$, we note 
	that the former resembles the definition of the Schur polynomial \eqref{eq:Schur_poly_def}.
	There is a Schur level analogue of \eqref{eq:C_lambda_H}, namely,
	the Jacobi--Trudi formula
	$s_\lambda(u_1,\ldots,u_N )=
	\det[h_{\lambda_i+j-i}(u_1,\ldots,u_N )]_{i,j=1}^{N}$ 
	which is obtained from \eqref{eq:C_lambda_H} by setting $s_0=\gamma=0$.
\end{remark}

The computations above in this subsection
lead to the following result.

\begin{theorem}[Schur expansion]
	\label{thm:Schur_expansion}
	If $|u_iv_j|<1$ for all $i,j$, then we have
	\begin{equation*}
		\frac{\det\left[
			\mathsf{z}(u_i,v_j;s_0)
		\right]_{i,j=1}^{N}}{\prod_{1\le i<j\le N}(u_i-u_j)(v_i-v_j)}
		=
		\sum_{\lambda\in \mathrm{Sign}_N}
		\mathsf{C}_\lambda(u_1,\ldots,u_N;s_0)
		\,s_\lambda(v_1,\ldots,v_N ).
	\end{equation*}
	Here $s_\lambda$ are the Schur
	symmetric polynomials, and 
	$\mathsf{C}_\lambda$ are symmetric polynomials which are
	deformations of the $s_\lambda$'s depending on four additional parameters
	$q,s_0,\xi_0,\gamma$.
	The polynomials $\mathsf{C}_\lambda$ are given by a
	double alternant formula \eqref{eq:C_lambda_det_function}--\eqref{eq:small_c_function}
	or by a Jacobi--Trudi type formula \eqref{eq:C_lambda_H}.
\end{theorem}

Combining this with the refined Cauchy identity (\Cref{thm:intro_sHL_refined}), 
we get:

\begin{corollary}
	\label{cor:refined_sHL_and_Schur}
	Let $u_i,v_j$ satisfy $|u_iv_j|<1$ and condition \eqref{eq:admissible_u_v}
	for all $i,j$. Then we have
	\begin{equation*}
		\begin{split}
			&\prod_{i,j=1}^{N}\frac{1}{1-qu_iv_j}
			\sum_{\lambda\in \mathrm{Sign}_N}
			\frac{(\gamma q;q)_{m_0(\lambda)}(\gamma^{-1}s_0^2;q)_{m_0(\lambda)}}{(q;q)_{m_0(\lambda)}(s_0^2;q)_{m_0(\lambda)}}
			\,
			\mathsf{F}_{\lambda}(u_1,\ldots,u_N )\,
			\mathsf{F}^*_{\lambda}(v_1,\ldots,v_N )
			\\&\hspace{50pt}=
			\prod_{j=1}^{N}\frac{1}{(1-s_0\xi_0u_j)(1-s_0\xi_0^{-1}v_j)}
			\sum_{\lambda\in \mathrm{Sign}_N}
			\mathsf{C}_\lambda(u_1,\ldots,u_N;\gamma^{-1}s_0)
			\,s_\lambda(v_1,\ldots,v_N ),
		\end{split}
	\end{equation*}
	where $\mathsf{C}_\lambda$ is given by \eqref{eq:C_lambda_det_function}
	or \eqref{eq:C_lambda_H}.
\end{corollary}


\subsection{Connection to measures on partitions for $s_0=0$}
\label{sub:measures_on_partitions}

Formulas from the previous \Cref{sub:Schur_expansion}
become much simpler if we set $s_0=0$.
The resulting identities are related to expectations of certain observables 
with respect to probability measures on partitions.

We start by noting that
\begin{equation}
	\label{eq:C_lambda_s0}
	\mathsf{C}_\lambda(u_1,\ldots,u_N; 0)=
	\prod_{j=1}^{N}(1-(\gamma q)q^{\lambda_j+N-j})\,
	s_\lambda(u_1,\ldots,u_N ),
\end{equation}
this follows by comparing the $s_0=0$ case of \eqref{eq:C_lambda_det_function}
with \eqref{eq:Schur_poly_def}.

Using stable spin Hall--Littlewood functions $\widetilde{\mathsf{F}}_\lambda,
\widetilde{\mathsf{F}}_\lambda^*$ (see \eqref{eq:stable_f_2}, \eqref{eq:dual_stable_f_2}),
\Cref{cor:refined_sHL_and_Schur} specializes at $s_0=0$
as follows:
\begin{equation}
	\label{eq:Cor53_s_0}
		\begin{split}
			&\prod_{i,j=1}^{N}\frac{1}{1-qu_iv_j}
			\sum_{\lambda\in \mathrm{Sign}_N}
			(\gamma q;q)_{m_0(\lambda)}
			\,
			\widetilde{\mathsf{F}}_{\lambda}(u_1,\ldots,u_N )\,
			\widetilde{\mathsf{F}}^*_{\lambda}(v_1,\ldots,v_N )
			\\&\hspace{50pt}=
			\sum_{\lambda\in \mathrm{Sign}_N}
			\prod_{j=1}^{N}
			(1-(\gamma q)q^{\lambda_j+N-j})
			\,s_\lambda(u_1,\ldots,u_N)
			\,s_\lambda(v_1,\ldots,v_N ),
		\end{split}
\end{equation}
\begin{remark}
	\label{rmk:Cor53_s0_det}
	Using the refined Cauchy identity (\Cref{thm:intro_sHL_refined}) and \Cref{thm:Cuenca_style_formula}, we note that 
	\eqref{eq:Cor53_s_0}
	is also equal to either of the following expressions:
	\begin{equation}
		\label{eq:det_s_0}
		\begin{split}
			&\frac{1}{\prod_{1\le i<j\le N}(u_i-u_j)(v_i-v_j)}\,
			\det\left[ 
				\frac{1-q+q(1-\gamma)(1-u_iv_j)}{(1-u_iv_j)(1-q u_iv_j)}
			\right]_{i,j=1}^{N}
			\\&\hspace{40pt}=
			\frac{1}{\prod_{1\le i<j\le N}(u_i-u_j)}\,
			\det
			\left[ 
				u_j^{N-i}
				\left\{ 
					\frac{1}{\prod_{l=1}^{N}(1-u_jv_l)}
					-
					\frac{\gamma q^{N-i+1}}{\prod_{l=1}^{N}(1-qu_jv_l)}
				\right\}
			\right]_{i,j=1}^N
			.
		\end{split}
	\end{equation}
	However, for the purposes of the discussion in this subsection we 
	focus only on the identity between the two infinite sums in \eqref{eq:Cor53_s_0}.
\end{remark}

The right-hand side of \eqref{eq:Cor53_s_0} is a 
specialization of a more general summation identity involving 
Macdonald symmetric polynomials.
This identity may be interpreted as an expectation with respect to a Macdonald measure.

\begin{definition}[\cite{fulman1997probabilistic}, \cite{ForresterRains2005Macdonald}, \cite{BorodinCorwin2011Macdonald}]
	\label{def:MM}
	The \emph{Macdonald measure} with parameters 
	$\qm,\tm\in [0,1)$ and variables
	$u_1,\ldots,u_N ,v_1,\ldots,v_N>0$, such that $|u_iv_j|<1$,
	is a probability measure on $\mathrm{Sign}_N$
	with probability weights given by 
	\begin{equation*}
		\mathrm{Prob}_{\mathbf{MM}(\qm,\tm)}(\lambda)=
		\prod_{i,j=1}^{N}\frac{(u_iv_j;\qm)_{\infty}}{(\tm u_iv_j;\qm)_{\infty}}\,
		P_\lambda(u_1,\ldots,u_N;\qm,\tm )
		Q_\lambda(v_1,\ldots,v_N;\qm,\tm ),
	\end{equation*}
	where $P_\lambda,Q_\lambda$ are the Macdonald symmetric polynomials (see \Cref{sub:homog_Macd_polys}
	and references therein).
	By $\mathbb{E}_{\mathbf{MM}(\qm,\tm)}$ we denote expectations with respect to this 
	Macdonald measure, where $\lambda\in \mathrm{Sign}_N$ is viewed as the corresponding random signature.
\end{definition}
Recall the following particular cases of Macdonald polynomials
which are important for the present discussion:
\begin{itemize}
	\item When $\qm=\tm$, we have $P_\lambda=Q_\lambda=s_\lambda$, the Schur polynomials \eqref{eq:Schur_poly_def};
		notably, they do not depend on the value of the parameter $\qm=\tm$);
	\item When $\qm=0$, the Macdonald polynomials become the Hall--Littlewood symmetric polynomials
		$P_\lambda^{HL}$ and $Q_\lambda^{HL}$, see \Cref{sub:HL_deg_of_interpolation}.
\end{itemize}

\begin{definition}
	\label{def:sHL_M}
	Using the Cauchy identity for the stable spin Hall--Littlewood functions (\Cref{prop:stable_Cauchy}),
	we may define the measure on partitions
	\begin{equation*}
		\mathrm{Prob}_{\mathbf{sHL}(q)}(\lambda):=\prod_{i,j=1}^{N}\frac{1-u_iv_j}{1-qu_iv_j}
		\,
		\widetilde{\mathsf{F}}_\lambda(u_1,\ldots,u_N )\,
		\widetilde{\mathsf{F}}_\lambda^*(v_1,\ldots,v_N ).
	\end{equation*}
	This measure was discussed in
	\cite{BufetovMucciconiPetrov2018} in connection with stochastic
	vertex models and one-dimensional interacting particle systems. The probability weights
	$\mathrm{Prob}_{\mathbf{sHL}(q)}(\lambda)$
	are nonnegative when $q\in[0,1)$, $s_x\in(-1,0]$, $\xi_x u_i,v_j/\xi_x \in[0,1)$ for all $x,i,j$.
	We denote expectations with respect to this measure by $\mathbb{E}_{\mathbf{sHL}(q)}$.
\end{definition}

With this notation, identity \eqref{eq:Cor53_s_0} (which is the $s_0=0$ degeneration of \Cref{cor:refined_sHL_and_Schur}),
multiplied by $\prod_{i,j=1}^{N}(1-u_iv_j)$ and with a changed parameter $\zeta=-\gamma q$,
is equivalent to
an identity of expectations:
\begin{equation}
	\label{eq:Cor53_s_0_exp}
	\mathbb{E}_{\mathbf{sHL}(q)}(-\zeta;q)_{m_0(\lambda)}
	=
	\mathbb{E}_{\mathbf{MM}(q,q)}\prod_{j=1}^{N}\left( 1+\zeta q^{\lambda_j+N-j} \right).
\end{equation}

At the same time, the right-hand side of \eqref{eq:Cor53_s_0_exp}
extends to the full Macdonald level:
\begin{proposition}[$\qm$-independence in Macdonald measure]
	\label{prop:Macdonald_q_independence}
	The expectation 
	\begin{equation}
		\label{eq:MM_expectation}
		\mathbb{E}_{\mathbf{MM}(\qm,\tm)}\prod_{j=1}^{N}(1+\zeta \qm^{\lambda_j}\tm^{N-j})
	\end{equation}
	is independent of the parameter $\qm$, and is equal to any of the expressions
	\eqref{eq:Cor53_s_0}, \eqref{eq:det_s_0}, or \eqref{eq:Cor53_s_0_exp},
	with the spin Hall--Littlewood parameter $q$ replaced by 
	$\tm$ in each of them.
\end{proposition}
This result goes back to 
\cite{kirillov1999q}, see also \cite{warnaar2008bisymmetric}. 
More recently, in \cite{borodin2016stochastic_MM} 
expectation \eqref{eq:MM_expectation} was related to another type of an expectation
with respect to a stochastic higher spin six vertex model. We do not use the latter connection here.
The fact that \eqref{eq:MM_expectation} equals $\eqref{eq:Cor53_s_0_exp}\big\vert_{\text{rename $q$ to $\tm$}}$ 
follows from the $\qm$-independence in \Cref{prop:Macdonald_q_independence} 
after setting $\qm=\tm$.

\begin{corollary}
	\label{cor:HL_sHL}
	Let $\lambda\in \mathrm{Sign}_N$ be the random signature 
	distributed according to the Hall--Littlewood measure
	$\mathbf{MM}(0,\tm)$, and
	$\nu\in \mathrm{Sign}_N$ be distributed according to the spin Hall--Littlewood
	measure $\mathbf{sHL}(\tm)$
	(that is, with parameter $q$ renamed to $\tm$).
	Then the random variables $m_0(\lambda)$ and $m_0(\nu)$ have the same distribution.
\end{corollary}
\begin{proof}
	Comparing \eqref{eq:Cor53_s_0_exp} with $q$ renamed to $\tm$ and 
	\Cref{prop:Macdonald_q_independence} with $\qm=0$, we see that 
	\begin{equation*}
		\mathbb{E}_{\mathbf{MM}(0,\tm)}(-\zeta;\tm)_{m_0(\lambda)}=
		\mathbb{E}_{\mathbf{sHL}(\tm)}(-\zeta;\tm)_{m_0(\nu)}. 
	\end{equation*}
	Since $\zeta$ is an arbitrary complex number and $m_0\in \left\{ 0,1,\ldots,N  \right\}$, 
	the equality of expectations of $(-\zeta;\tm)_{m_0}$ is enough to 
	conclude equality of distributions of $m_0$.
\end{proof}

Let us conclude this section with several remarks.

\begin{remark}
	We have derived \Cref{cor:HL_sHL}
	from the refined Cauchy identity together with \Cref{prop:Macdonald_q_independence}.
	An alternative path 
	via stochastic models
	already appeared in the literature:
	\begin{itemize}
		\item 
			In \cite{BufetovMatveev2017} the quantity $m_0(\lambda)$, $\lambda\sim \mathbf{MM}(0,\tm)$,
			was identified in distribution (via a very nontrivial
			$\tm$-deformation of the Robinson--Schensted--Knuth correspondence)
			as the value of the height function of the stochastic six vertex model \cite{GwaSpohn1992}, \cite{BCG6V}.
		\item 
			In
			\cite{BufetovMucciconiPetrov2018},
			the same height function of the stochastic six vertex model
			is identified in distribution with $m_0(\nu)$, 
			where $\nu\in \mathbf{sHL}(\tm)$.
			This identification proceeds by 
			another probabilistic construction, bijectivization of the Yang--Baxter equation.
			This already settles the result of \Cref{cor:HL_sHL}.
		\item 
			Yet another argument alternative to \cite{BufetovMucciconiPetrov2018}
			is present
			in the earlier work
			\cite{BufetovPetrovYB2017}.
			There, a
			``dynamic'' $s_0$-deformation of the stochastic six vertex model
			was introduced.
			Its height function is 
			identified (also via a bijectivization of the Yang--Baxter equation)
			with
			$m_0(\mu)$, where $\mu$ is distributed
			according to a measure with probability weights proportional to 
			$\mathsf{F}_\mu(u_1,\ldots,u_N ) \,\mathsf{G}_\mu^*(v_1,\ldots,v_N )$,
			the term in
			the Cauchy identity
			\eqref{eq:F_G_Cauchy}.
			Setting $s_0=0$ in the latter also recovers 
			$\mathbf{sHL}(\tm)$.
	\end{itemize}
\end{remark}

\begin{remark}
	A measure
	with probability weights proportional to 
	$\mathsf{F}_\lambda(u_1,\ldots,u_N )\,\mathsf{F}^*_\lambda(v_1,\ldots,v_N )$
	may also be defined. Its normalization constant
	would not have a product form,
	unlike for Macdonald or spin Hall--Littlewood measures of \Cref{def:MM,def:sHL_M}.
	Rather, this normalization constant 
	is simply the right-hand side of the refined Cauchy identity \eqref{eq:intro_sHL_refined}
	with $\gamma=1$, and it is
	given by 
	\eqref{eq:Z_gamma_1}.
	The positivity of this normalization constant for $u_i,v_j\in[0,1)$
	follows by interpreting it as the
	domain wall partition function
	(\Cref{def:Z}) with nonnegative weights $R_{u_iv_j}$
	given in \Cref{fig:R_weights}.

	At this point it is not clear whether this
	version of a spin Hall--Littlewood 
	measure
	can be applied to interesting particle systems or analyzed asymptotically as $N\to\infty$.
\end{remark}

\begin{remark}
	As we saw in this subsection, setting $s_0=0$ makes 
	the deformed functions
	$\mathsf{C}_\lambda$
	proportional to the Schur polynomials as in \eqref{eq:C_lambda_s0}.
	There could be other 
	interesting degenerations of $\mathsf{C}_\lambda$ 
	simplifying the determinant \eqref{eq:C_lambda_H}.
	For example, the degeneration 
	$s_0\xi_0\to 0$, $s_0/\xi_0\to q^{1-N}$, $\gamma=q^{-N}\chi$ 
	considered in \Cref{sub:intHL_degen_all} produces
	\begin{equation}
		\mathsf{C}_\lambda(u_1,\ldots,u_N ;s_0 \gamma^{-1})
		\to
		\det\Bigl[
			-(q^{1-N}-q^{\lambda_j-j+1})h_{\lambda_j+i-j-1}
			+
			\left( 1-\chi q^{\lambda_j-j+1} \right)
			h_{\lambda_j+i-j}
		\Bigr]_{i,j=1}^N.
		\label{eq:C_lambda_another_degen}
	\end{equation}
	The determinant of linear combinations of two consecutive complete homogeneous
	functions $h_k$ bears some resemblance with determinants arising in the 
	study of Grothendieck symmetric polynomials (e.g., \cite{yeliussizov2020positive}).
	However, a direct connection is unclear at this point.
	Moreover, it is also not very
	clear whether the degeneration \eqref{eq:C_lambda_another_degen}
	has any probabilistic interpretation like the
	one
	discussed above in this subsection.
\end{remark}

\section{Application to ASEP eigenfunctions}
\label{sec:6v_reduction}

In this section we specialize the spin Hall--Littlewood functions to
eigenfunctions of the ASEP (Asymmetric Simple Exclusion Process).
This leads to determinantal summation formulas and
certain multitime observables of the ASEP.
In the context of ASEP, determinantal summation formulas
were discovered earlier by Corwin and Liu 
\cite{LiuCorwinPrivate2019}.

\subsection{ASEP and its eigenfunctions}

The ASEP is a continuous time Markov chain on particle configurations
on $\mathbb{Z}$
which depends on a single parameter $q\in[0,1)$.
We will only consider ASEP with finitely many particles (say,~$N$).
In continuous time, each particle has two independent exponential 
clocks, of rates $1$ and $q$.\footnote{By definition, an exponential clock
of rate $r>0$ rings after an exponential random time $T$ distributed such that
$\mathrm{Prob}\left( T>t \right)=e^{-rt}$, $t\ge0$.}
Clocks of different particles are independent.
When a clock rings, the particle attempts to jump by one to the left (for the clock of rate $1$)
or to the right (for the clock of rate $q$).
If the destination is occupied, the jump is suppressed, and the clock 
restarts.
See \Cref{fig:ASEP} for an illustration.

\begin{figure}[htpb]
	\centering
	\includegraphics{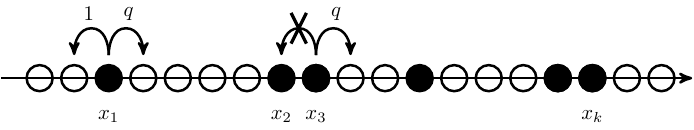}
	\caption{The ASEP particle system.}
	\label{fig:ASEP}
\end{figure}

We denote the particles' coordinates by $\vec x=(x_1<\ldots<x_N )$.
Denote by $\mathscr{A}$ the Markov generator of the ASEP acting on functions 
of $\vec x$:
\begin{equation}
	\label{eq:ASEP_generator}
	(\mathscr{A}f)(\vec x)=
	\sum_{i=1}^{N}
	\left( 
		q 
		\left( f(\vec x+\mathrm{e}_i)-f(\vec x) \right)
		\mathbf{1}_{x_{i+1}>x_i+1}
		+
		\left( f(\vec x-\mathrm{e}_i)-f(\vec x) \right)
		\mathbf{1}_{x_{i-1}<x_i-1}
	\right),
\end{equation}
where $x_0=-\infty$ and $x_{N+1}=+\infty$, by agreement, 
and $\mathrm{e}_i$ is the $i$-th standard basis vector.
Using the coordinate Bethe Ansatz
(e.g., see
	\cite{TW_ASEP1} or 
\cite[Section 7]{BCPS2014_arXiv_v4}),
the left and right eigenfunctions of $\mathscr{A}$ can be written in the following
form:\footnote{In the notation we
	put the variables $\vec x$ into the index, and $\vec z\in \mathbb{C}^N$ are complex parameters,
	to essentially match the notation of symmetric functions.}
\begin{equation}
	\label{eq:ASEP_ef}
	\begin{split}
		\Psi^r_{\vec x}(\vec z)&=
		\sum_{\sigma\in S_N}
		\prod_{1\le i<j\le N}\frac{z_{\sigma(i)}-q z_{\sigma(j)}}{z_{\sigma(i)}-z_{\sigma(j)}}
		\prod_{i=1}^{N}\left( \frac{1-z_{\sigma(i)}}{1-z_{\sigma(i)}/q} \right)^{-x_i}
		,
		\qquad 
		\mathscr{A}\Psi^r_{\vec x}(\vec z)=
		\mathsf{ev}(\vec z)\,\Psi^r_{\vec x}(\vec z);\\
		\Psi^\ell_{\vec x}(\vec z)&=
		\sum_{\sigma\in S_N}
		\prod_{1\le i<j\le N}\frac{q z_{\sigma(i)}- z_{\sigma(j)}}{z_{\sigma(i)}-z_{\sigma(j)}}
		\prod_{i=1}^{N}\left( \frac{1-z_{\sigma(i)}}{1-z_{\sigma(i)}/q} \right)^{x_i}
		,
		\qquad 
		\mathscr{A}^{\text{transpose}}\,\Psi^\ell_{\vec x}(\vec z)
		=
		\mathsf{ev}(\vec z)\,\Psi^\ell_{\vec x}(\vec z),
	\end{split}
\end{equation}
where the transposed generator is the same as
\eqref{eq:ASEP_generator}, but with rates $1$ and $q$ interchanged.
The eigenvalues are
\begin{equation}
	\label{eq:ASEP_eigenvalue}
	\mathsf{ev}(\vec z)=
	-
	\sum_{j=1}^{N}
	\frac{(1-q)^2}{(1-z_j)(1-q/z_j)}.
\end{equation}

We will also need the ASEP transition function,
$P_t(\vec x \to\vec y)$, $t\ge0$,
which is equal to the probability
that the process started from state $\vec x$ at time $0$, is at state $\vec y$ at time $t$.
This transition probability can be written down as an suitable multiple contour integral of the 
product of a left and a right eigenfunction.
A crucial ingredient for such a representation is 
the following orthogonality of the eigenfunctions:
\begin{equation}
	\label{eq:ASEP_Plancherel}
	\frac{1}{N!(2\pi \mathbf{i})^N}\oint
	dz_1 \ldots\oint dz_N 
	\,
	\frac{\prod_{i<j}(z_i-z_j)^2}{\prod_{i\ne j}(z_i-qz_j)}
	\prod_{j=1}^{N}\frac{1-1/q}{(1-z_j)(1-z_j/q)}
	\,
	\Psi^r_{\vec x}(\vec z)
	\,
	\Psi^\ell_{\vec y}(\vec z)
	=\mathbf{1}_{\vec x=\vec y}.
\end{equation}
All integrals here are over a small positively oriented circle around $1$.
Now having this orthogonality and eigenrelations \eqref{eq:ASEP_ef},
it is possible to solve the ASEP master equation\footnote{Also 
referred to as Kolmogorov forward equation, Smoluchowski equation, or Fokker–Planck equation.} 
in $(t,\vec y)$
with the initial condition $\mathbf{1}_{\vec y=\vec x}$ at $t=0$, 
and write
\begin{equation}
	\label{eq:ASEP_transition_function}
	\begin{split}
		&
		P_t(\vec x\to\vec y)=
		\frac{1}{N!(2\pi \mathbf{i})^N}\oint
		dz_1 \ldots\oint dz_N 
		\,
		\frac{\prod_{i<j}(z_i-z_j)^2}{\prod_{i\ne j}(z_i-qz_j)}
		\\&\hspace{120pt}\times
		\prod_{j=1}^{N}\frac{1-1/q}{(1-z_j)(1-z_j/q)}
		\,
		\exp\{t\cdot\mathsf{ev}(\vec z)\}
		\,
		\Psi^r_{\vec x}(\vec z)
		\,
		\Psi^\ell_{\vec y}(\vec z),
	\end{split}
\end{equation}
where all contours are small positive circles around $1$.
We refer to the proofs
of 
\eqref{eq:ASEP_Plancherel}, \eqref{eq:ASEP_transition_function}
and to further details on solving the ASEP particle system to
\cite{TW_ASEP1} or
\cite[Section 7]{BCPS2014_arXiv_v4}.

\subsection{Specialization of the spin Hall--Littlewood functions}

Recall that the spin Hall--Littlewood functions 
$\mathsf{F}_\lambda$
admit
symmetrization formulas
\eqref{eq:phi_notation}--\eqref{eq:F_symmetrization}.
The dual functions $\mathsf{F}^*_\lambda$ are given by \eqref{eq:F_star_explicit}.
To specialize these spin Hall--Littlewood functions to the ASEP eigenfunctions 
\eqref{eq:ASEP_ef}, we take homogeneous parameters
$s_x= s$, $\xi_x= 1$ for all $x\ge0$.
Then we set $s=-1/\sqrt{q}$. This corresponds to passing from the 
higher vertical spin in our vertex model
(as in \Cref{fig:w_weights,fig:w_star_weights})
to spin $\tfrac12$. The latter means that now at most one 
vertical path is allowed per edge.
Then we have for $\vec x=(x_1< \ldots <x_N )$:
\begin{equation}
	\label{eq:sHL_to_ASEP}
	\begin{split}
		\mathsf{F}_{(x_N,\ldots,x_1 )}\left( -\sqrt{q}/z_1,\ldots,-\sqrt{q}/z_N  \right)\Big\vert_{s=-1/\sqrt{q}}&=
		\frac{(1-q)^Nq^{(x_1+\ldots+x_N )/2}}{\prod_{i=1}^{N}(1-1/z_i)}\,
		\Psi^r_{\vec x}(\vec z)
		,
		\\
		\mathsf{F}^*_{(x_N,\ldots,x_1)}\left( -z_1/\sqrt{q},\ldots,-z_N/\sqrt{q}  \right)\Big\vert_{s=-1/\sqrt{q}}&=
		\left( -q \right)^{-N}\,
		\frac{(1-q)^N q^{-(x_1+\ldots+x_N )/2}}{\prod_{i=1}^{N}(1-z_i/q)}\,
		\Psi^\ell_{\vec x}(\vec z)
		.
	\end{split}
\end{equation}
Note that the prefactor in $\mathsf{F}^*_\lambda$ \eqref{eq:F_star_explicit}
vanishes at $s=-1/\sqrt{q}$ unless all multiplicities $m_r(\lambda)$
are either zero or one. In particular, $(s^2;q)_1 /(q;q)_1=(-q)^{-1}$, 
which produces the factor $(-q)^{-N}$ in the second formula in 
\eqref{eq:sHL_to_ASEP}.

\subsection{Summation identities for the ASEP eigenfunctions}
\label{sub:summation_ASEP_ef}

Specializing our main result (\Cref{thm:intro_sHL_refined}), we obtain the following:
\begin{corollary}
	\label{cor:ASEP_summation_formula}
	Let
	\begin{equation}
		\label{eq:ASEP_summability_condition}
		\left|\frac{(1-w_j)(1-z_i/q)}{(1-z_i)(1-w_j/q)}\right|<1
		\qquad 
		\textnormal{for all $i,j$}.
	\end{equation}
	Then
	\begin{align}
		\nonumber
			&
			\sum_{0\le x_1<x_2<\ldots<x_N }
			\Psi^r_{\vec x}(\vec z)\,
			\Psi^\ell_{\vec x}(\vec w)
			\\&\hspace{20pt}=
			\prod_{j=1}^{N}
			(1-z_j)(1-w_j/q)\,
			\frac{(1/q-1)^{-N}\prod_{i,j=1}^{N}(z_i-qw_j)}{\prod_{1\le i< j\le N}(z_i-z_j)(w_j-w_i)}
			\,\det\left[ 
				\frac{1}{(z_i-w_j)(z_i-q w_j)}
			\right]_{i,j=1}^{N}
			\nonumber
			\\&\hspace{20pt}=
			\frac{(1-q)^{-2N}}{\prod_{1\le i<j\le N}(z_j-z_i)}
		\label{eq:ASEP_summation_formula}
			\,
			\det\Biggl[ 
				z_j^{i}
				\biggl\{
					\left( 1-1/z_j \right)
					\left( 1-q/z_j \right)
					\prod_{l=1}^{N}\frac{z_j-qw_l}{z_j-w_l}
					\\&\hspace{240pt}
					\nonumber
					-
					q^{N-i}
					(1-1/z_j)(1-q^2/z_j)
				\biggr\}
			\Biggr]_{i,j=1}^{N}.
	\end{align}
\end{corollary}
In a different language concerning the six-vertex model,
the first of identities \eqref{eq:ASEP_summation_formula}
(leading to the Izergin--Korepin determinant)
essentially appears in \cite{cantini2020integral}.
\begin{proof}[Proof of \Cref{cor:ASEP_summation_formula}]
	This is simply the ASEP specialization of \Cref{thm:intro_sHL_refined} with $\gamma=1$.
	Indeed, observe that for $s=-1/\sqrt{q}$ the 
	summation over $\lambda\in \mathrm{Sign}_N$ 
	turns into a summation over strictly ordered 
	$N$-tuples $\vec x$ with $x_1\ge0$.
	After necessary simplifications we arrive at the desired determinants in the right-hand side.
\end{proof}
\begin{remark}
	\label{rmk:no_refinement_ASEP}
	Note that the refined Cauchy identity of
	\Cref{thm:intro_sHL_refined}
	with $\gamma\ne 1$
	does not specialize nicely to the ASEP case.
	Indeed, for general $\gamma$ the denominator $(s_0^2;q)_{m_0(\lambda)}$ cancels with the 
	same prefactor in $\mathsf{F}^*_{\lambda}$, which means that arbitrarily 
	many vertical paths can be at location $0$.
	Setting $\gamma=1$ removes this issue and leads to a strictly ordered summation
	which precisely matches the summation of the ASEP eigenfunctions.
\end{remark}

The sum of the products
of two ASEP eigenfunctions 
can start from an arbitrary location, not necessarily zero:
\begin{lemma}
	\label{cor:shift_inv}
	Assuming \eqref{eq:ASEP_summability_condition}, 
	for any $k\in \mathbb{Z}$ we have
	\begin{equation*}
		\sum_{k\le x_1<x_2<\ldots<x_N }
		\Psi^r_{\vec x}(\vec z)\,
		\Psi^\ell_{\vec x}(\vec w)
		=
		\prod_{j=1}^{N}
		\left( \frac{(1-w_j)(1-z_j/q)}{(1-z_j)(1-w_j/q)} \right)^{k}
		\sum_{0\le x_1<x_2<\ldots<x_N }
		\Psi^r_{\vec x}(\vec z)\,
		\Psi^\ell_{\vec x}(\vec w).
	\end{equation*}
\end{lemma}
\begin{proof}
	The identity follows from 
	\begin{equation}
		\label{eq:ASEP_eigenfunctions_translation_invariance}
		\Psi^r_{\vec x+k}(\vec z)=\Psi^r_{\vec x}(\vec z)
		\prod_{j=1}^{N}\left( \frac{1-z_j}{1-z_j/q} \right)^{-k}
		,\qquad 
		\Psi^\ell_{\vec x+k}(\vec w)=
		\Psi^\ell_{\vec x}(\vec w)
		\prod_{j=1}^{N}\left( \frac{1-w_j}{1-w_j/q} \right)^{k},
	\end{equation}
	which ensures the desired shifting property.
\end{proof}

We will also need a summation identity for
a single ASEP eigenfunction. This identity goes back to
\cite{TW_ASEP1}, and can also be linked to
the orthogonality of the ASEP eigenfunctions \cite{BCPS2014_arXiv_v4}.

\begin{proposition}
	\label{prop:TW_magic_identity}
	Assume that
	$\left|\frac{1-z_i}{1-z_i/q}\right|<1$ for all $i=1,\ldots,N $.
	Then we have
	\begin{equation*}
		\sum_{0\le x_1<x_2<\ldots<x_N }
		\Psi_{\vec x}^{\ell}(\vec z)
		=
		\frac{(-q)^{\frac{N(N-1)}{2}}(1-z_1/q)\ldots(1-z_N/q) }{(1-1/q)^N z_1 \ldots z_N }.
	\end{equation*}
\end{proposition}
\begin{proof}
	In the proof we use the notation $\zeta_i=(1-z_i)/(1-z_i/q)$.
	Expanding the definition of $\Psi_{\vec x}^{\ell}(\vec z)$ and summing the geometric progressions, we get
	\begin{align*}
		&\sum_{0\le x_1<x_2<\ldots<x_N }
		\Psi_{\vec x}^{\ell}(\vec z)
		=
		\sum_{\sigma\in S_N}
		\sigma\left( \prod_{i<j}\frac{z_j-qz_i}{z_j-z_i}
			\sum_{0\le x_1<x_2<\ldots<x_N }\,
			\prod_{i=1}^{N}\left( \frac{1-z_i}{1-z_i/q} \right)^{x_i}
		\right)
		\\&\hspace{45pt}=
		\sum_{\sigma\in S_N}
		\sigma\left( 
			\prod_{i<j}\frac{q-(1+q)\zeta_i+\zeta_i\zeta_j}{\zeta_j-\zeta_i}
			\,\frac{\zeta_2\zeta_3^2 \ldots \zeta_N^{N-1} }{(1-\zeta_1\zeta_2\ldots \zeta_N )\ldots (1-\zeta_{N-1}\zeta_N)(1-\zeta_N)}
		\right),
	\end{align*}
	where the permutation $\sigma$ acts by permuting the variables $z_i$ or, equivalently, $\zeta_i$.
	The symmetrization in the previous formula
	is simplified using identity \cite[(1.6)]{TW_ASEP1} to
	\begin{equation*}
		\frac{(-q)^{\frac{N(N-1)}{2}}}{(1-\zeta_1)\ldots(1-\zeta_N) }
		=
		\frac{(-q)^{\frac{N(N-1)}{2}}(1-z_1/q)\ldots(1-z_N/q) }{(1-1/q)^N z_1 \ldots z_N },
	\end{equation*}
	which gives the result.
\end{proof}

\subsection{A two-time formula for ASEP}

We will now use the summation identities from 
\Cref{sub:summation_ASEP_ef} to compute 
a two-time probability in ASEP in a contour integral form.
In a similar way one can write down multitime probabilities, 
and we consider two times only to shorten the notation.

\begin{theorem}
	\label{thm:ASEP_two_time}
	Let the $N$-particle ASEP $\vec x(t)$ start from a configuration
	$\vec x(0)=\vec x$, and take arbitrary times $0\le t_1\le t_2$ and locations $k_1,k_2\in \mathbb{Z}$.
	The probability that at time $t_j$ all particles are to the right of $k_j$, $j=1,2$, is equal to
	\begin{align*}
		&
		\mathrm{Prob}\bigl(x_1(t_1)\ge k_1,\ x_1(t_2)\ge k_2\bigr)=
		\frac{(-1)^{N}q^{\frac{N(N-1)}{2}}}{(N!)^2(2\pi \mathbf{i})^{2N}}
		\oint \frac{dz_1}{1-z_1} \ldots\oint \frac{dz_N}{1-z_N}
		\oint \frac{dw_1}{w_1} \ldots\oint \frac{dw_N }{w_N}
		\\&\hspace{40pt}\times
		\frac{\prod_{i<j}(z_i-z_j)(w_i-w_j)\prod_{i,j=1}^{N}(w_i-qz_j)}{\prod_{i\ne j}(z_i-qz_j)(w_i-qw_j)}
		\,\det\left[ 
			\frac{1}{(w_i-z_j)(w_i-q z_j)}
		\right]_{i,j=1}^{N}
		\\&\hspace{40pt}\times
		\exp\{t_1\,\mathsf{ev}(\vec z)+(t_2-t_1)\,\mathsf{ev}(\vec w)\}
		\,
		\prod_{j=1}^{N}
		\left( \frac{1-z_j}{1-z_j/q} \right)^{k_1}\left( \frac{1-w_j}{1-w_j/q} \right)^{k_2-k_1}
		\Psi^r_{\vec x}(\vec z)
		.
	\end{align*}
	All integration contours are small positively oriented 
	circles around $1$, 
	with $|z_i-1|< |w_i-1|$ for all $z_i,w_j$ on the contours.
\end{theorem}
In the case of TASEP (which is ASEP with $q=0$, i.e., only with left jumps),
multitime formulas and their asymptotics were recently studied in
\cite{JohanssonRahman2019},
\cite{liu2019multi}.
\begin{proof}[Proof of \Cref{thm:ASEP_two_time}]
	We can write using \eqref{eq:ASEP_transition_function}:
	\begin{align*}
		&
		\mathrm{Prob}\bigl(x_1(t_1)\ge k_1,\ x_1(t_2)\ge k_2\bigr)=
		\sum_{x_1'\ge k_1,\ x_1''\ge k_2}
		P_{t_1}(\vec x\to \vec x')\,
		P_{t_2-t_1}(\vec x'\to \vec x'')
		\\&\hspace{20pt}=
		\frac{1}{(N!)^2(2\pi \mathbf{i})^{2N}}
		\sum_{x_1'\ge k_1,\ x_1''\ge k_2}
		\oint dz_1 \ldots\oint dz_N 
		\oint dw_1 \ldots\oint dw_N 
		\,
		\frac{\prod_{i<j}(z_i-z_j)^2}{\prod_{i\ne j}(z_i-qz_j)}
		\\&\hspace{40pt}\times
		\frac{\prod_{i<j}(w_i-w_j)^2}{\prod_{i\ne j}(w_i-qw_j)}
		\prod_{j=1}^{N}\frac{1-1/q}{(1-z_j)(1-z_j/q)}
		\prod_{j=1}^{N}\frac{1-1/q}{(1-w_j)(1-w_j/q)}
		\\&\hspace{40pt}\times
		\exp\{t_1\,\mathsf{ev}(\vec z)+(t_2-t_1)\,\mathsf{ev}(\vec w)\}
		\,
		\Psi^r_{\vec x}(\vec z)
		\,
		\Psi^\ell_{\vec x'}(\vec z)
		\Psi^r_{\vec x'}(\vec w)
		\,
		\Psi^\ell_{\vec x''}(\vec w)
		.
	\end{align*}
	All integration contours are small positive circles around $1$.
	By deforming the $z_i$ contours to be sufficiently closer to $1$ than the $w_j$ ones (the $z$ variables are so far independent from the $w$'s), 
	we can make sure that on the new contours we have
	\begin{equation*}
		\left|\frac{(1-z_i)(1-w_j/q)}{(1-w_j)(1-z_j/q)}\right|<1,\qquad 
		\left|\frac{1-w_i}{1-w_i/q}\right|<1
		\qquad \textnormal{for all $i,j$}.
	\end{equation*}
	This allows to bring the summations inside the integrals, 
	and apply \Cref{cor:ASEP_summation_formula} (with \Cref{cor:shift_inv}) and \Cref{prop:TW_magic_identity}.
	After simplifications we obtain the desired formula.
\end{proof}

In \cite{TW_ASEP1},
a single-time formula for ASEP
(which is essentially \eqref{eq:ASEP_transition_function})
was transformed, in the special case of the step initial data $x_i(0)=i$, $i=1,\ldots,N $, and $N\to+\infty$,
to a Fredholm determinantal type formula for the distribution
$\mathrm{Prob}(x_m(t)=k)$ of the $m$-th particle for arbitrary $m$.
This allowed to perform, in \cite{TW_ASEP2}, an asymptotic analysis of this distribution, 
and obtain the GUE Tracy--Widom fluctuation behavior on the scale $t^{1/3}$.
The two-time formula of \Cref{thm:ASEP_two_time}
is more complicated: it contains a nontrivial Izergin--Korepin type determinant under the integral.
It is not clear yet how to proceed 
to asymptotic results from this formula.
It is worth noting that in determinantal models
such as the last-passage percolation, 
two- (and multi-)time and asymptotic behavior was investigated
in, e.g.,
\cite{johansson2015two}, \cite{JohanssonRahman2019}, 
and \cite{baik2019multipoint}.
At the same time in non-determinantal models
such as ASEP and the six-vertex model,
multitime and multipoint asymptotic analysis is severely limited,
cf. \cite{dimitrov2020two}.

\medskip

\textsc{L. Petrov, University of Virginia, Department of Mathematics, 141 Cabell Drive, Kerchof Hall, P.O. Box 400137, Charlottesville, VA 22904, USA, and Institute for Information Transmission Problems, Bolshoy Karetny per. 19, Moscow, 127994, Russia}

E-mail: \texttt{lenia.petrov@gmail.com}


\begin{thebibliography}{GdGW17}

\bibitem[BC14]{BorodinCorwin2011Macdonald}
A.~Borodin and I.~Corwin.
\newblock Macdonald processes.
\newblock {\em Probab. Theory Relat. Fields}, 158:225--400, 2014.
\newblock arXiv:1111.4408 [math.PR].

\bibitem[BCG16]{BCG6V}
A.~Borodin, I.~Corwin, and V.~Gorin.
\newblock Stochastic six-vertex model.
\newblock {\em Duke J. Math.}, 165(3):563--624, 2016.
\newblock arXiv:1407.6729 [math.PR].

\bibitem[BCPS15]{BCPS2014_arXiv_v4}
A.~Borodin, I.~Corwin, L.~Petrov, and T.~Sasamoto.
\newblock Spectral theory for interacting particle systems solvable by
  coordinate bethe ansatz.
\newblock {\em Commun. Math. Phys.}, 339(3):1167--1245, 2015.
\newblock Updated version including erratum. Available at
  \url{https://arxiv.org/abs/1407.8534v4}.

\bibitem[BL19]{baik2019multipoint}
J.~Baik and Z.~Liu.
\newblock Multipoint distribution of periodic tasep.
\newblock {\em Jour. AMS}, 2019.
\newblock arXiv:1710.03284 [math.PR].

\bibitem[BM18]{BufetovMatveev2017}
A.~Bufetov and K.~Matveev.
\newblock Hall-littlewood rsk field.
\newblock {\em Selecta Math.}, 24(5):4839--4884, 2018.
\newblock arXiv:1705.07169 [math.PR].

\bibitem[BMP19]{BufetovMucciconiPetrov2018}
A.~Bufetov, M.~Mucciconi, and L.~Petrov.
\newblock Yang-baxter random fields and stochastic vertex models.
\newblock {\em arXiv preprint}, 2019.
\newblock arXiv:1905.06815 [math.PR]. To appear in Adv. Math.

\bibitem[Bor17]{Borodin2014vertex}
A.~Borodin.
\newblock On a family of symmetric rational functions.
\newblock {\em Adv. Math.}, 306:973--1018, 2017.
\newblock arXiv:1410.0976 [math.CO].

\bibitem[Bor18]{borodin2016stochastic_MM}
A.~Borodin.
\newblock Stochastic higher spin six vertex model and macdonald measures.
\newblock {\em Jour. Math. Phys.}, 59(2):023301, 2018.
\newblock arXiv:1608.01553 [math-ph].

\bibitem[BP18]{BorodinPetrov2016inhom}
A.~Borodin and L.~Petrov.
\newblock Higher spin six vertex model and symmetric rational functions.
\newblock {\em Selecta Math.}, 24(2):751--874, 2018.
\newblock arXiv:1601.05770 [math.PR].

\bibitem[BP19]{BufetovPetrovYB2017}
A.~Bufetov and L.~Petrov.
\newblock {Yang-Baxter field for spin Hall-Littlewood symmetric functions}.
\newblock {\em Forum Math. Sigma}, 7:e39, 2019.
\newblock arXiv:1712.04584 [math.PR].

\bibitem[BW16]{betea2016refined}
D.~Betea and M.~Wheeler.
\newblock Refined cauchy and littlewood identities, plane partitions and
  symmetry classes of alternating sign matrices.
\newblock {\em Journal of Combinatorial Theory, Series A}, 137:126--165, 2016.
\newblock arXiv:1402.0229 [math.CO].

\bibitem[BW17]{BorodinWheelerSpinq}
A.~Borodin and M.~Wheeler.
\newblock Spin $q$-whittaker polynomials.
\newblock {\em arXiv preprint}, 2017.
\newblock arXiv:1701.06292 [math.CO].

\bibitem[BWZJ15]{betea2015refined}
D.~Betea, M.~Wheeler, and P.~Zinn-Justin.
\newblock Refined cauchy/littlewood identities and six-vertex model partition
  functions: Ii. proofs and new conjectures.
\newblock {\em J. Algebr. Comb.}, 42(2):555--603, 2015.
\newblock arXiv:1405.7035 [math.CO].

\bibitem[CCP20]{cantini2020integral}
L.~Cantini, F.~Colomo, and A.~Pronko.
\newblock Integral formulas and antisymmetrization relations for the six-vertex
  model.
\newblock {\em Annales Henri Poincar{\'e}}, 21(3):865--884, 2020.
\newblock arXiv:1906.07636 [math-ph].

\bibitem[CD21]{chen2021stable}
K.~Chen and X.~Ding.
\newblock Stable spin hall-littlewood symmetric functions, combinatorial
  identities, and half-space yang-baxter random field.
\newblock {\em arXiv preprint}, 2021.
\newblock arXiv:2106.12557 [math-ph].

\bibitem[CL20]{LiuCorwinPrivate2019}
I.~Corwin and Z.~Liu.
\newblock In preparation.
\newblock 2020.

\bibitem[CP16]{CorwinPetrov2015arXiv}
I.~Corwin and L.~Petrov.
\newblock {Stochastic higher spin vertex models on the line}.
\newblock {\em Commun. Math. Phys.}, 343(2):651--700, 2016.
\newblock Updated version including erratum. Available at
  \url{https://arxiv.org/abs/1502.07374v2}.

\bibitem[Cue18]{cuenca2018interpolation}
C.~Cuenca.
\newblock Interpolation macdonald operators at infinity.
\newblock {\em Advances in Applied Mathematics}, 101:15--59, 2018.
\newblock arXiv:1712.08014 [math-ph].

\bibitem[Dim20]{dimitrov2020two}
E.~Dimitrov.
\newblock Two-point convergence of the stochastic six-vertex model to the airy
  process.
\newblock {\em arXiv preprint}, 2020.
\newblock arXiv:2006.15934 [math.PR].

\bibitem[FR05]{ForresterRains2005Macdonald}
P.J. Forrester and E.M. Rains.
\newblock {Interpretations of some parameter dependent generalizations of
  classical matrix ensembles}.
\newblock {\em Probab. Theory Relat. Fields}, 131(1):1--61, 2005.

\bibitem[Ful97]{fulman1997probabilistic}
J.~Fulman.
\newblock {Probabilistic measures and algorithms arising from the Macdonald
  symmetric functions}.
\newblock {\em arxiv preprint}, 1997.
\newblock arXiv:math/9712237 [math.CO].

\bibitem[Gav21]{gavrilova2021refined}
S.~Gavrilova.
\newblock Refined littlewood identity for spin hall-littlewood symmetric
  rational functions.
\newblock {\em arXiv preprint}, 2021.
\newblock arXiv:2104.09755 [math.CO].

\bibitem[GdGW17]{deGierWheeler2016}
A.~Garbali, J.~de~Gier, and M.~Wheeler.
\newblock {A new generalisation of Macdonald polynomials}.
\newblock {\em Comm. Math. Phys}, 352(2):773--804, 2017.
\newblock arXiv:1605.07200 [math-ph].

\bibitem[GS92]{GwaSpohn1992}
L.-H. Gwa and H.~Spohn.
\newblock Six-vertex model, roughened surfaces, and an asymmetric spin
  {H}amiltonian.
\newblock {\em Phys. Rev. Lett.}, 68(6):725--728, 1992.

\bibitem[HL18]{harnad2018symmetric}
J.~Harnad and E.~Lee.
\newblock Symmetric polynomials, generalized jacobi-trudi identities and
  $\tau$-functions.
\newblock {\em Journal of Mathematical Physics}, 59(9):091411, 2018.
\newblock arXiv:1304.0020 [math-ph].

\bibitem[Ize87]{Izergin1987}
A.~G. Izergin.
\newblock Partition function of the six-vertex model in a finite volume.
\newblock {\em Sov. Phys. Dokl.}, 32:878--879, 1987.

\bibitem[Joh16]{johansson2015two}
K.~Johansson.
\newblock {Two time distribution in Brownian directed percolation}.
\newblock {\em Commun. Math. Phys.}, pages 1--52, 2016.
\newblock arXiv:1502.00941 [math-ph].

\bibitem[JR21]{JohanssonRahman2019}
K.~Johansson and M.~Rahman.
\newblock Multi-time distribution in discrete polynuclear growth.
\newblock {\em Comm. Pure Appl. Math., online}, 2021.
\newblock arXiv:1906.01053 [math.PR].

\bibitem[KBI93]{QISM_book}
V.~Korepin, N.~Bogoliubov, and A.~Izergin.
\newblock {\em Quantum inverse scattering method and correlation functions}.
\newblock Cambridge University Press, Cambridge, 1993.

\bibitem[KN99]{kirillov1999q}
A.N. Kirillov and M.~Noumi.
\newblock {q-difference raising operators for Macdonald polynomials and the
  integrality of transition coefficients}.
\newblock In {\em Algebraic Methods and q-Special Functions, CRM Proceedings
  and Lecture Notes}, volume~22, pages 227--243, 1999.
\newblock arXiv:q-alg/9605005.

\bibitem[Kno97]{knop1996symmetric}
F.~Knop.
\newblock Symmetric and non-symmetric quantum capelli polynomials.
\newblock {\em Comment. Math. Helv}, (72):84--100, 1997.
\newblock arXiv:q-alg/9603028.

\bibitem[Kra99]{krattenthaler1999advanced}
C.~Krattenthaler.
\newblock Advanced determinant calculus.
\newblock {\em S{\'e}minaire Lotharingien Combin}, 42:B42q, 1999.
\newblock arXiv:math/9902004 [math.CO].

\bibitem[KS96]{Koekoek1996}
R.~Koekoek and R.F. Swarttouw.
\newblock {The Askey-scheme of hypergeometric orthogonal polynomials and its
  q-analogue}.
\newblock Technical report, Delft University of Technology and Free University
  of Amsterdam, 1996.

\bibitem[Liu19]{liu2019multi}
Z.~Liu.
\newblock Multi-time distribution of tasep.
\newblock {\em arXiv preprint}, 2019.
\newblock arXiv:1907.09876 [math.PR].

\bibitem[Mac95]{Macdonald1995}
I.G. Macdonald.
\newblock {\em Symmetric functions and {H}all polynomials}.
\newblock Oxford University Press, 2nd edition, 1995.

\bibitem[Man14]{Mangazeev2014}
V.~Mangazeev.
\newblock {On the Yang--Baxter equation for the six-vertex model}.
\newblock {\em Nuclear Physics B}, 882:70--96, 2014.
\newblock arXiv:1401.6494 [math-ph].

\bibitem[Mui23]{Muir_det}
T.~Muir.
\newblock {\em The theory of determinants in the historical order of
  development, 4 vols.}
\newblock Macmillan, London, 1906-1923.

\bibitem[Oko97]{okounkov1997binomial}
A.~Okounkov.
\newblock Binomial formula for macdonald polynomials and applications.
\newblock {\em Math. Res. Lett.}, 4(4):533--553, 1997.
\newblock arXiv:q-alg/9608021.

\bibitem[Oko98]{okounkov_newton_int}
A.~Okounkov.
\newblock On newton interpolation of symmetric functions: A characterization of
  interpolation macdonald polynomials.
\newblock {\em Adv. Appl. Math.}, 20:395--428, 1998.

\bibitem[Ols19]{olshanski2019interpolation}
G.~Olshanski.
\newblock Interpolation macdonald polynomials and cauchy-type identities.
\newblock {\em Jour. Comb. Th. A}, 162:65--117, 2019.
\newblock arXiv:1712.08018 [math.CO].

\bibitem[Pov13]{Povolotsky2013}
A.~Povolotsky.
\newblock {On integrability of zero-range chipping models with factorized
  steady state}.
\newblock {\em J. Phys. A}, 46:465205, 2013.
\newblock arXiv:1308.3250 [math-ph].

\bibitem[Sah94]{Sah1994spectrum}
S.~Sahi.
\newblock The spectrum of certain invariant differential operators associated
  to hermitian symmetric spaces.
\newblock In {Brylinski, J.-L., et al.}, editor, {\em Lie theory and geometry},
  volume 123 of {\em Progress Math}, pages 569--576, Boston, 1994. Birkhauser.

\bibitem[Sah96]{sahi1996interpolation}
S.~Sahi.
\newblock Interpolation, integrality, and a generalization of macdonald's
  polynomials.
\newblock {\em Intern. Math. Res. Notices}, 1996(10):457--471, 1996.

\bibitem[SV14]{sergeev2014jacobi}
A.N. Sergeev and A.P. Veselov.
\newblock Jacobi--trudy formula for generalized schur polynomials.
\newblock {\em Moscow Mathematical Journal}, 14(1):161--168, 2014.
\newblock arXiv:0905.2557 [math.RT].

\bibitem[TW08]{TW_ASEP1}
C.~Tracy and H.~Widom.
\newblock {Integral formulas for the asymmetric simple exclusion process}.
\newblock {\em Comm. Math. Phys}, 279:815--844, 2008.
\newblock arXiv:0704.2633 [math.PR]. Erratum: Commun. Math. Phys.,
  304:875--878, 2011.

\bibitem[TW09]{TW_ASEP2}
C.~Tracy and H.~Widom.
\newblock {Asymptotics in ASEP with step initial condition}.
\newblock {\em Commun. Math. Phys.}, 290:129--154, 2009.
\newblock arXiv:0807.1713 [math.PR].

\bibitem[War08]{warnaar2008bisymmetric}
S.O. Warnaar.
\newblock Bisymmetric functions, macdonald polynomials and basic hypergeometric
  series.
\newblock {\em Compos. Math.}, 144(2):271--303, 2008.
\newblock arXiv:math/0511333 [math.CO].

\bibitem[WZJ16]{wheeler2015refined}
M.~Wheeler and P.~Zinn-Justin.
\newblock {Refined Cauchy/Littlewood identities and six-vertex model partition
  functions: III. Deformed bosons}.
\newblock {\em Adv. Math}, 299:543--600, 2016.
\newblock arXiv:1508.02236 [math-ph].

\bibitem[Yel20]{yeliussizov2020positive}
D.~Yeliussizov.
\newblock Positive specializations of symmetric grothendieck polynomials.
\newblock {\em Advances in Mathematics}, 363:107000, 2020.
\newblock arXiv:1907.06985 [math.CO].

\end{thebibliography}
\end{document}